\documentclass[a4paper]{amsart}
\usepackage{amssymb,stmaryrd}
\usepackage{amsfonts}
\usepackage{amstext}
\usepackage{algorithmic}
\usepackage{algorithm}
\usepackage{color}
\usepackage{graphicx}
\usepackage{epstopdf}
\usepackage[all]{xy}
\usepackage{MnSymbol}
\usepackage{tikz-cd}
\tikzcdset{arrow style=tikz, diagrams={>=stealth}}
\numberwithin{equation}{section}
\newtheorem{theorem}{Theorem}[section]
\newtheorem{lemma}[theorem]{Lemma}
\newtheorem{proposition}[theorem]{Proposition}
\newtheorem{corollary}[theorem]{Corollary}

\theoremstyle{definition}
\newtheorem{definition}[theorem]{Definition}
\newtheorem{example}[theorem]{Example}

\theoremstyle{remark}
\newtheorem{remark}[theorem]{\bf{Remark}}

\newcommand{\Hom}{{\rm{Hom}}}

\newcommand{\C}{{\mathbb{C}}}
\newcommand{\Z}{{\mathbb{Z}}}

\newcommand{\N}{{\mathbb{N}}}
\newcommand{\A}{{\mathbb{A}}}

\newcommand{\<}{{\langle}}

\renewcommand{\>}{{\rangle}}

\newcommand{\BB}{\overline{B}}

\newcommand{\BG}{{B^{\Gamma}}}
\newcommand{\CC}{{\mathcal{C}}}
\newcommand{\CH}{{\mathcal{H}}}
\newcommand{\CD}{{\mathcal{D}}}

\newcommand{\cD}{{\mathcal{D}}}

\newcommand{\CG}{{\cdot_{\Gamma}}}
\newcommand{\CJ}{{\mathcal{J}}}
\newcommand{\hJ}{{\hat{\mathcal{J}}}}
\newcommand{\HJ}{{\mathbb{J}}}
\newcommand{\BJ}{{\mathbb{J}}}

\newcommand{\cI}{{\mathcal{I}}}

\newcommand{\CL}{{\mathcal{L}}}
\newcommand{\cL}{{\mathcal{L}}}
\newcommand{\BL}{{\mathbb{L}}}
\newcommand{\CM}{{\mathcal{M}}}
\newcommand{\CN}{{\mathcal{N}}}

\newcommand{\wedgeq}{{\wedge\kern-5pt\cdot}}

\newcommand{\DG}{{\Delta^{\Gamma}}}
\newcommand{\GH}{{\Gamma^{\#}}}
\renewcommand{\ker}{{\rm{ker}}}

\newcommand{\di}{{\diamond}}
\newcommand{\p}{\hat{+}}

\newcommand{\m}{\hat{-}}

\newcommand{\tens}{\otimes}
\newcommand{\ot}{\otimes}

\newcommand{\one}[1]{{#1}{}_{\scriptscriptstyle{(1)}}}
\newcommand{\two}[1]{{#1}{}_{\scriptscriptstyle{(2)}}}

\newcommand{\id}{{\rm id}}

\newcommand{\z}{{}_{\scriptscriptstyle{(0)}}}
\renewcommand{\o}
{{}_{\scriptscriptstyle{(1)}}}
\newcommand{\mo}{{}_{\scriptscriptstyle{(-1)}}}
\renewcommand{\t}{{}_{\scriptscriptstyle{(2)}}}
\newcommand{\mt}{{}_{\scriptscriptstyle{(-2)}}}
\renewcommand{\th}{{}_{\scriptscriptstyle{(3)}}}

\newcommand{\fo}{{}_{\scriptscriptstyle{(4)}}}

\newcommand{\fiv}{{}_{\scriptscriptstyle{(5)}}}
\newcommand{\si}{{}_{\scriptscriptstyle{(6)}}}

\newcommand{\rz}{{}_{\scriptscriptstyle{[0]}}}
\newcommand{\rmo}
{{}_{\scriptscriptstyle{[-1]}}}
\newcommand{\ro}
{{}_{\scriptscriptstyle{[1]}}}
\newcommand{\rt}
{{}_{\scriptscriptstyle{[2]}}}
\newcommand{\rth}{{}_{\scriptscriptstyle{[3]}}}
\newcommand{\extd}{{\rm d}}

\renewcommand{\imath}{\mathrm{i}}

\newcommand{\la}{{\triangleright}}

\allowdisplaybreaks
\begin{document}

\title{Quantum Jet Hopf algebroids by cotwist}


\author{Xiao Han and Shahn Majid}
\address{Queen Mary University of London,
School of Mathematical Sciences, Mile End Rd, London E1 4NS, UK}

\email{x.h.han@qmul.ac.uk,    s.majid@qmul.ac.uk}


\begin{abstract} We introduce a cotwist construction for  Hopf algebroids that also entails cotwisting or `quantisation' of the base and which is dual to a previous twisting construction of P. Xu. Whereas the latter applied the construction to the algebra of differential operators on a classical base $B$, we show that the dual of this is the algebra of sections $\CJ(B)$ of the jet bundle and hence that the latter forms a Hopf algebroid. This is constructed for commutative algebras $B$ in a pro-object setting via quotients of the pair Hopf algebroid $B\tens B$ and can then be deformed by our cotwist construction to give a possibly noncommutative jet Hopf algebroid over a noncommutative base. We also observe in the commutative case that $\CJ^k(B)$ for jets of order $k$ can be identified with $\CJ^1(B_k)$ where $B_k$ denotes $B$ equipped with a certain noncommutative first order differential calculus. \end{abstract}
\maketitle

\begin{quote}  2010 Mathematics Subject Classification: 16T05, 20G42, 46L87, 58B32\end{quote}

\section{Introduction}\label{secintro}

A long-standing open problem in noncommutative geometry is the construction of jet bundles over a unital potentially noncommutative algebra $B$ in the role of `coordinate algebra'. Recently, there was some progress for $B$ equipped with a differential graded algebra $(\Omega_B,\extd)$ of `differential forms', see \cite{MaSim} as well as a subsequent approach via a jet endofunctor\cite{Flo} of the category of $B$-modules. Of most interest is the split case where there is a jet prolongation map $j_\infty: B\to \CJ^{\infty}(B)$ that sends a coordinate `function' to a section of the jet bundle over $B$, which required in \cite{MaSim} the additional data of a flat connection with certain properties. One also has some partial results for the jet bundle sections $\CJ^\infty(E)$ associated to the sections $E$ of vector bundle in the form of a (finitely generated projective) $B$-bimodule. In the present work, we introduce a third approach using Hopf algebroids and quantisation via cocycle cotwists of a certain kind. We will see that the jet prolongation map then appears naturally as the source map (and there is a similar target map) of the Hopf algebroid. Bialgebroids and Hopf algebroids have recently been of interest in their own right, see \cite{Boehm} and many recent works.

Our starting point is that if $M$ is a smooth manifold, it is known that the algebra of differential operators $\CD(M)$ is a noncommutative cocommutative bialgebroid, see \cite{Xu}. Hence one could expect that there is some kind of dual bialgebroid or Hopf algebroid over the same base $B=C^\infty(M)$, which is essentially the jet sections algebra $\CJ^\infty(M)$ and which in our algebraic form will indeed  become a Hopf algebroid. To explain this, recall \cite[Prop 1.9]{Ver} that if $E,F$ are sections of vector bundles over $M$, differential operators $E\to F$ of degree $k$ can be identified with bundle maps
\[ \CD_k(E,F)= \Hom_{C^\infty(M)}(\CJ^k(E),F),\]
where $\CJ^k(E)$ denotes sections of the relevant $k$-th jet bundle. Notice that if $E=F=C^\infty(M)$ as sections of a trivial bundle then we get
\[ \CD_k(M)=\Hom_{C^\infty(M)}(\CJ^k(M),C^\infty(M))\]
 for the underlying jet bundles on $M$ itself. This says that in each degree, differential operators are the sections of the dual bundle to the jet bundle. Then $\CD(M):=\CD_\infty(M)$ suggests that some version of $\CJ^\infty(M)$ is a commutative noncocomutative bialgebroid or Hopf algebroid dually paired to this. We show this in Section~\ref{secpair} in an algebraic formulation $\CD(B)$ for differential operators and $\CJ^\infty(B)$ (more simply denoted by $\CJ(B)$ in Section~\ref{secpair}) for jets, following ideas from \cite{Ver}. Here $\CJ^k(B)$ turns out to be a quotient of the pair Hopf algebroid $B\tens B$. This `classical' level of the theory where $B$ is a commutative algebra is also reminiscent of Connes' tangent groupoid\cite{Car} as a `thickenning' of the tangent groupoid. We note that the  approach of \cite{MaSim} also equipped $\CJ^\infty(B)$, prior to completion, with the structure of a braided-Hopf algebra in the category  ${}_B\CM_B$ of $B$-modules with respect to a geometric braiding as part of the additional data. In the case of commutative $B$, this  can be viewed as a Hopf algebroid over $B$ related in the geometric case to functions on the tangent bundle $TM$ that are polynomial on the fibre. Another observation  in the case of $B$ commutative, see Remark~\ref{remJ1}, is  that one can think of $\CJ^k(B)$ as $\CJ^1(B_k)$ where $B_k$ is the same algebra  but equipped with a new differential structure $\Omega(B_k)$ which is noncommutative for $k>1$.

Turning now to the main results of the paper, the idea is that Ping Xu in \cite{Xu} showed how to twist $\CD(M)$ to a noncommutative noncocommutative Hopf algebroid $\CD^F(M)$ while at the same time twisting $B=C^\infty(M)$ to  a noncommutative algebra $B^F$, with a respect to a `cocycle' $F$. With this in mind, our main result is to introduce a dual version, i.e. a cotwist by $\Gamma$
which turns a Hopf algebroid $\CL$ over base $B$ (which could be noncommutative) into a new one $\CL^\Gamma$ over a new base $B^\Gamma$, see Theorem~\ref{thm. twisted Hopf algebroid}.  This cotwist construction is more general than the previous `Drinfeld cotwist' of Hopf algebroids in \cite{HM22}, but includes the latter as shown in Theorem~\ref{thm. comparison}. The previous Drinfeld cotwist for Hopf algebroids did not change the base and hence would not serve our purposes of constructing examples for noncommutative geometry by cotwist. Note that even finding  suitable axioms for $\Gamma$ is not at all straightforward precisely because the base is also being changed in the construction. We show that when formulated and related correctly, our construction is indeed dually paired with Xu's twist of the dual Hopf algebroid, see Theorem~\ref{thm. twisted dual pairing}. We similarly show in  Lemma~\ref{lem. comodule equ} that the category of $\CL$-comodules over $B$ is monoidally equivalent to that of $\CL^\Gamma$-comodules over $B^\Gamma$.

Our cotwist construction and the above remarks about duality then imply that we can start with a classical jet Hopf algebroid $\CJ^\infty(B)$ over a commutative base such as $B=C^\infty(M)$ and cotwist to obtain noncommutative noncocommutative jet Hopf algebroids $\CJ^\infty(B)^\Gamma$ over a noncommutative base $B^\Gamma$, see Theorem~\ref{thm:quant}. Moreover, as \cite{Xu} also showed how to provide examples of cocycles in his required form (they are closely linked with Kontsevich quantization), dualizing these in principle provides the required $\Gamma$ so that one can expect an abundance of examples.

As well as these main results, we study the cotwist construction further and show in Section~\ref{secgroupoid} that the collection of Hopf algebrioids equipped with invertible cocycles is itself a groupoid. A follow-up work \cite{HS25} applies this cotwisting theory to the Ehresmann-Schauenburg Hopf algebroid of a quantum principle bundle or quantum homogeneous space.  We also note that there are other approaches to Hopf algebroids of differential operators, notably by Ghobadi\cite{Gho}. Dualising the latter will provide another class of quantum jet bundles in some generality, to be looked at elsewhere. Noncommutative jet bundles  are also of interest for mathematical physics in order to define variational calculus and hence classical and quantum field theory on a noncommutative spacetime. For this, one will need a $*$ involution, for which a formulation of $*$-Hopf algebroids has recently been proposed in \cite{BHM24}. Such  $*$-structures on our jet Hopf algebroids is another interesting direction for further work.

\section{Preliminaries}\label{secpre}

In this section, we will recall some basic definitions and notation.

\subsection{Balanced tensor products} Let $B$ be a unital algebra over a field $k$. We denote the opposite algebra by $\BB$ and let $B\to \BB$, $b\mapsto\Bar{b}$ for any $b\in B$ be the obvious $k$-algebra antiisomorphism. Define $B^{e}:=B\ot \BB$, so $B$ and $\BB$ are commuting subalgebras of $B^{e}$. Let $M, N$ be $B^{e}$-bimodules. Following \cite{schau1}, it is  useful to define
\begin{align*}
    M\diamond_{B} N:=\int_{b} {}_{\Bar{b}}M\ot {}_{b}N:=&M\ot N/\langle \Bar{b}m\ot n-m\ot bn|b\in B, m\in M, n\in N\rangle,\\
    M\ot_{B} N:=\int_{b} M_{b}\ot {}_{b}N:=&M\ot N/\langle mb\ot n-m\ot bn|b\in B, m\in M, n\in N\rangle,\\
    M\ot_{\BB} N:=\int_{b} M_{\Bar{b}}\ot {}_{\Bar{b}}N:=&M\ot N/\langle m\Bar{b}\ot n-m\ot \Bar{b}n|b\in B, m\in M, n\in N\rangle.
\end{align*}
For convenience, we also define $N\ot^{B}M=\int_{b} {}_{b}N\ot M_{b}$ and  $N\ot^{\BB}M=\int_{b} {}_{\Bar{b}}N\ot M_{\Bar{b}}$. Moreover, we let
\begin{align*}
    \int^{b}M_{\Bar{b}}\ot N_{b}:=\{\sum_{i}m_{i}\ot n_{i}\in M\ot N\quad |\quad m_{i}\Bar{b}\ot n_{i}=m_{i}\ot n_{i}b, \forall b\in B\}.
\end{align*}
 For any $\BB$-bimodule $M$ and any $B$-bimodule $N$, we also define
\begin{align*}
    M\times_{B}N:=\int^{a}\int_{b} {_{\Bar{b}}M_{\Bar{a}}}\ot {_bN_a}
\end{align*}
as the {\em Takeuchi product} of $M$ and $N$. If $P$ is a $B^{e}$-bimodule, then $P\times_{B}N$ is a $B$-bimodule with $B$ acting on $P$. Similarly, $M\times_{B}P$ is a $\BB$-bimodule with $\BB$ acting on $P$. If both $M$ and $N$ are $B^{e}$-bimodules, then $M\times_{B} N$ is also a $B^{e}$-bimodule. However, the product $\times_{B}$ is neither associative nor unital in the category of $B^{e}$-bimodules. For any $M,N, P\in {}_{B^{e}}\CM_{B^{e}}$, we let
\begin{align*}
    M\times_{B}P\times_{B}N:=\int^{a,b}\int_{c,d} {}_{\Bar{c}}M_{\Bar{a}}\ot {}_{c,\Bar{d}}P_{a, \Bar{b}}\ot {}_{d}N_{b},
\end{align*}
where $\int^{a,b}:=\int^{a}\int^{b}$ and $\int_{c,d}:=\int_{c}\int_{d}$. There are maps
\begin{align*}
    &\alpha:(M\times_{B}P)\times_{B} N\to M\times_{B}P\times_{B}N,\quad m\ot p\ot n\mapsto m\ot p\ot n,\\
    &\alpha':M\times_{B}(P\times N)\to M\times_{B}P\times_{B}N,\quad m\ot p\ot n\mapsto m\ot p\ot n, 
\end{align*}
but neither $\alpha$ nor $\alpha'$ are isomorphisms  in general. For the rest of the paper, we will assume that all $B$-module and $\BB$-module structures are faithfully flat. In particular, this implies that $\alpha$ and $\alpha'$ are in fact isomorphisms.

\subsection{Left bialgebroids}\label{secB}\cite{Boehm,BW,schau1}  Let $B$ be a unital algebra over $k$.
A {\em $B$-ring}  means a unital algebra in the monoidal category ${}_B\CM_B$ of $B$-bimodules. Likewise,  a  {\em $B$-coring} is a coalgebra in ${}_B\CM_B$. Morphisms of $B$-(co)rings are morphisms in the category ${}_B\CM_B$ that preserve the (co)ring structures. Let $\mathcal{L}$ be an algebra and $s:B\to \mathcal{L}$,  $t :B^{op}\to \mathcal{L}$ algebra maps such that
$s(a)$ commutes with $t(b)$ for all $a,b\in B$. 
Then $\mathcal{L}$ is a left $B^e$-module by
$(a\tens b).X=s(a)\, t(b)\, X$ for all $X\in \mathcal{L}$, or equivalently $\mathcal{L}$ is a $B$-bimodule by
\begin{align}\label{eq:rbgd.bimod}
    a.X.b:=s(a)\, t(b)\, X  = a\,\overline{b}\,X,
\end{align}
where we adopt the convention that $a\in B$ can be viewed in $\CL$ via $s$ and $\bar b\in \overline{B}$ can be viewed in $\CL$ via $t$, and we take the product there.

\begin{definition}\label{def:left.bgd} Let $B$ be a unital algebra. A left bialgebroid over $B$ (or left $B$-bialgebroid) is an algebra $\cL$ with (`source' and `target') algebra maps  $s:B\to \cL$ and $t:B^{op}\to \cL$ whose images commute (making $\CL$ a $B^e$-ring) and a $B$-coring for the bimodule structure (\ref{eq:rbgd.bimod}) which is compatible in the sense
\begin{itemize}
\item[(i)] The coproduct $\Delta$ corestricts to an algebra map  $\cL\to \cL\times_{B} \cL$, where the Takeuchi product
\begin{equation*} \cL\times_{B} \cL :=\{\ \sum_i X_i \ot_{B} Y_i\ |\ \sum_i X_i\,\overline{b} \ot_{B} Y_i=
\sum_i X_i \ot_{B}  Y_i  \,b,\quad \forall b\in B\ \}\subseteq \cL\tens_{B}\cL
\end{equation*}
 is regarded as algebra via factorwise multiplication.
\item[(ii)] The counit $\varepsilon$ is a left character in the sense
\begin{equation*}\varepsilon(1_{\cL})=1_{B},\quad \varepsilon(X\,\varepsilon(Y))=\varepsilon(XY)=\varepsilon(X\overline{\varepsilon(Y)})\end{equation*}
for all $X,Y\in \cL$. 
\end{itemize}
\end{definition}
We will often use  {\em Sweedler notation}  $\Delta X= X\o\tens X\t$, where the numerical subscripts indicate a sum of such terms (as commonly used for Hopf algebras\cite{Swe,Ma:book}).

\begin{definition}\label{def. left Hopf and anti-Hopf algebroids}

A left $B$-bialgebroid $\cL$ is a left Hopf algebroid (\cite{schau1}, Thm and Def 3.5.) if
\[\lambda: \cL\ot_{\BB}\cL\to \cL\diamond_{B}\cL,\quad
    \lambda(X\ot Y)=X\o \ot X\t Y\]
is invertible. A left $B$-bialgebroid $\cL$ is an anti-left Hopf algebroid if
\[\mu: \cL\ot_{B}\cL\to \cL\diamond_{B}\cL,\quad
    \mu(X\ot Y)=X\o Y\ot_{B}X\t\]
is invertible. $\cL$ is called a Hopf algebroid if it is both a left and anti-left Hopf algebroid.
\end{definition}

We adopt the following shorthand for the `(anti-)translation maps' 
\begin{equation}\label{X+-} X_{+}\ot_{\BB}X_{-}:=\alpha(X):=\lambda^{-1}(X\diamond_{B}1),\end{equation}
\begin{equation}\label{X[+][-]} X_{[+]}\ot_{B}X_{[-]}:=\beta(X):=\mu^{-1}(1\diamond_{B}X).\end{equation}

We recall from \cite[Prop.~3.7]{schau1} that for a left Hopf algebroid, and all $X, Y\in \cL$ and $a,a',b,b'\in B$,
\begin{align}
    \one{X_{+}}\diamond_{B}{}\two{X_{+}}X_{-}&=X\diamond_{B}{}1\label{equ. inverse lamda 1},\\
    \one{X}{}_{+}\ot_{\BB}\one{X}{}_{-}\two{X}&=X\ot_{\BB}1\label{equ. inverse lamda 2},\\
    (XY)_{+}\ot_{\BB}(XY)_{-}&=X_{+}Y_{+}\ot_{\BB}Y_{-}X_{-}\label{equ. inverse lamda 3},\\
    1_{+}\ot_{\BB}1_{-}&=1\ot_{\BB}1\label{equ. inverse lamda 4},\\
    \one{X_{+}}\diamond_{B}{}\two{X_{+}}\ot_{\BB}X_{-}&=\one{X}\diamond_{B}{}\two{X}{}_{+}\ot_{\BB}\two{X}{}_{-}\label{equ. inverse lamda 5},\\
    X_{+}\ot\one{X_{-}}\ot{}\two{X_{-}}&=X_{++}\ot X_{-}\ot{}X_{+-}\in\int^{a,b}\int_{c,d}{}_{\Bar{a}}\cL_{\Bar{c}}\ot {}_{\Bar{d}}\cL_{\Bar{b}}\ot {}_{\Bar{c},d}\cL_{b,\Bar{a}}
    \label{equ. inverse lamda 6},\\
    X&=X_{+}\overline{\varepsilon(X_{-})}\label{equ. inverse lamda 7},\\
    X_{+}X_{-}&=\varepsilon(X)\label{equ. inverse lamda 8},\\
    aX_{+}b\ot_{\BB}b'X_{-}a'&=(a\Bar{a'}Xb\Bar{b'})_{+}\ot_{\BB}(a\Bar{a'}Xb\Bar{b'})_{-}\label{equ. inverse lamda 9},\\
    \Bar{b}X_{+}\ot_{\BB}X_{-}&=X_{+}\ot_{\BB}X_{-}\Bar{b}\label{equ. inverse lamda 10}.
\end{align}

\begin{definition}
   Given a left Hopf algebroid $\cL$ over $B$, a (left) Hopf ideal $\cI$ of $\cL$ is  a  left $B^{e}$-submodule of $\cL$ such that
\begin{itemize}
    \item [(1)] $\cI$ is an ideal of $\cL$;
    \item[(2)] $\cI$ is a coideal of $\cL$, i.e., $\Delta(I)\subseteq I\diamond_{B}\cL+\cL\diamond_{B}\cI$;
    \item[(3)] For all $i\in\cI$,
    \[i_{+}\ot_{\overline{B}}i_{-}\in \cI\ot_{\overline{B}}\cL+\cL\ot_{\overline{B}}\cI. \]
\end{itemize}
\end{definition}
\begin{proposition}\cite{Gho}\label{prop. quotient Hopf algebroids}
    If $\cL$ is a left Hopf algebroid over $B$ and $\cI$ is a Hopf ideal of $\cL$ then $\cL/\cI$ is a left Hopf algebroid over $B$.
\end{proposition}

 \subsection{Modules and comodules}\label{secmod}

 (1) Let $\CL$ be a left bialgebroid over $B$. A left module of $\CL$ is a left module $M$ of $\CL$ as an algebra, and becomes a $B$-bimodule by
\[ a. m.b=s(a)\, t(b)\la m,\quad\forall m\in M.\]
We make the category ${}_{\mathcal{L}}\mathcal{M} $ of left $\CL$-modules into a monoidal category by using
$\tens_B$ with respect to this $B$-bimodule structure, and the action of $\CL$ via by the coproduct, i.e.,
\[
X\la (m\tens n) = (X\o\la m)\tens (X\t\la n), \quad\forall m\in M\in {}_{\mathcal{L}}\mathcal{M},\  n\in N\in {}_{\mathcal{L}}\mathcal{M}.
\]

(2) A left $\cL$-comodule of a left $B$-bialgebroid $\CL$ is defined\cite{schau1} as a  $B$-bimodule $M$, together with a $B$-bimodule map $\delta: M\to \cL\times_{B} M\subseteq \cL\diamond_{B}M$ such that
    \[(\id\diamond_{B}\delta)\circ \delta=(\Delta\diamond_{B}\id)\circ \delta,\qquad (\varepsilon\diamond_{B}\id)\circ \delta=\id.\]
Using the notation $\delta(m)=m\mo \ot m\z$, that $\delta$ is a  $B$-bimodule map here amounts to $\delta(b\, m\, b')=b\, m\mo\, b'\ot m\z$.  The category ${}^{\mathcal{L}}\mathcal{M}$ of left $\CL$-comodules  is monoidal with tensor product of $M,N\in{}^{\mathcal{L}}\mathcal{M}$ given by $M\ot_{B}N$, with comodule structure 
     \[\delta(m\ot\,n)=m\mo\,n\mo\ot (m\z\ot\,n\z)\]
      for all $m\in M,n\in N$.

\section{Cotwist construction of Hopf algebroids}\label{seccotwist}

In this section, we provide a new cotwist construction that is more general than in \cite{Boehm,HM22,HM23} as it also modifies the base. We will eventually show that it is dual to a twist construction in \cite{Xu} when that applies, for a certain duality pairing.

\subsection{Invertible left 2-cocycles on bialgebroids}\label{secBcotwist}

\begin{definition} \label{Lcotwist}
Let $\cL$ be a left $B$-bialgebroid. A \textup{left 2-cocycle} on
$\cL$ is an element $\Gamma\in {}\Hom_{\overline{B}-}(\cL\otimes_{\overline{B}}\cL, B)$, such that
\begin{itemize}
    \item [(1)]$\Gamma(X, \Gamma(Y\o, Z\o)Y\t Z\t)=\Gamma(\Gamma(X\o, Y\o)X\t Y\t, Z);$
    \item[(2)] $\Gamma(1_{\cL}, X)=\varepsilon(X)=\Gamma(X, 1_{\cL}),$
\end{itemize}
for all $X, Y, Z\in \cL$. We denote by $Z^{2}(\cL, B)$ the collection of such 2-cocycles on $\cL$ over $B$. A right 2-cocycle is an element $\Sigma\in \Hom_{B-}(\cL\otimes_{B}\cL, B)$, such that
\begin{itemize}
    \item [(1)]$\Sigma(X, \overline{\Sigma(Y\t, Z\t)}Y\o Z\o)=\Sigma(\overline{\Sigma(X\t, Y\t)}X\o Y\o, Z);$
    \item[(2)] $\Sigma(1_{\cL}, X)=\varepsilon(X)=\Sigma(X, 1_{\cL}),$
\end{itemize}
for all $X, Y, Z\in \cL$. We call $\hat{\varepsilon}:=\varepsilon\circ m_{\cL}:\cL\ot_{\overline{B}}\cL\to B, X\ot Y\mapsto \varepsilon(XY)$ the trivial  left 2-cocycle on $\cL$.
\end{definition}

We will also need the pull-back along $s$ of a left 2-cocycle $\Gamma$ on $\CL$, which it will be convenient to  continue to denote as $\Gamma(a,b)$ for all $a,b\in B$ viewed via $s$ as elements of $\CL$. 

\begin{corollary} Given a  left 2-cocycle $\Gamma$ on a  left $B$-bialgebroid $\cL$, $a\cdot_{\Gamma}b=\Gamma(a,b)$ for all $a,b\in B$  defines an associative product on the same underlying vector space as $B$ and with the same unit. We denote $B$ with this product by $B^\Gamma$.   
\end{corollary}
\begin{proof}
    The cocycle axioms imply that this is associative, 
    \begin{align*}
        (a\cdot_{\Gamma} b)\cdot_{\Gamma} c=\Gamma(\Gamma(a,b),c)=\Gamma(a,\Gamma(b,c))=a\cdot_{\Gamma}(b\cdot_{\Gamma}c),
    \end{align*}
    for all $a,b,c\in B$, and that $1_B$ remains the unit. 
\end{proof}

\begin{proposition}\label{prop. twisted bimodule structures}
 If $M\in {}^{\cL}\mathcal{M}$ and $\Gamma$ is a left 2-cocycle on $\cL$ then $M$ is a $B^{\Gamma}$-bimodule with the bimodule structure given by
\begin{align}
a\cdot_{\Gamma}m=\Gamma(a,m\mo)m\z,\qquad m\cdot_{\Gamma}a=\Gamma(m\mo,a)m\z,
\end{align}
for all $a\in B$ and $m\in M$.
\end{proposition}
\begin{proof}
    For all $a,b\in B$ and $m\in M$,
    \begin{align*}
        (a\cdot_{\Gamma} b)\cdot_{\Gamma}m=\Gamma(\Gamma(a,b),m\mo)m\z=\Gamma(a, \Gamma(b, m\mt)m\mo)m\z=a\cdot_{\Gamma}(b\cdot_{\Gamma}m).
    \end{align*}
    As a result, $M$ is a left $B^{\Gamma}$-module. Similarly, $M$ is a right $B^{\Gamma}$-module. Moreover,
    \begin{align*}
        (a\cdot_{\Gamma} m)\cdot_{\Gamma}b=\Gamma(\Gamma(a,m\mt)m\mo,b)m\z=\Gamma(a, \Gamma(m\mt,b)m\mo)m\z= a\cdot_{\Gamma} (m\cdot_{\Gamma}b).
    \end{align*}
\end{proof}
\begin{remark}
     Given a left 2-cocycle $\Gamma$ on a left $B$-bialgebroid $\cL$, one can check that
     \[\Gamma(X\cdot_{\Gamma}a, Y)=\Gamma(X,a\cdot_{\Gamma}Y)\]
   for all $a\in B$ and $X,Y\in \CL$. Hence $\Gamma$ descends to a well-defined map $\CL\tens_{B^\Gamma}\CL\to B$.
\end{remark}
Recall that for any $N,M\in {}^{\cL}\CM$, $N\ot_{B} M$ is a left $\cL$-comodule with the codiagonal coaction
\[\delta(m\ot_{B}n)=m\mo\,n\mo\diamond_{B}m\z\ot_{B}n\z,\]
for all $m\in M$ and $n\in N$.

\begin{lemma}\label{lem. coherent map is BG-bilinear}
    Let $M,N\in {}^{\cL}\CM$. The map $\Gamma^{\#}:M\ot_{B^{\Gamma}}N\to M\ot_{B}N$ given by
    \begin{align*}
        \Gamma^{\#}(m\ot_{B^{\Gamma}}n)=\Gamma(m\mo, n\mo)m\z\ot_{B}n\z,
    \end{align*}
    is well defined. Moreover,
    \begin{align*}
        \Gamma^{\#}(b\cdot_{\Gamma}m\ot_{B^{\Gamma}}n)=b\cdot_{\Gamma}\Gamma^{\#}(m\ot_{B^{\Gamma}}n),\qquad \Gamma^{\#}(m\ot_{B^{\Gamma}}n\CG b)=\Gamma^{\#}(m\ot_{B^{\Gamma}}n)\CG b
    \end{align*}
    for all $b\in B, m\in M, n\in N$.
\end{lemma}
\begin{proof}
    To see  that $\Gamma^{\#}$ is well defined, we check $\Gamma^{\#}$ factors through $\ot_{B^{\Gamma}}$. On the one hand,
    \begin{align*}
        \Gamma^{\#}(m\cdot_{\Gamma}b\ot n)=\Gamma^{\#}(\Gamma(m\mo,b)\ot n)=\Gamma(\Gamma(m\mt,b)m\mo,n\mo)m\z\ot n\z, 
    \end{align*}
  while on the other hand,
    \begin{align*}
        \Gamma^{\#}(m\ot b\cdot_{\Gamma} n)=\Gamma^{\#}(m\ot \Gamma(b,n\mo)n\z)=\Gamma(m\mo, \Gamma(b, n\mt)n\mo)m\z\ot n\z.
    \end{align*}
The two are the same by the 2-cocycle condition. By the same method, one can show  $\Gamma^{\#}(b\cdot_{\Gamma}m\ot_{B^{\Gamma}}n)=b\cdot_{\Gamma}\Gamma^{\#}(m\ot_{B^{\Gamma}}n)$ and $\Gamma^{\#}(m\ot_{B^{\Gamma}}n\CG b)=\Gamma^{\#}(m\ot_{B^{\Gamma}}n)\CG b$.
\end{proof}
Sometimes, we will denote the above map by $\GH_{M,N}$ in order to mention explicitly the corresponding modules $M$ and $N$. We say that $\Gamma$ is \textit{invertible}, if $\Gamma^{\#}$ is invertible for any left $\cL$-comodule $M,N$.
\begin{corollary}
    If $\Gamma$ is an invertible left 2-cocycle on a left $B$-bialgebroid then
    \begin{align*}
        \Gamma^{\#-1}(b\cdot_{\Gamma}(m\ot_{B}n ))=b\cdot_{\Gamma} \Gamma^{\#-1}(m\ot_{B}n),
    \end{align*}
    for all $b\in B$, $m\in M\in{}^{\cL}\CM$ and $n\in N\in{}^{\cL}\CM$.
\end{corollary}
Using the 2-cocycle condition for $\Gamma$, it is straightforward to show:
\begin{corollary}
    Let $N, M, P\in {}^{\cL}\CM$ and  let $\Gamma$ be a left 2-cocycle on $\cL$. Then
 \[
        \GH_{N\ot_{B}M,P}\circ(\GH_{N,M}\ot_{\BG}\id_{P})=\GH_{N, M\ot_{B}P}\circ(\id_{N}\ot_{\BG}\GH_{M,P})\]
       as maps $N\ot_{\BG}M\ot_{\BG}P\to N\ot_{B}M\ot_{B}P$.
\end{corollary}

\subsection{Cotwist of Hopf algebroids}\label{secHcotwist}

\begin{lemma}
      If $\Gamma$ is a left 2-cocycle on a left $B$-Hopf algebroid $\cL$ then the underlying vector space equipped with the product
\begin{align}\label{twistprod}
    X\cdot_{\Gamma} Y:=\Gamma(X\o, Y\o)X\t{}_{+} Y\t{}_{+}\overline{\Gamma(Y\t{}_{-}, X\t{}_{-})},
\end{align}
is a $(B^{\Gamma})^e$-ring.
\end{lemma}
\begin{proof}
    The $B^{\Gamma}$-bimodule structure is given by $b\cdot_\Gamma X \cdot_\Gamma b'$,
    for all $b.b'\in B^\Gamma$ and $X\in \cL$. The $\overline{B^{\Gamma}}$-bimodule structure is given by $\overline{b}\cdot_\Gamma X \cdot_\Gamma \overline{b'}$. The associativity follows straightforwardly from the 2-cocycle condition of $\Gamma$ and (\ref{equ. inverse lamda 6}). Moreover, $b\cdot_\Gamma \overline{b'}=\overline{b'}\cdot_\Gamma b$.\end{proof}
    
It is clear that $\cL$ is itself a left  $\cL$-comodule  by its coproduct. Moreover, there is another left $\cL$-comodule structure given by
 \begin{align}\label{regular left comodule structure}
     \delta:\cL\to \cL\di_{B} \cL,\quad X\mapsto X_{-}\ot X_{+}, \qquad\forall X\in \CL,
 \end{align}
where the $B$-bimodule structure is given by $b.X.b':=\overline{b'}X\overline{b}$. One can check that this defines a comodule structure by using (\ref{equ. inverse lamda 6}) and (\ref{equ. inverse lamda 7}). Thus, $\cL$ is equipped with two different left $\cL$-comodule structures. By applying Lemma~\ref{lem. coherent map is BG-bilinear} in the case that $M=\cL$ equipped with the left comodule structure given by (\ref{regular left comodule structure}), and $N=\cL$ equipped with the comodule structure given by its coproduct $\Delta$, we obtain a map $\Gamma^{\#}:\cL^{\Gamma}\di_{B^{\Gamma}} \cL^{\Gamma}\to \cL\di_{B} \cL$  given by \begin{align}\label{equ. special gamma}
    \Gamma^{\#}(X\di_{B^\Gamma} Y)=X_{+}\overline{\Gamma(X_{-}, Y\o)}\di_{B}Y\t,
\end{align}
which is invertible if the left 2-cocycle $\Gamma$ is invertible.
 \begin{lemma}
     If $\Gamma$ is an invertible left 2-cocycle on a left $B$-Hopf algebroid $\cL$ then $\cL^{\Gamma}$ is a $B^\Gamma$-coring, with $B^\Gamma$-bimodule structure
     \[b.X.b'=b\cdot_\Gamma \overline{b'}\cdot_\Gamma X,\qquad\forall X\in \cL,\quad \forall b,b'\in B^\Gamma\]
  and the  coproduct and counit
     \[\Delta^{\Gamma}(X)=\Gamma^{\#-1}(X\o\di_{B}X\t),\qquad \varepsilon^{\Gamma}(X)=\Gamma(X_{+}, X_{-}),\]
     where $\Gamma^{\#}:\cL^{\Gamma}\di_{B^{\Gamma}} \cL^{\Gamma}\to \cL\di_{B} \cL$ is given by (\ref{equ. special gamma}). We will use the Sweedler type notation $\Delta^\Gamma(X)=X\ro\di X\rt$.
 \end{lemma}
\begin{proof}
    We see that $\varepsilon^{\Gamma}$ is $B^{\Gamma}$-bilinear,
    \begin{align*}
        \varepsilon^{\Gamma}(b\cdot_{\Gamma}X)=&\Gamma(\Gamma(b, X\o)X\t{}_{+}, X\t{}_{-})=\Gamma(\Gamma(b, X{}_{+}\o)X{}_{+}\t, X_{-})\\
        =&\Gamma(b, \Gamma(X_{+}, X_{-}))=b\cdot_{\Gamma} \varepsilon^{\Gamma}(X).
    \end{align*}
    Similarly, $\varepsilon^{\Gamma}$ is right $B^{\Gamma}$-linear. It is not hard to see that the twisted coproduct is well defined. That  $\Delta^{\Gamma}$ is $B^{\Gamma}$-bilinear follows from
    \begin{align*}
       \Gamma^{\#}((b\cdot_{\Gamma} X)\di_{\BG} Y)= b\cdot_{\Gamma}X_{+}\overline{\Gamma(X_{-}, Y\o)}\di_{B}Y\t,
    \end{align*}
    and
    \begin{align*}
        \Gamma^{\#}(X\di_{\BG} (\overline{b}\cdot_{\Gamma} Y))=X_{+}\overline{\Gamma(X_{-}, Y\o)}\di_{B}\overline{b}\cdot_{\Gamma}Y\t.
    \end{align*}
We next check that the image of the coproduct belongs to Takeuchi product. Indeed, using our notation for $\Delta^\Gamma(X)$, we have on the one hand
 \begin{align*}
     \Gamma^{\#}(X\ro\CG \overline{b}\di_{\BG}X\rt)=&\Gamma^{\#}(X\ro{}_{+}\overline{\Gamma(b, X\ro{}_{-})}\di_{\BG}X\rt)\\
     =&X\ro{}_{++}\overline{\Gamma(\Gamma(b, X\ro{}_{-})X\ro{}_{+-}, X\rt\o)}\di_{B}X\rt\t\\
     =&X\ro{}_{+}\overline{\Gamma(\Gamma(b, X\ro{}_{-}\o)X\ro{}_{-}\t, X\rt\o)}\di_{B}X\rt\t\\
     =&X\ro{}_{+}\overline{\Gamma(b, \Gamma(X\ro{}_{-}\o, X\rt\o)X\ro{}_{-}\t X\rt\t)}\di_{B}X\rt\th\\
     =&X\ro{}_{++}\overline{\Gamma(b, \Gamma(X\ro{}_{-}, X\rt\o)X\ro{}_{+-} X\rt\t)}\di_{B}X\rt\th\\
     =&X\o{}_{+}\overline{\Gamma(b, X\o{}_{-} X\t)}\di_{B}X\th=X\o\overline{b}\di_{B}X\t,
 \end{align*}
 where the 6th step uses the fact that
 \begin{align}\label{equ. relation between old and new coproduct}
     X\o\di_{B}X\t=X\ro{}_{+}\overline{\Gamma(X\ro{}_{-}, X\rt\o)}\di_{\BG}X\rt\t.
 \end{align}
 On the other hand,
 \begin{align*}
      \Gamma^{\#}(X\ro\di_{\BG}X\rt\CG b)=&X\ro{}_{+}\overline{\Gamma(X\ro{}_{-}, \Gamma(X\rt\o,b)X\rt\t)}\di_{\BG}X\rt\th\\
      =&X\ro{}_{++}\overline{\Gamma(\Gamma(X\ro{}_{-}, X\rt\o) X\ro{}_{+-}X\rt\t,b)}\di_{\BG}X\rt\th\\
     =&X\o{}_{+}\overline{\Gamma(X\o{}_{-} X\t, b)}\di_{B}X\th 
 \end{align*}
which reduces to the same expression. Next, we  check the counity property,
 \begin{align*}
     (\id\di_{\BG}\varepsilon^{\Gamma})\circ \Delta^{\Gamma}(X)=&\overline{\Gamma(X\rt{}_{+}, X\rt{}_{-})}\CG X\ro\\
     =& X\ro{}_{+}\overline{\Gamma(X\ro{}_{-}, \Gamma(X\rt{}_{+}, X\rt{}_{-}))}\\
     =& X\ro{}_{+}\overline{\Gamma(X\ro{}_{-}, \Gamma(X\rt{}_{+}\o, X\rt{}_{-}\o)X\rt{}_{+}\t X\rt{}_{-}\t)}\\
     =& X\ro{}_{++}\overline{\Gamma(\Gamma(X\ro{}_{-}, X\rt{}_{+}\o)X\ro{}_{+-}X\rt{}_{+}\t, X\rt{}_{-})}\\
     =& X\ro{}_{++}\overline{\Gamma(\Gamma(X\ro{}_{-}, X\rt\o)X\ro{}_{+-}X\rt\t{}_{+}, X\rt\t{}_{-})}\\
     =& X\o{}_{+}\overline{\Gamma(X\o{}_{-}X\t{}_{+}, X\t{}_{-})}\\
     =& X_{+}\o{}_{+}\overline{\Gamma(X_{+}\o{}_{-}X_{+}\t{}, X{}_{-})}= X_{+}\overline{ \varepsilon(X{}_{-})}=X.
 \end{align*}
 By the similar method, one has that $(\id\di_{\BG}\varepsilon^{\Gamma})\circ \Delta^{\Gamma}=\id$. In order to show the coassociativity, we first show that
 \[(\Delta\di_{\BG} \id)\circ \GH^{-1}=(\id\di_{\BG}\GH^{-1})\circ (\Delta\di_{B} \id):\cL\di_{B}\cL\to\cL\di_{B}\cL\di_{\BG}\cL,\]
 which is well defined as $\Delta$ is $\BG$-bilinear. Moreover, this is equivalent to 
 \[(\Delta\di_{B} \id)\circ \GH=(\id\di_{B}\GH)\circ (\Delta\di_{\BG} \id):\cL\di_{\BG}\cL\to \cL\di_{B}\cL\di_{B}\cL, \]
which indeed holds as
 \begin{align*}
     (\Delta\di_{B} \id)\circ &\GH(X\ot Y)=X_{+}\o\ot X_{+}\t\overline{\Gamma(X_{-}, Y\o)}\ot Y\t\\
     =&X\o\ot X\t{}_{+}\overline{\Gamma(X\t{}_{-}, Y\o)}\ot Y\t=(\id\di_{B}\GH)\circ (\Delta\di_{\BG} \id)(X\ot Y).
 \end{align*}
 One can similarly show that
  \begin{equation}\label{Gammaid}(\id\di_{\BG} \Delta)\circ \GH^{-1}=(\GH^{-1}\di_{\BG}\id)\circ (\id\di_{B} \Delta):\cL\di_{B}\cL\to\cL\di_{\BG}\cL\di_{B}\cL.\end{equation}
  Moreover, we have that
 \[(\GH\di_{\BG}\id)\circ(\id\di_{\BG}\GH)=(\id\di_{\BG}\GH)\circ(\GH\di_{\BG}\id):\cL\di_{\BG}\cL\di_{\BG}\cL\to \cL\di_{B}\cL\di_{B}\cL\]
since
 \begin{align*}
     (\GH\di_{\BG}\id)\circ(\id\di_{\BG}\GH)(X\ot Y\ot Z)=&(\GH\di_{\BG}\id)(X\ot Y_{+}\overline{\Gamma(Y_{-}, Z\o)}\di_{B}Z\t)\\
     =&X_{+}\overline{\Gamma(X_{-}, Y_{+}\o)}\di_{B}Y_{+}\t\overline{\Gamma(Y_{-}, Z\o)}\di_{B}Z\t\\
     =&X_{+}\overline{\Gamma(X_{-}, Y\o)}\di_{B}Y\t{}_{+}\overline{\Gamma(Y\t{}_{-}, Z\o)}\di_{B}Z\t\\
     =&(\id\di_{\BG}\GH)(X_{+}\overline{\Gamma(X_{-}, Y\o)}\di_{B}Y\t\ot Z)\\
     =&(\id\di_{\BG}\GH)\circ(\GH\di_{\BG}\id)(X\ot Y\ot Z).
 \end{align*}
Using these results, it follows that  
 \begin{align*}
      (\DG\ot \id)\circ \DG=&(\GH^{-1}\ot \id)\circ(\Delta\ot\id )\circ \GH^{-1}\circ \Delta\\
      =&(\GH^{-1}\ot \id)\circ(\id\ot \GH^{-1})\circ(\Delta\ot\id )\circ \Delta\\
      =&(\id\ot \GH^{-1})\circ(\GH^{-1}\ot \id)\circ(\id\ot\Delta )\circ \Delta\\
      =&(\id\ot \GH^{-1})\circ(\id\ot\Delta )\circ \GH^{-1}\circ \Delta=(\id\ot \DG)\circ \DG.
 \end{align*}
\end{proof}

\begin{theorem}\label{thm. twisted Hopf algebroid}
    Let $\Gamma$ be an invertible left 2-cocycle on a left $B$-Hopf algebroid $\CL$. Then $\cL^\Gamma$ is a left Hopf algebroid over $\BG$ with the $\BG^e$-ring structure and $\BG$-coring structure given above. In addition, if $\CL$ is an anti-left Hopf algebroid then $\CL^{\Gamma}$ is also an anti-left Hopf algebroid.
\end{theorem}
\begin{proof}
    We first show that $\varepsilon^\Gamma$ is a left character,
   \begin{align*}
      \varepsilon^\Gamma(X\CG Y)=&\Gamma(\Gamma(X\o, Y\o)X\t{}_{++}Y\t{}_{++}, \Gamma(Y\t{}_{-}, X\t{}_{-})Y\t{}_{+-}X\t{}_{+-}) \\
      =&\Gamma(\Gamma(X\o, Y\o)X\t{}_{+}Y\t{}_{+}, \Gamma(Y\t{}_{-}\o, X\t{}_{-}\o)Y\t{}_{-}\t X\t{}_{-}\t) \\
      =&\Gamma(\Gamma(\Gamma(X\o, Y\o)X\t{}_{+}\o Y\t{}_{+}\o,Y\t{}_{-}\o) X\t{}_{+}\t Y\t{}_{+}\t Y\t{}_{-}\t, X\t{}_{-}) \\
      =&\Gamma(\Gamma(\Gamma(X\o, Y\o)X\t{}_{+}\o Y\t{}_{+},Y\t{}_{-}) X\t{}_{+}\t , X\t{}_{-}) \\
      =&\Gamma(\Gamma(\Gamma(X\o, Y\o)X\t Y\t{}_{+},Y\t{}_{-}) X\th{}_{+} , X\th{}_{-}) \\
      =&\Gamma(\Gamma(\Gamma(X\o, Y_{+}\o)X\t Y{}_{+}\t,Y{}_{-}) X\th{}_{+} , X\th{}_{-})\\
      =&\Gamma(\Gamma(X\o, \Gamma(Y_{+}\o, Y{}_{-}\o)Y_{+}\t Y{}_{-}\t) X\t{}_{+} , X\t{}_{-})\\
      =&\Gamma(\Gamma(X\o, \Gamma(Y_{+}, Y{}_{-})) X\t{}_{+} , X\t{}_{-})=\varepsilon^\Gamma(X\CG \varepsilon^\Gamma(Y)).
   \end{align*}
   We also have
   \begin{align*}
       \varepsilon^\Gamma(X\CG \overline{b})=&\Gamma(X_{++}, \Gamma(b, X_{-})X_{+-})=\Gamma(X_{+}, \Gamma(b, X_{-}\o)X_{+}\t)=\Gamma(\Gamma(X_{+}\o,b)X_{+}\t, X_{-})\\
       =&\Gamma(\Gamma(X\o,b)X\t{}_{+}, X\t{}_{-})= \varepsilon^\Gamma(X\CG b).
   \end{align*}
   We next check that the coproduct is an algebra map. Indeed,
   \begin{align*}
       \GH(&X\ro\CG Y\ro\di_{\BG} X\rt\CG Y\rt)\\
       =&\Gamma(X\ro\o, Y\ro\o)X\ro\t{}_{++}Y\ro\t{}_{++}\\
       &\overline{\Gamma(\Gamma(Y\ro\t{}_{-}, X\ro\t{}_{-})Y\ro\t{}_{+-}X\ro\t{}_{+-}, \Gamma(X\rt\o, Y\rt\o)X\rt\t Y\rt\t)}\\
       &\di_{B}X\rt\th{}_{+} Y\rt\th{}_{+}\overline{\Gamma(Y\rt\th{}_{-}, X\rt\th{}_{-})}\\
       =&\Gamma(X\ro\o, Y\ro\o)X\ro\t{}_{+}Y\ro\t{}_{+}\\
       &\overline{\Gamma(\Gamma(Y\ro\t{}_{-}\o, X\ro\t{}_{-}\o)Y\ro\t{}_{-}\t X\ro\t{}_{-}\t, \Gamma(X\rt\o, Y\rt\o)X\rt\t Y\rt\t)}\\
       &\di_{B}X\rt\th{}_{+} Y\rt\th{}_{+}\overline{\Gamma(Y\rt\th{}_{-}, X\rt\th{}_{-})}\\
       =&\Gamma(X\ro\o, Y\ro\o)X\ro\t{}_{+}Y\ro\t{}_{+}\\
       &\overline{\Gamma(\Gamma(\Gamma(Y\ro\t{}_{-}\o, X\ro\t{}_{-}\o)Y\ro\t{}_{-}\t X\ro\t{}_{-}\t, X\rt\o)Y\ro\t{}_{-}\th X\ro\t{}_{-}\th X\rt\t,  Y\rt\o)}\\
          &\di_{B}X\rt\th{}_{+} Y\rt\t{}_{+}\overline{\Gamma(Y\rt\t{}_{-}, X\rt\th{}_{-})}\\
          =&\Gamma(X\ro\o, Y\ro\o)X\ro\t{}_{+}Y\ro\t{}_{+}\\
       &\overline{\Gamma(\Gamma(Y\ro\t{}_{-}\o, \Gamma(X\ro\t{}_{-}\o, X\rt\o) X\ro\t{}_{-}\t X\rt\t)Y\ro\t{}_{-}\t X\ro\t{}_{-}\th X\rt\th,  Y\rt\o)}\\
          &\di_{B}X\rt\fo{}_{+} Y\rt\t{}_{+}\overline{\Gamma(Y\rt\t{}_{-}, X\rt\fo{}_{-})}\\
           =&\Gamma(X\ro\o, Y\ro\o)X\ro\t{}_{++}Y\ro\t{}_{+}\\
       &\overline{\Gamma(\Gamma(Y\ro\t{}_{-}\o, \Gamma(X\ro\t{}_{-}, X\rt\o) X\ro\t{}_{+-}\o X\rt\t)Y\ro\t{}_{-}\t X\ro\t{}_{+-}\t X\rt\th,  Y\rt\o)}\\
          &\di_{B}X\rt\fo{}_{+} Y\rt\t{}_{+}\overline{\Gamma(Y\rt\t{}_{-}, X\rt\fo{}_{-})}\\
          =&\Gamma(X\ro{}_{+}\o, Y\ro\o)X\ro{}_{+}\t{}_{+}Y\ro\t{}_{+}\\
       &\overline{\Gamma(\Gamma(Y\ro\t{}_{-}\o, \Gamma(X\ro{}_{-}, X\rt\o) X\ro{}_{+}\t{}_{-}\o X\rt\t)Y\ro\t{}_{-}\t X\ro{}_{+}\t{}_{-}\t X\rt\th,  Y\rt\o)}\\
          &\di_{B}X\rt\fo{}_{+} Y\rt\t{}_{+}\overline{\Gamma(Y\rt\t{}_{-}, X\rt\fo{}_{-})}\\
           =&\Gamma(X\o\o, Y\ro\o)X\o\t{}_{+}Y\ro\t{}_{+}\\
       &\overline{\Gamma(\Gamma(Y\ro\t{}_{-}\o,  X\o\t{}_{-}\o X\t)Y\ro\t{}_{-}\t X\o\t{}_{-}\t X\th,  Y\rt\o)}\\
          &\di_{B}X\fo{}_{+} Y\rt\t{}_{+}\overline{\Gamma(Y\rt\t{}_{-}, X\fo{}_{-})}\\
           =&\Gamma(X\o, Y\ro\o)X\t Y\ro\t{}_{+}\overline{\Gamma(Y\ro\t{}_{-},  Y\rt\o)}\di_{B}X\th{}_{+} Y\rt\t{}_{+}\overline{\Gamma(Y\rt\t{}_{-}, X\th{}_{-})}\\
           =&\Gamma(X\o, Y\ro{}_{+}\o)X\t Y\ro{}_{+}\t\overline{\Gamma(Y\ro{}_{-},  Y\rt\o)}\di_{B}X\th{}_{+} Y\rt\t{}_{+}\overline{\Gamma(Y\rt\t{}_{-}, X\th{}_{-})}\\
           =&\Gamma(X\o, Y\o)X\t Y\t\di_{B}X\th{}_{+} Y\th{}_{+}\overline{\Gamma(Y\th{}_{-}, X\th{}_{-})}\\
           =&\Delta(X\CG Y)=\GH\circ \Delta^{\Gamma}(X\CG Y).
   \end{align*}
 To show that $\cL^\Gamma$ is a left Hopf algebroid, it is sufficient to show that  
    \[
\begin{tikzcd}
  &\cL^{\Gamma}\ot_{\overline{\BG}}\cL^{\Gamma} \arrow[d, "\GH"] \arrow[r, "\lambda^{\Gamma}"] & \cL^{\Gamma}\di_{\BG}\cL^{\Gamma} \arrow[d, "\GH"] &\\
   & \cL\ot_{\overline{B}}\cL   \arrow[r, "\lambda"] & \cL\di_{B}\cL, &
\end{tikzcd}
\]
commutes, where the left $\GH$  given by
\[\GH(X\ot_{\overline{\BG}} Y)=X_{+}\ot_{\overline{B}}Y_{+}\overline{\Gamma(Y_{-}, X_{-})}\]
is invertible as we consider both $\cL$ have the left $\cL$-comodule structure given by (\ref{regular left comodule structure}). We see that on the one hand,
\begin{align*}
    \lambda\circ \GH(X\ot_{\overline{\BG}} Y)=X_{+}\o\di_{B}X_{+}\t Y_{+}\overline{\Gamma(Y_{-}, X_{-})}.
\end{align*}
On the other hand,
\begin{align*}
    \GH\circ& \lambda^{\Gamma}(X\ot_{\overline{\BG}} Y)\\
    =&\GH(X\ro\diamond_{\BG} \Gamma(X\rt\o, Y\o)X\rt\t{}_{+}Y\t{}_{+}\overline{\Gamma(Y\t{}_{-}, X\rt\t{}_{-})}\,)\\
    =&X\ro{}_{+}\overline{\Gamma(X\ro{}_{-}, \Gamma(X\rt\o, Y\o)X\rt\t Y\t)}\di_{B}X\rt\th{}_{+}Y\th{}_{+}\overline{\Gamma(Y\th{}_{-}, X\rt\th{}_{-})}\\
    =&X\ro{}_{+}\overline{\Gamma(\Gamma(X\ro{}_{-}\o, X\rt\o)X\ro{}_{-}\t X\rt\t,  Y\o)}\di_{B}X\rt\th{}_{+}Y\t{}_{+}\overline{\Gamma(Y\t{}_{-}, X\rt\th{}_{-})}\\
    =&X\o{}_{+}\overline{\Gamma(X\o{}_{-} X\t,  Y\o)}\di_{B}X\th{}_{+}Y\t{}_{+}\overline{\Gamma(Y\t{}_{-}, X\th{}_{-})}\\
    =&X\o\di_{B}X\t{}_{+}Y{}_{+}\overline{\Gamma(Y{}_{-}, X\t{}_{-})}.
\end{align*}
The last statement for $\cL^\Gamma$  an anti-left Hopf algebroid can be similarly shown via the commutative diagram
 \[
\begin{tikzcd}
  &\cL^{\Gamma}\ot_{\BG}\cL^{\Gamma} \arrow[d, "\GH"] \arrow[r, "\mu^{\Gamma}"] & \cL^{\Gamma}\di_{\BG}\cL^{\Gamma} \arrow[d, "\GH"] &\\
   & \cL\ot_{B}\cL   \arrow[r, "\mu"] & \cL\di_{B}\cL. &
\end{tikzcd}
\]
\end{proof}

\begin{lemma}\label{lem. comodule equ}
    Let $\cL$ be a left $B$-Hopf algebroid and $\Gamma$ be an invertible left 2-cocycle on $\cL$. Then $({}^{\cL}\CM,\ot_{B})\cong ({}^{\cL^{\Gamma}}\CM,\ot_{\BG})$ as monoidal categories.
\end{lemma}
\begin{proof}
    Given a left $\cL$-comodule $M$, we know it is a $\BG$-bimodule by Proposition~\ref{prop. twisted bimodule structures}. Moreover, $M$ is a left $\cL^{\Gamma}$-comodule with  coaction given by
    \[\delta^{\cL^{\Gamma}}:=\GH{}^{-1}\circ \delta^{\cL}:M\to \cL^{\Gamma}\diamond_{\BG}M.\]
    By a similar proof as for $\Delta^{\Gamma}$ and $\varepsilon^{\Gamma}$ above, one can check that this is a well defined coaction. Denote $\Gamma(M)$ be the $\cL^{\Gamma}$-comodule with the twisted coaction. We write the coaction as $\delta^{\cL^{\Gamma}}(m):=m\rmo\ot m\rz$ for all $m\in M$.
    If $N$  is another left $\cL$-comodule then $\Gamma(M)\ot_{\BG}\Gamma(N)$ is also a left $\cL^{\Gamma}$-comodule with codiaginal $\cL^{\Gamma}$-coaction. We set
    \[\GH_{M,N}:\Gamma(M)\ot_{\BG}\Gamma(N)\to \Gamma(M\ot_{B}N),\qquad m\ot n\mapsto \Gamma(m\mo, n\mo)m\z\ot n\z,\]
for all $M, N\in {}^{\cL}\CM$ as the natural transformation needed for an equivalence of monoidal categories. This is $\BG$-bilinear by Lemma \ref{lem. coherent map is BG-bilinear}. By a similar methods as our proof that $\Delta^{\Gamma}$ is an algebra map for the twisted product, one can  check that $\GH_{M,N}$ is left $\cL^{\Gamma}$-colinear. Moreover, $\GH$ satisfies the required coherence condition and is invertible since $\Gamma$ is an invertible 2-cocycle.
\end{proof}

\begin{remark}
 Similarly, for a right 2-cocycle $\Sigma$ on an anti-left Hopf algebroid $\cL$, one can cotwist $\cL$ to a new anti-left Hopf algebroid over a deformed base algebra, with twisted product given by
\[X\cdot_{\Sigma}Y=\overline{\Sigma(X\t, Y\t)}X\o{}_{[+]}Y\o{}_{[+]}\Sigma(Y\o{}_{[-]},X\o{}_{[-]}).\]
\end{remark}

\begin{proposition}\label{prop. cocomute coaction}
     Let $\cL$ be a left $B$-Hopf algebroid and $\Gamma$ be an invertible left 2-cocycle on $\cL$. Then
\begin{itemize}\item[(1)] $\Delta$ is $(\BG)^{e}$-bilinear and $\DG$ is  $B^e$-bilinear;
        \item[(2)] $(\id\ot\Delta^{\Gamma})\circ\Delta=(\Delta\ot\id)\circ\Delta^{{\Gamma}},\qquad(\id\ot\Delta)\circ\Delta^{\Gamma}=(\Delta^{\Gamma}\ot\id)\circ\Delta$; 
                \item[(3)]  
        \begin{align*}
            X_{+}\ro\ot X_{+}\rt\ot X_{-}=&X\ro\ot X\rt{}_{+}\ot X\rt{}_{-}\in \cL\di_{\BG}\cL\ot_{\overline{B}}\cL,\\
            X_{\p}\o\ot X_{\p}\t\ot X_{\m}=&X\o\ot X\t{}_{\p}\ot X\t{}_{\m}\in \cL\di_{B}\cL\ot_{\overline{\BG}}\cL;
        \end{align*}
        \item[(4)]  $ X_{\p +}\ot X_{\m} \ot X_{\p -}=X_{+}\ot X_{-}\ro\ot X_{-}\rt\in\cL\ot_{\overline{B}}(\cL\di_{\BG} \cL)$
       \end{itemize}
      for all $X\in\CL$, where we  write $X_{\p}\ot_{\BG} X_{\m}=(\lambda^{\Gamma}){}^{-1}(X\diamond_{\BG}1)$.\end{proposition}
\begin{proof}
(1) By direct computation, it is not hard to see that $\Delta$ is $(\BG)^{e}$-bilinear. That  $\DG$ is $B^e$-bilinear follows from the fact that $\Delta$ is $B^e$-bilinear and 
    \begin{align*}
        \GH(bX\di_{\BG} Y)=(b\di_{B} 1)\GH(X\di_{\BG}Y)\quad,\quad  \GH(X\di_{\BG} \overline{b}Y)=(1\di_{B} \overline{b})\GH(X\di_{\BG}Y).
    \end{align*}
By (1), it follows  that 
\begin{equation}
  (a\CG\Bar{a'}\CG X\CG b\CG\Bar{b'})_{+}\ot_{\BB}(a\CG\Bar{a'}\CG X\CG b\CG\Bar{b'})_{-}=a\CG X_{+}\CG b\ot_{\BB}b'\CG X_{-}\CG a'\end{equation}
  \begin{equation}
       (a\Bar{a'}Xb\Bar{b'})_{\p}\ot_{\BB}(a\Bar{a'}Xb\Bar{b'})_{\m}=aX_{\p}b\ot_{\BB}b'X_{\m}a'\end{equation}
       for all $a,a',b,b'\in B$. Therefore, (2), (3) and (4) are well defined.
              For (2),  we have
     \begin{align*}
      (\DG\ot \id)\circ \Delta=&(\GH^{-1}\ot \id)\circ(\Delta\ot\id )\circ  \Delta
      =(\GH^{-1}\ot \id)\circ(\id\ot\Delta )\circ \Delta\\
      =&(\id\ot\Delta )\circ \GH^{-1}\circ \Delta
      =(\id\ot \Delta)\circ \DG,
 \end{align*}
 where in the 3rd step we use (\ref{Gammaid}).
  Similarly for the 2nd equality. For (3), it is not hard to see that the formulae are well defined. By applying $\id\ot \lambda$ on the left hand side, we have
  \begin{align*}
  (\id\ot \lambda)( X_{+}\ro\ot X_{+}\rt\ot X_{-})=&X_{+}\ro\ot X_{+}\rt\o\ot X_{+}\rt\t X_{-}
  =X_{+}\o\ro\ot X_{+}\o\rt\ot X_{+}\t X_{-}\\
  =&X\ro\ot X\rt\ot 1=(\id\ot \lambda)(X\ro\ot X\rt{}_{+}\ot X\rt{}_{-}).
  \end{align*}
  The second equality can be shown similarly by applying $\id\ot \lambda^{\Gamma}$ on both sides.
  For (4), we have
  \begin{align*}
      X_{\p +}\ot X_{\m}\ot X_{\p -}=&X_{\p \p +}\ot X_{\m}\ot X_{\p\m +}\overline{\Gamma(X_{\p\m -}, X_{\p\p -})}\\
      =&X_{\p +}\ot X_{\m}\ro\ot X_{\m}\rt{}_{ +}\overline{\Gamma(X_{\m}\rt{}_{-}, X_{\p -})}\\
       =&X_{\p +}\ot X_{\m +}\ro\ot X_{\m +}\rt\overline{\Gamma(X_{\m-}, X_{\p -})}=X_{+}\ot X_{-}\ro\ot X_{-}\rt.
  \end{align*}
\end{proof}
By a similar proof, we have
\begin{corollary}\label{cor. cocommute coaction}
    Let $\cL$ be a left $B$-Hopf algebroid and $\Gamma$ be an invertible left 2-cocycle on $\cL$. If $M$ is a left $\cL$-comodule with coaction $\delta^{\cL}$ (hence a left $\cL^{\Gamma}$-comodule with coaction $\delta^{\cL^{\Gamma}}$ by Lemma \ref{lem. comodule equ}), then
    \[(\Delta\ot \id)\circ \delta^{\cL^{\Gamma}}=(\id\ot \delta^{\cL^{\Gamma}})\circ \delta^{\cL},\quad (\Delta^{\Gamma}\ot \id)\circ \delta^{\cL}=(\id\ot \delta^{\cL})\circ \delta^{\cL^{\Gamma}}.\]
\end{corollary}

\subsection{Comparing with Drinfeld cotwist in \cite{Boehm,HM22,HM23}}\label{sec:Dri}
Recall that a convolution-invertible $B$-bilinear left 2-cocycle $\Gamma$ on a left Hopf algebroid $\cL$ 
is a convolution invertible element $\Gamma\in {}_{B^{e}}\Hom(\cL\otimes_{B^e}\cL, B)$ such that
\begin{align*}(i)\quad &\Gamma(X, \Gamma(\one{Y}, \one{Z})\two{Y}\two{Z})=\Gamma(\Gamma(\one{X}, \one{Y})\two{X}\two{Y}, Z);\\
(ii)\quad & \Gamma(1_{\cL}, X)=\varepsilon(X)=\Gamma(X, 1_{\cL});\\
(iii)\quad &\Gamma(X, Y\,b)=\Gamma(X, Y\,\overline{b})\end{align*}
for all $X, Y, Z\in \cL$. The cotwisted left Hopf algebroid $\CL_D^\Gamma$ has the original coring structure and a twisted product given by
\[X\ast_{\Gamma}Y=\Gamma(X\o, Y\o)\,\overline{\Gamma^{-1}(X\th, Y\th)}\, X\t\,Y\t,\]
for all $X,Y\in \cL$. It is not hard to see that  a convolution-invertible $B$-bilinear left 2-cocycle $\Gamma$ in the above sense is a special case of an invertible left 2-cocycle on $\cL$ in our more general sense in Definition~\ref{Lcotwist}, where $B^\Gamma=B$ (due to the $B^e$-linearity). It factors through $\cL\ot_{\BB}\cL$ by the composition with the projection $\pi:\cL\ot_{\BB}\cL\to \cL\ot_{B^{e}}\cL$ and this composition, which we still denote $\Gamma$ for convenience, provides the cocycle in our more general sense. Moreover, we see that $\GH$ in Lemma~\ref{lem. coherent map is BG-bilinear} is invertible with
    \begin{align*}
        \GH^{-1}(m\ot_{B} n)=\Gamma^{-1}(m\mo, n\mo)\,m\z\ot n\z,
    \end{align*}
    for all $m\in M\in {}^{\cL}\CM$ and  $n\in N\in {}^{\cL}\CM$. Indeed,
    \begin{align*}
        \GH\circ \GH^{-1}(m\ot_{B} n)=&\Gamma(m\mt, n\mt)\Gamma^{-1}(m\mo, n\mo)m\z\ot_{B}n\z=m\ot_{B} n,
    \end{align*}
   and similarly for $ \GH^{-1}\circ \GH=\id_{M\ot_{B}N}$.  
    
    \begin{theorem}\label{thm. comparison}
      Let $\cL$ be a left Hopf algebroid over $B$, and $\Gamma$  a convolution-invertible $B$-bilinear left 2-cocycle on $\cL$.  Then $\cL^{\Gamma}\cong \cL_{D}^{\Gamma}$ as left Hopf algebroids. The isomorphism is given by
      \[\Psi:\cL^{\Gamma}\to \cL_{D}^{\Gamma},\qquad X\mapsto \overline{\Gamma(X\t{}_{+}, X\t{}_{-})}\, X\o,\]
      with inverse
      \[\Psi^{-1}(X)=X\o{}_{+}\overline{\Gamma^{-1}(X\o{}_{-},X\t)},\]
      for all $X\in \cL$.
\end{theorem}
\begin{proof} We first observe that $\Psi$ and $\Psi^{-1}$ are well defined. As $\Gamma$ is left $B$-linear and factors through the balanced tensor product $\ot_{B^e}$, it is clear that the base algebra is unchanged.  Now, we check that $\Psi$ is a coring map. It is not hard to see that $\Psi$ is left $B^{e}$-linear and $\varepsilon_{D}^{\Gamma}\circ \Psi=\varepsilon\circ \Psi=\varepsilon^{\Gamma}$. Moreover, $\Delta^{\Gamma}(X)=X\o{}_{+}\overline{\Gamma^{-1}(X\o{}_{-}, X\t)}\di X\th$ since it is not hard to check that $\GH\circ \Delta^{\Gamma}(X)=\Delta(X)$. Therefore, we have
    \begin{align*}
       (\Psi\di_{B}& \Psi)\circ \Delta^{\Gamma}(X)\\
       =& (\Psi\di_{B} \Psi)(X\o{}_{+}\overline{\Gamma^{-1}(X\o{}_{-}, X\t)}\di X\th)\\
       =&\overline{\Gamma(X\o{}_{+}\t{}_{+},\Gamma^{-1}(X\o{}_{-}, X\t)X\o{}_{+}\t{}_{-})}X\o{}_{+}\o\di \overline{\Gamma(X\fo{}_{+},X\fo{}_{-})}X\th\\
       =&\overline{\Gamma(X\t{}_{++},\Gamma^{-1}(X\t{}_{-}, X\th)X\t{}_{+-})}X\o\di \overline{\Gamma(X\fiv{}_{+},X\fiv{}_{-})}X\fo\\
       =&\overline{\Gamma(X\t{}_{+},\Gamma^{-1}(X\t{}_{-}\o, X\th)X\t{}_{-}\t)}X\o\di \overline{\Gamma(X\fiv{}_{+},X\fiv{}_{-})}X\fo\\
       =&\overline{\Gamma(X\t{}_{+}\o,X\t{}_{-}\o\, X\th)\Gamma^{-1}(X\t{}_{+}\t\,X\t{}_{-}\t,X\fo)}X\o\di \overline{\Gamma(X\si{}_{+},X\si{}_{-})}X\fiv\\
       =&\overline{\Gamma(X\t{}_{+},X\t{}_{-}\, X\th)}X\o\di \overline{\Gamma(X\fiv{}_{+},X\fiv{}_{-})}X\fo\\
       =&X\o\di \overline{\Gamma(X\th{}_{+},X\th{}_{-})}X\t =\Delta\circ \Psi(X)=\Delta_{D}^{\Gamma}\circ \Psi(X),
    \end{align*}
where the 5th step uses $\Gamma(\one{X}, \one{Y}\one{Z})\Gamma^{-1}(\two{X}\two{Y}, \two{Z})=\Gamma(X\,\Gamma^{-1}(\one{Y}, Z), \two{Y})$ in \cite{HM22}. As a result, $\Psi$ is a coring map. Next, we check that $\Psi$ is also an algebra map,
    \begin{align*}
        \Psi(&X\cdot_{\Gamma}Y)\\
        =&\Psi(\Gamma(X\o, Y\o)\,X\t{}_{+}\,Y\t{}_{+}\overline{\Gamma(Y\t{}_{-}, X\t{}_{-})}\, )\\
        =&\Gamma(X\o, Y\o)\overline{\Gamma(X\th{}_{++}\,Y\th{}_{++},\Gamma(Y\th{}_{-},X\th{}_{-})Y\th{}_{+-}X\th{}_{+-})}X\t Y\t\\
         =&\Gamma(X\o, Y\o)\overline{\Gamma(X\th{}_{+}\,Y\th{}_{+},\Gamma(Y\th{}_{-}\o,X\th{}_{-}\o)Y\th{}_{-}\t\,X\th{}_{-}\t)}X\t Y\t\\
         =&\Gamma(X\o, Y\o)\overline{\Gamma^{-1}(X\th{}_{+}\o, Y\th{}_{+}\o)}\\
         &\overline{\Gamma(X\th{}_{+}\t,\Gamma(Y\th{}_{+}\t,\Gamma(Y\th{}_{-}\o,X\th{}_{-}\o)Y\th{}_{-}\t\,X\th{}_{-}\t)Y\th{}_{+}\th\,Y\th{}_{-}\th\,X\th{}_{-}\th)}X\t Y\t\\
         =&\Gamma(X\o, Y\o)\overline{\Gamma^{-1}(X\th{}_{+}\o, Y\th{}_{+}\o)}\\
         &\overline{\Gamma(X\th{}_{+}\t,\Gamma(Y\th{}_{+}\t,\Gamma(Y\th{}_{-}\o,X\th{}_{-}\o)Y\th{}_{-}\t\,X\th{}_{-}\t)X\th{}_{-}\th)}X\t Y\t\\
          =&\Gamma(X\o, Y\o)\overline{\Gamma^{-1}(X\th{}_{+}\o, Y\th{}_{+}\o)}\\
         &\overline{\Gamma(X\th{}_{+}\t,\Gamma(\Gamma(Y\th{}_{+}\t, Y\th{}_{-}\o)Y\th{}_{+}\th\,Y\th{}_{-}\t,X\th{}_{-}\o)X\th{}_{-}\t)}X\t Y\t\\
         =&\Gamma(X\o, Y\o)\overline{\Gamma^{-1}(X\th{}_{+}\o, Y\th{}_{+}\o)\Gamma(X\th{}_{+}\t,\Gamma(Y\th{}_{+}\t, Y\th{}_{-})X\th{}_{-})}X\t Y\t\\
         =&\Gamma(X\o, Y\o)\overline{\Gamma^{-1}(X\th, Y\th)\Gamma(X\fo{}_{+},\Gamma(Y\fo{}_{+}, Y\fo{}_{-})X\fo{}_{-})}X\t Y\t\\
         =&\Gamma(X\o, Y\o)\overline{\Gamma^{-1}(\overline{\Gamma(X\fo{}_{+}\,\Gamma(Y\fo{}_{+},Y\fo{}_{-}),X\fo{}_{-})}X\th, Y\th)}X\t\,Y\t\\
          =&\Gamma(X\o, Y\o)\overline{\Gamma^{-1}(\overline{\Gamma(X\fo{}_{+},X\fo{}_{-})}X\th, \overline{\Gamma(Y\fo{}_{+},Y\fo{}_{-})}Y\th)}X\t\,Y\t\\
          =&\Psi(X)\ast_{\Gamma}\Psi(Y),
    \end{align*}
    where the 4th step uses $\Gamma(XY,Z)=\Gamma^{-1}(X\o, Y\o)\Gamma(X\t,\Gamma(Y\t, Z\o)Y\th\,Z\t)$. To see that $\Psi$ is invertible, we have
    \begin{align*}
        \Psi^{-1}\,(\Psi(X))=&X\o{}_{+}\overline{\Gamma^{-1}(X\o{}_{-},\overline{\Gamma(X\th{}_{+},X\th{}_{-})}X\t)}\\
        =&X\o{}_{+}\overline{\Gamma^{-1}(X\o{}_{-},\overline{\Gamma(X\t{}_{+}\t,X\t{}_{-})}X\t{}_{+}\o)}\\
        =&X\o{}_{+}\overline{\Gamma^{-1}(X\o{}_{-}\o\,X\t{}_{+}\o,X\t{}_{-}\o)\Gamma(X\o{}_{-}\t, X\t{}_{+}\t\,X\t{}_{-}\t)}\\
        =&X\o{}_{+}\overline{\Gamma^{-1}(X\o{}_{-}\,X\t{}_{+},X\t{}_{-})}=X_{+}\overline{\varepsilon(X_{-})}=X,
    \end{align*}
    where the 3rd step uses $ \Gamma(\one{X}\one{Y}, \one{Z})\Gamma^{-1}(\two{X}, \two{Y} \two{Z})=\Gamma^{-1}(X, \overline{\Gamma(\two{Y}, Z)}\one{Y})$ in \cite{HM22}. The proof that $\Psi\circ \Psi^{-1}=\id$ is similar.
\end{proof}
\begin{remark}\label{remHopf}
When the base is trivial and we have a Hopf algebra $H$ over $k$, our general notion reduces to a convolution invertible left 2-cocycle $\Gamma:H\ot H\to k$ in the usual sense\cite{Ma:book}, and our general construction gives a Hopf algebra $H^\Gamma$. For all $h\in H$, $h_{+}\ot h_{-}=h\o\ot S(h\t)$, the product is
\[h\CG g=\Gamma(h\o, g\o)h\t g\t \Gamma(S(g\th), S(h\th)),\]
the coproduct and counit are
\[\Delta^{\Gamma}(h)=h\o\Gamma^{-1}(S(h\t), h\th)\ot h\fo,\qquad \varepsilon^{\Gamma}(h)=\Gamma(h\o, S(h\t)).\]
Indeed, we can check that
\begin{align*}
    \GH(h\o\Gamma^{-1}(S(h\t), h\th)\ot h\fo)=&h\o \Gamma^{-1}(S(h\th), h\fo)\Gamma(S(h\t), h\fiv)\ot h\si\\
    =&h\o\ot h\t.
\end{align*}
Moreover,   $H^{\Gamma}$ is isomorphic to the usual Drinfeld cotwist of $H$  (denoted here by $H^{\Gamma}_{D}$) by
\[ \Psi:H^{\Gamma}\to H^{\Gamma}_{D},\quad \Psi(h)=h\o\Gamma(h\t, S(h\th)),\]
    for all $h\in H$. 

    \end{remark}

\subsection{Groupoid structure on 2-cocycles of Hopf algebroids}\label{secgroupoid}

Given a Hopf algebra $H$ and an invertible left 2-cocycle $\gamma$ on $H$, by applying the Drinfeld cotwist on $H$, we get another Hopf algebra $H^{\gamma}$ with cotwisted product, see\cite{Ma:book}. Moreover, $\gamma^{-1}$  is a left 2-cocycle on $H^{\gamma}$ and by applying the Drinfeld cotwist by $\gamma^{-1}$ on $H^{\gamma}$,  we can cotwist $H^{\gamma}$ back to the original $H$. Moreover, if $\sigma$ is an invertible left 2-cocycle on $H^{\gamma}$, we can cotwist $H^{\gamma}$ to a new Hopf algebra $(H^{\gamma})^{\sigma}$ which is equal to $H^{\sigma\ast \gamma}$, where $\sigma\ast \gamma$ is an invertible left 2-cocycle on $H$ given by
\[(\sigma\ast \gamma)(h,g)=\sigma(h\o, g\o)\gamma(h\t,g\t),\]
for all $h,g\in H$. Hence,  the collection of 2-cocycles and the cotwisted Hopf algebras can be viewed as a groupoid. The source of $\gamma$ is $H$ and the target of $\gamma$ is $H^{\gamma}$. The product of the 2-cocycles is given by convolution product as above.

Motivated by this observation for Hopf algebras, is there a groupoid structure on the collections of Hopf algebroids? This is clear if the 2-cocycles are both left $B$-linear and left $\overline{B}$-linear as in Section~\ref{sec:Dri} and \cite{HM23}, as we can define the convolution product. However, we cannot  define the convolution product of two 2-cocycles  as in Definition \ref{Lcotwist} because these are only left $\overline{B}$-linear. In this section, we show that there is still an analogous groupoid structure on Hopf algebroids and their 2-cocycles, however, the product is no longer the convolution product. We will also see that if $\Gamma$ is an invertible left 2-cocycle on a left Hopf algebroid $\cL$ then we can twist $\cL^{\Gamma}$ back to $\cL$.

\begin{lemma}\label{prop. left inverse}
      Let $\cL$ be a left Hopf algebroid over $B$ and $\Gamma$ an invertible left 2-cocycle on $\cL$. Then $\Sigma:\cL^{\Gamma}\ot_{\BG}\cL^{\Gamma}\to \BG$ given by
    \[\Sigma(X, Y)=\Gamma(X_{+}Y_{+}, \Gamma(Y_{-}\o, X_{-}\o)Y_{-}\t\, X_{-}\t)\]
    is a left $\overline{\BG}$-linear map, such that
    \[\Gamma(\Sigma(X\ro, Y\ro), \Gamma(X\rt\o\, ,Y\rt\o)X\rt\t\, Y\rt\t)=\varepsilon(XY),\]
    for all $X, Y\in\cL$.
\end{lemma}
\begin{proof}
     First, it is easy to see that $\Sigma$ factors through all the balanced tensor products. We also see that $\Sigma$ is left $\overline{\BG}$-linear,
    \begin{align*}
       \Sigma(\overline{b}\CG X, Y)=&\Sigma(X_{+}\overline{\Gamma(X_{-}, b)}, Y)\\
       =&\Gamma(X_{++}Y_{+}, \Gamma(Y_{-}\o, \Gamma(X_{-}, b)\,X_{+-}\o)Y_{-}\t\, X_{+-}\t)\\
       =&\Gamma(X_{+}Y_{+}, \Gamma(Y_{-}\o, \Gamma(X_{-}\o, b)\,X_{-}\t)Y_{-}\t\, X_{-}\th)\\
       =&\Gamma(X_{+}Y_{+}, \Gamma(\Gamma(Y_{-}\o,X_{-}\o) X_{-}\t\, Y_{-}\t, b)Y_{-}\th\, X_{-}\th)\\
       =&\Gamma(\Gamma(X_{+}Y_{+}, \Gamma(Y_{-}\o, X_{-}\o)Y_{-}\t\, X_{-}\t), b) =\Sigma(X, Y)\CG b.
    \end{align*}
    By a similar method, one can check that  $\Sigma(X\CG\overline{b}, Y)=\Sigma(X,\overline{b}\CG Y)$.
    Moreover,
    \begin{align*}
        b\CG (\Gamma(X\o, Y\o)X\t\, Y\t)=&\Gamma(b, \Gamma(X\o, Y\o)X\t\, Y\t)X\th\, Y\th\\
        =&\Gamma(b\CG X\o, Y\o)X\t Y\t
    \end{align*}
    for all $b\in B$, $X,Y\in \cL$.
    Therefore, the second stated equality is also well defined. That is indeed holds is:
    \begin{align*}
        \Gamma(&\Sigma(X\ro, Y\ro), \Gamma(X\rt\o\, Y\rt\o)X\rt\t\, Y\rt\t)\\
        =&\Gamma(\Gamma(X\ro{}_{+}Y\ro{}_{+}, \Gamma(Y\ro{}_{-}\o, X\ro{}_{-}\o)Y\ro{}_{-}\t\, X\ro{}_{-}\t), \Gamma(X\rt\o, Y\rt\o)X\rt\t\,Y\rt\t)\\
        =&\Gamma(\Gamma(X\ro{}_{++}\o\,Y\ro{}_{++}\o\,, \Gamma(Y\ro{}_{-}, X\ro{}_{-})Y\ro{}_{+-}\o\, X\ro{}_{+-}\o)X\ro{}_{++}\t\,Y\ro{}_{++}\t\,Y\ro{}_{+-}\t\, X\ro{}_{+-}\t,\\
        \Gamma&(X\rt\o, Y\rt\o)X\rt\t\,Y\rt\t)\\
        =&\Gamma(X\ro{}_{++}\,Y\ro{}_{++}\,,\\
        \Gamma&(\Gamma(Y\ro{}_{-}, X\ro{}_{-})Y\ro{}_{+-}\o\,X\ro{}_{+-}\o,\Gamma(X\rt\o, Y\rt\o)X\rt\t\,Y\rt\t)Y\ro{}_{+-}\t\,X\ro{}_{+-}\t\,X\rt\th\,Y\rt\th),\\
        =&\Gamma(X\ro{}_{+}\,Y\ro{}_{+}\,,\\
        \Gamma&(\Gamma(Y\ro{}_{-}\o, X\ro{}_{-}\o)Y\ro{}_{-}\t\,X\ro{}_{-}\t,\Gamma(X\rt\o, Y\rt\o)X\rt\t\,Y\rt\t)Y\ro{}_{-}\th\,X\ro{}_{-}\th\,X\rt\th\,Y\rt\th),\\
        =&\Gamma(X\ro{}_{+}\,Y\ro{}_{+}\,,\\
        \Gamma&(Y\ro{}_{-}\o,\Gamma( X\ro{}_{-}\o, \Gamma(X\rt\o, Y\rt\o)X\rt\t\,Y\rt\t)\,X\ro{}_{-}\t\,X\rt\th\,Y\rt\th)Y\ro{}_{-}\t\,X\ro{}_{-}\th\,X\rt\fo\,Y\rt\fo)\\
        =&\Gamma(X\ro{}_{+}\,Y\ro{}_{+}\,,\\
        \Gamma&(Y\ro{}_{-}\o,\Gamma( \Gamma(X\ro{}_{-}\o, X\rt\o)X\ro{}_{-}\t\,X\rt\t,Y\rt\o)\,X\ro{}_{-}\th\,X\rt\th\,Y\rt\t)Y\ro{}_{-}\t\,X\ro{}_{-}\fo\,X\rt\fo\,Y\rt\th)\\
        =&\Gamma(X\ro{}_{++}\,Y\ro{}_{+}\,,\\
        \Gamma&(Y\ro{}_{-}\o,\Gamma( \Gamma(X\ro{}_{-}, X\rt\o)X\ro{}_{+-}\o\,X\rt\t,Y\rt\o)\,X\ro{}_{+-}\t\,X\rt\th\,Y\rt\t)Y\ro{}_{-}\t\,X\ro{}_{+-}\th\,X\rt\fo\,Y\rt\th)\\
        =&\Gamma(X\o{}_{+}\,Y\ro{}_{+}\,,
        \Gamma(Y\ro{}_{-}\o,\Gamma( X\o{}_{-}\o\,X\t,Y\rt\o)\,X\o{}_{-}\t\,X\th\,Y\rt\t)Y\ro{}_{-}\t\,X\o{}_{-}\th\,X\fo\,Y\rt\th)\\
        =&\Gamma(X\,Y\ro{}_{+}\,,
        \Gamma(Y\ro{}_{-}\o,\Gamma( 1,Y\rt\o)\,Y\rt\t)Y\ro{}_{-}\t\,Y\rt\th)\\
         =&\Gamma(X\,Y\ro{}_{+}\,,\Gamma(Y\ro{}_{-}\o,\,Y\rt\o)Y\ro{}_{-}\t\,Y\rt\t)\\
        =&\Gamma(X\,Y\o{}_{+}\,,     Y\o{}_{-}\,Y\t)\\
        =&\Gamma(X\,Y\,,     1) =\varepsilon(XY).
    \end{align*}
\end{proof}
\begin{corollary}\label{cor. comodule twist back}
    Let $\cL$ be a left Hopf algebroid over $B$, $\Gamma$ an invertible left 2-cocycle on $\cL$ and $\Sigma$ as defined above. Then for any $M,N\in {}^{\cL}\CM$, we have
    \[\Gamma(\Sigma(m\rmo, n\rmo),\Gamma(m\rz\mt, n\rz\mt)m\rz\mo n\rz\mo)m\rz\z\ot_{B} n\rz\z=m\ot_{B}n,\]
    for all $m\in M$ and $n\in N$. Moreover, we have
    \[\Sigma(X\ro, Y\ro)\CG\,X\rt\CG\,Y\rt =X_{+}Y_{+}\overline{\Gamma(Y_{-},X_{-})},\]
    for all $X, Y\in\CL$. We also have
    \[X_{\p +}\ot Y_{\p +}\,\overline{\Gamma(\Sigma(Y_{\m}, X_{\m}), \Gamma(Y_{\p -}\o, X_{\p -}\o)Y_{\p -}\t\, X_{\p -}\t)}=X\ot Y.\]
\end{corollary}
\begin{proof}
    By Lemma~\ref{prop. left inverse}, we have
    \begin{align*}
        \Gamma(\Sigma(&m\rmo, n\rmo),\Gamma(m\rz\mt, n\rz\mt)m\rz\mo n\rz\mo)m\rz\z\ot_{B} n\rz\z\\
        =&\Gamma(\Sigma(m\mo\ro, n\mo\ro),\Gamma(m\mo\rt\o, n\mo\rt\o)m\mo\rt\t n\mo\rt\t)m\z\ot_{B} n\z\\
        =&\varepsilon(m\mo\,n\mo)m\z\ot_{B} n\z=m\ot_{B}n.
    \end{align*}
     As a special case, we have
    \[\Gamma(\Sigma(X\ro, Y\ro),\Gamma(X\rt\o, Y\rt\o)X\rt\t\, Y\rt\t)X\rt\th\, Y\rt\th=XY,\]
    for all $X, Y\in\CL$. Therefore,
    \begin{align*}
        \Sigma(&X\ro, Y\ro)\CG\,X\rt\CG\,Y\t=\Sigma(X\ro, Y\ro)\CG\,(X\rt\CG\,Y\rt)\\
        =&\Gamma(\Sigma(X\ro, Y\ro),\Gamma(X\rt\o, Y\rt\o)X\rt\t{}_{+}\o\, Y\rt\t{}_{+}\o)\\
        &X\rt\t{}_{+}\t\, Y\rt\t{}_{+}\t\,\overline{\Gamma(Y\rt\t{}_{-},X\rt\t{}_{-})}\\
         =&\Gamma(\Sigma(X\ro, Y\ro),\Gamma(X\rt\o, Y\rt\o)X\rt\t\, Y\rt\t)
        X\rt\th{}_{+}\, Y\rt\th{}_{+}\,\overline{\Gamma(Y\rt\th{}_{-},X\rt\th{}_{-})}\\
        =&\Gamma(\Sigma(X\ro, Y\ro),\Gamma(X\rt{}_{+}\o, Y\rt{}_{+}\o)X\rt{}_{+}\t\, Y\rt{}_{+}\t)
        X\rt{}_{+}\th\, Y\rt{}_{+}\th\,\overline{\Gamma(Y\rt{}_{-},X\rt{}_{-})}\\
        =&\Gamma(\Sigma(X{}_{+}\ro, Y{}_{+}\ro),\Gamma(X{}_{+}\rt\o, Y{}_{+}\rt\o)X{}_{+}\rt\t\, Y{}_{+}\rt\t)
        X{}_{+}\rt\th\, Y{}_{+}\rt\th\,\overline{\Gamma(Y{}_{-},X{}_{-})}\\
        =&X_{+}Y_{+}\overline{\Gamma(Y{}_{-},X{}_{-})},
    \end{align*}
    where the 5th step use Proposition \ref{prop. cocomute coaction}. For the last equality, recall that $\cL$ is a left $\cL$-comodule via $X\mo\ot X\z=X_{-}\ot X_{+}$. Using $\Gamma(X_{\p}\ot X_{\m})=X_{+}\ot X_{-}$ by Theorem \ref{thm. twisted Hopf algebroid}, it is not hard to see (by flipping the position of the two terms) that $X\rmo\ot X\rz=X_{\m}\ot X_{\p}$, which results in the last equality.
\end{proof}
\begin{lemma}\label{lem. sigma is 2 cocycle}
     Let $\cL$ be a left Hopf algebroid over $B$ and $\Gamma$ an invertible left 2-cocycle on $\cL$. Then $\Sigma$ as defined above is a left 2-cocycle on $\cL^{\Gamma}$.
\end{lemma}
\begin{proof}
    It is not hard to see that $\Sigma(X, 1)=\varepsilon^{\Gamma}(X)=\Sigma(1, X)$. For $\Sigma$ a 2-cocycle, on the one hand,
    \begin{align*}
        \Sigma(\Sigma(X\ro, &Y\ro)\CG\,X\rt\,\CG\, Y\rt, Z)\\
        =&\Sigma(X_{+}Y_{+}\overline{\Gamma(Y_{-}, X_{-})}, Z)\\
        =&\Gamma(X_{++}Y_{++}Z_{+}, \Gamma(Z_{-}\o, \Gamma(Y_{-}, X_{-})\,Y_{+-}\o\,X_{+-}\o)Z_{-}\t\, Y_{+-}\t\,X_{+-}\t)\\
        =&\Gamma(X_{+}Y_{+}Z_{+}, \Gamma(Z_{-}\o, \Gamma(Y_{-}\o, X_{-}\o)\,Y_{-}\t\,X_{-}\t)Z_{-}\t\, Y_{-}\th\,X_{-}\th), 
    \end{align*}
    where the 1st step uses Corollary \ref{cor. comodule twist back}. On the other hand,
    \begin{align*}
        \Sigma(X, \Sigma(Y\ro, &Z\ro)\CG Y\rt\CG Z\ro)\\
=&\Sigma(X,Y_{+}Z_{+}\overline{\Gamma(Z_{-}, Y_{-})})\\
=&\Gamma(X_{+}Y_{++}Z_{++}, \Gamma(\Gamma(Z_{-},Y_{-})Z_{+-}\o\,Y_{+-}\o, X_{-}\o)Z_{+-}\t\,Y_{+-}\t\, X_{-}\t)\\
=&\Gamma(X_{+}Y_{+}Z_{+}, \Gamma(\Gamma(Z_{-}\o,Y_{-}\o)Z_{-}\t\,Y_{-}\t, X_{-}\o)Z_{-}\th\,Y_{-}\th\, X_{-}\t).
    \end{align*}
The two expressions are equal since $\Gamma$ is  a 2-cocycle.
\end{proof}

We still need to show that $\Sigma$ is an invertible left 2-cocycle, i.e. that $\Sigma^{\#}$ is invertible.

\begin{lemma}\label{prop. right inverse}
     Let $\cL$ be a left Hopf algebroid over $B$ and $\Gamma$ an invertible left 2-cocycle on $\cL$. Then $\Sigma$ as defined above satisfies $bX=\Sigma(b, X\ro)\CG X\rt$ and
    \[\Sigma(\Gamma(X\o, Y\o), \Sigma(X\t\ro\,, Y\t\ro)\CG\,X\t\rt\,\CG\, Y\t\rt)=\varepsilon^{\Gamma}(X\CG\,Y),\]
    for all $X, Y\in \cL$.
\end{lemma}
\begin{proof}
First, it is not hard to see that $\Sigma(Xb, Y)=\Sigma(X, bY)$  
for all $b\in B$, $X, Y\in \cL$. Also, by the 2nd equality of Corollary \ref{cor. comodule twist back}, we have $bX=\Sigma(b, X\ro)\CG X\rt$.
As a result,
\begin{align*}
    b(\Sigma(X\ro\,, Y\ro)&\CG\,X\rt\,\CG\, Y\rt)\\
    =&\Sigma(b, \Sigma(X\ro\,, Y\ro)\CG\,X\rt\,\CG\, Y\rt)\CG\,X\rth\,\CG\, Y\rth\\
    =&\Sigma(\Sigma(b, X\ro)\CG\,X\rt,  Y\rt)\CG\,X\rth\,\CG\, Y\rth =\Sigma(b\,X\ro,  Y\rt)\CG\,X\rt\,\CG\, Y\rth,
\end{align*}
where the 2nd step uses Lemma \ref{lem. sigma is 2 cocycle}.
Therefore, the formula in the statement is well defined. We now compute
\begin{align*}
    \Sigma(\Gamma(X\o, &Y\o), \Sigma(X\t\ro\,, Y\t\ro)\CG\,X\t\rt\,\CG\, Y\t\rt)\\
     =&\Sigma(\Gamma(X\o, Y\o), X\t{}_{+}Y\t{}_{+}\overline{\Gamma(Y\t{}_{-},X\t{}_{-})})\\
    =&\Gamma(\Gamma(X\o, Y\o)\, X\t{}_{++}Y\t{}_{++}\,,\Gamma(Y\t{}_{-},X\t{}_{-})\,Y\t{}_{+-}X\t{}_{+-}) =\varepsilon^{\Gamma}(X\CG\,Y),
\end{align*}
where the 1st step uses Corollary \ref{cor. comodule twist back}.
\end{proof}

\begin{corollary}\label{cor. comodule twist forward}
    Let $\cL$ be a left Hopf algebroid over $B$, $\Gamma$ an invertible left 2-cocycle on $\cL$ and $\Sigma$  defined as above. Then for any $M,N\in {}^{\cL}\CM$, we have
    \[\Sigma(\Gamma(m\mo, n\mo),\Sigma(m\z{}_{\scriptscriptstyle{[-2]}}, n\z{}_{\scriptscriptstyle{[-2]}})\CG\,m\z\rmo\CG\, n\z\rmo)\CG\,m\z\rz\ot_{\BG} n\z\rz=m\ot_{\BG}n,\]
    for all $m\in M$ and $n\in N$.
\end{corollary}

\begin{proof}
    The proof is similar to Corollary \ref{cor. comodule twist back},  using Lemma~\ref{prop. right inverse}.
\end{proof}

We now combine these results.

\begin{theorem}\label{thm. invertible 2 cocycle}
     Let $\cL$ be a left Hopf algebroid over $B$, $\Gamma$ an invertible left 2-cocycle on $\cL$ and $\Sigma$ as defined above. Then $\Sigma$ is an invertible left 2-cocycle on $\cL^{\Gamma}$ with $\Sigma^{\#-1}=\GH$. Moreover, $(\cL^{\Gamma})^{\Sigma}=\cL$. We call $\Sigma$ the inverse of $\Gamma$ and denote it by $\Gamma^{-1}$.
\end{theorem}

\begin{proof}
First,  $(B^{\Gamma})^{\Sigma}=B$, since $a\cdot_{\Sigma}b=\Sigma(a, b)=\Gamma(ab,1)=ab$. We next check that $\Sigma^{\#}$ is invertible with inverse $\GH$. Recall that
    \[\Sigma^{\#}(m\ot_{B}n)=\Sigma(m\rmo, n\rmo)\CG m\rz\ot_{\BG}n\rz,\]
    for all $m\in M\in {}^{\cL^{\Gamma}}\CM$ and $n\in N\in {}^{\cL^{\Gamma}}\CM$. Then, by Corollary \ref{cor. comodule twist forward},
    \begin{align*}
        \Sigma^{\#}&\circ \GH(m\ot_{\BG}n)\\
        =&\Sigma^{\#}(\Gamma(m\mo,n\mo)m\z\ot_{B} n\z)\\
        =&\Sigma(\Gamma(m\mo,n\mo)m\z\ro, n\z\ro)\CG\,m\z\rz\ot_{\BG} n\z\rz\\
        =&\Sigma(\Sigma(\Gamma(m\mo, n\mo),m\z{}_{\scriptscriptstyle{[-2]}})\CG\,m\z{}_{\scriptscriptstyle{[-1]}},  n\z\rmo)\CG\,m\z\rz\ot_{\BG} n\z\rz\\
        =&\Sigma(\Gamma(m\mo, n\mo),\Sigma(m\z{}_{\scriptscriptstyle{[-2]}}, n\z{}_{\scriptscriptstyle{[-2]}})\CG\,m\z\rmo\CG\, n\z\rmo)\CG\,m\z\rz\ot_{\BG} n\z\rz\\
        =&m\ot_{\BG}n,
    \end{align*}
    where the 3rd step uses Lemma~\ref{prop. right inverse}. Also, we have
    \begin{align*}
        \GH&\circ \Sigma^{\#}(n\ot_{B}m)\\
        =&\GH(\Sigma(m\rmo, n\rmo)\CG m\rz\ot_{\BG}n\rz)\\
         =&\GH(\Gamma(\Sigma(m\rmo, n\rmo), m\rz\mo)m\rz\z \ot_{\BG}n\rz)\\
         =&\Gamma(\Gamma(\Sigma(m\rmo, n\rmo), m\rz\mt)m\rz\mo,n\rz\mo)m\rz\z\ot_{B}n\rz\z\\
         =&\Gamma(\Sigma(m\rmo, n\rmo),\Gamma(m\rz\mt, n\rz\mt)m\rz\mo n\rz\mo)m\rz\z\ot_{B} n\rz\z =n\ot_{B}m,
    \end{align*}
    where the last step uses Corollary \ref{cor. comodule twist back}. Hence, $\Sigma$ is an invertible left 2-cocycle on $\cL^{\Gamma}$. Similarly to Lemma~\ref{prop. right inverse}, one can show that $Xb=\Sigma(X\ro, b)\CG X\rt$. Moreover, 
    \begin{align*}
        X_{\p}\CG\overline{\Sigma(X_{\m}, b)}=&X_{\p +}\overline{\Gamma(\Sigma(X_{\m}, b),X_{\p -})}\\
        =&X_{\p +}\overline{\Gamma(\Gamma(X_{\m +}b,X_{\m -}),X_{\p -})}\\
        =&X_{\p +}\overline{\Gamma(\Gamma(X_{\m +}\o\,b,X_{\m -}\o)X_{\m +}\t\,X_{\m -}\t,X_{\p -})}\\
        =&X_{\p +}\overline{\Gamma(X_{\m +}\,b, \Gamma(X_{\m -}\o, X_{\p -}\o)X_{\m -}\t\,X_{\p -}\t)}\\
        =&X_{\p ++}\overline{\Gamma(X_{\m ++}\,b, \Gamma(X_{\m -}, X_{\p -})X_{\m +-}\,X_{\p +-})}\\
        =&X_{++}\overline{\Gamma(X_{- +}\,b, X_{--}\,X_{+-})}\\
        =&X_{+}\overline{\Gamma(X_{-}\o{}_{+}\,b, X_{-}\o{}_{-}\,X_{-}\t)}\\
        =&X_{+}\overline{\Gamma(X_{-}\,b, 1)}=X_{+}\overline{\varepsilon(X_{-}b)}=\overline{b}X.
    \end{align*}
    Similarly, $ X\overline{b}=X_{\p}\CG\overline{\Sigma(b,X_{\m})}$. Now, we show that $\Sigma$ cotwists $\cL^{\Gamma}$ to recover the $B^e$-ring structure on $\cL$,
    \begin{align*}
        \Sigma(X_{\p}\ro, &Y_{\p}\ro)\CG X_{\p}\rt\CG Y_{\p}\rt\CG\, \overline{\Sigma(Y{}_{\m}, X{}_{\m})}\\
        =&(X_{\p +}\,Y_{\p +}\overline{\Gamma(Y_{\p -},\,X_{\p -})})\CG\, \overline{\Sigma(Y{}_{\m}, X{}_{\m})}\\
        =&X_{\p ++}\,Y_{\p ++}\overline{\Gamma(\Sigma(Y{}_{\m}, X{}_{\m}),\Gamma(Y_{\p -},\,X_{\p -})Y_{\p +-}\,X_{\p +-})}\\
        =&X_{\p +}\,Y_{\p +}\overline{\Gamma(\Sigma(Y{}_{\m}, X{}_{\m}),\Gamma(Y_{\p -}\o,\,X_{\p -}\o)Y_{\p -}\t\,X_{\p -}\t)} =XY,
    \end{align*}
    where the 1st and the last steps use Corollary \ref{cor. comodule twist back}. As $\GH$ is the inverse of $\Sigma^{\#}$, we have $(\Delta^{\Gamma})^\Sigma=\Delta$. Also,
    \begin{align*}
        \Sigma(X_{\p}, X_{\m})=&\Gamma(X_{\p +} X_{\m +}, \Gamma(X_{\m -}\o, X_{\p -}\o)X_{\m -}\t X_{\p -}\t)\\
        =&\Gamma(X_{\p ++} X_{\m ++}, \Gamma(X_{\m -}, X_{\p -})X_{\m +-} X_{\p +-})\\
        =&\Gamma(X_{++} X_{-+}, X_{--} X_{+-})=\Gamma(X_{+} X_{-},1)=\varepsilon(X_{+}X_{-})=\varepsilon(X)
    \end{align*}
    so that $(\varepsilon^{\Gamma})^{\Sigma}=\varepsilon$.
\end{proof}

Here we also give a property for later use:
\begin{proposition}\label{prop. twist forward}
      Let $\cL$ be a left Hopf algebroid over $B$ and $\Gamma$ an invertible left 2-cocycle on $\cL$. Then $(\Gamma^{-1})^{-1}=\Gamma$,
and
    \[\Gamma(X\o, Y\o)\,X\t\,Y\t =X_{\p}\CG\,Y_{\p}\CG\overline{\Gamma^{-1}(Y_{\m},X_{\m})},\]
    for all $X, Y\in\CL$.
\end{proposition}
\begin{proof}
 As $\Gamma^{-1}$ is an invertible left 2-cocycle on $\cL^{\Gamma}$ by Theorem \ref{thm. invertible 2 cocycle}, we have for all $X, Y\in \cL$,
\begin{align*}
    (&\Gamma^{-1})^{-1}(X, Y)\\
    =&\Gamma^{-1}(X_{\p}\CG Y_{\p}, \Gamma^{-1}(Y_{\m}\ro, X_{\m}\ro)\CG\,Y_{\m}\rt\CG X_{\m}\rt)\\
    =&\Gamma^{-1}(\Gamma(X_{\p +}\o, Y_{\p +}\o)X_{\p +}\t Y_{\p +}\t\,\overline{\Gamma(Y_{\p -},X_{\p -})}, Y_{\m +}\, X_{\m +}\overline{\Gamma(X_{\m -}\,, Y_{\m -})})\\
     =&\Gamma(\Gamma(X_{\p +}\o, Y_{\p +}\o)X_{\p +}\t{}_{+} Y_{\p +}\t{}_{+}\,Y_{\m ++}X_{\m ++},\\
     &\Gamma(\Gamma(X_{\m -}\, ,Y_{\m -})X_{\m +-}\o\,Y_{\m +-}\o, \Gamma(Y_{\p -},X_{\p -})Y_{\p +}\t{}_{-}\o\,X_{\p +}\t{}_{-}\o)\\
     &X_{\m +-}\t\,Y_{\m +-}\t\,Y_{\p +}\t{}_{-}\t\,X_{\p +}\t{}_{-}\t)\\
   =&\Gamma(\Gamma(X_{\p ++}\o, Y_{\p ++}\o)X_{\p ++}\t Y_{\p ++}\t\,Y_{\m ++}X_{\m ++},\\
     &\Gamma(\Gamma(X_{\m -},\, Y_{\m -})X_{\m +-}\o\,Y_{\m +-}\o, \Gamma(Y_{\p -},X_{\p -})Y_{\p +-}\o\,X_{\p +-}\o)\,
     X_{\m +-}\t\,Y_{\m +-}\t\,Y_{\p +-}\t\,X_{\p +-}\t)\\
      =&\Gamma(\Gamma(X_{\p +}\o, Y_{\p +}\o)X_{\p +}\t Y_{\p +}\t\,Y_{\m +}X_{\m +},\\
     &\Gamma(\Gamma(X_{\m -}\o\,, Y_{\m -}\o)X_{\m -}\t\,Y_{\m -}\t, \Gamma(Y_{\p -}\o,X_{\p -}\o)Y_{\p -}\t\,X_{\p -}\t)\,
     X_{\m -}\th\,Y_{\m -}\th\,Y_{\p -}\th\,X_{\p -}\th)\\
        =&\Gamma(\Gamma(X_{\p +}\o,\,Y_{\p +}\o)X_{\p +}\t\,Y_{\p +}\t
\,Y{}_{\m +}\, X{}_{\m +}, \\
        &\Gamma(X{}_{\m -}\o, \Gamma(Y{}_{\m -}\o, \Gamma(Y_{\p -}\o,\,X_{\p -}\o)Y_{\p -}\t\,X_{\p -}\t) Y{}_{\m -}\t\,Y_{\p -}\th\,X_{\p -}\th )X{}_{\m -}\t\, Y{}_{\m -}\th\,Y_{\p -}\fo\,X_{\p -}\fo)\\
        =&\Gamma(\Gamma(X_{\p +}\o,\,Y_{\p +}\o)X_{\p +}\t\,Y_{\p +}\t
\,Y{}_{\m +}\, X{}_{\m +}, \\
        &\Gamma(X{}_{\m -}\o, \Gamma(\Gamma(Y{}_{\m -}\o,Y_{\p -}\o)Y{}_{\m -}\t\,Y_{\p -}\t, \,X_{\p -}\o) Y{}_{\m -}\th\,Y_{\p -}\th\,X_{\p -}\t )X{}_{\m -}\t\, Y{}_{\m -}\fo\,Y_{\p -}\fo\,X_{\p -}\th)\\    =&\Gamma(\Gamma(X_{\p ++}\o,\,Y_{\p ++}\o)X_{\p ++}\t\,Y_{\p ++}\t
\,Y{}_{\m ++}\, X{}_{\m ++}, \\
&\Gamma(X{}_{\m -}, \Gamma(\Gamma(Y{}_{\m -},Y_{\p -})Y{}_{\m +-}\o\,Y_{\p +-}\o, \,X_{\p -}) Y{}_{\m +-}\t\,Y_{\p +-}\t\,X_{\p +-}\o )X{}_{\m +-}\, Y{}_{\m +-}\th\,Y_{\p +-}\th\,X_{\p +-}\t)\\
       =&\Gamma(\Gamma(X_{\p ++}\o,\,Y_{++}\o)X_{\p ++}\t\,Y_{++}\t
\,Y{}_{-+}\, X{}_{\m ++}, \\
        &\Gamma(X{}_{\m -}, \Gamma(Y{}_{--}\o\,Y_{+-}\o, \,X_{\p -}) Y{}_{--}\t\,Y_{ +-}\t\,X_{\p +-}\o )X{}_{\m +-}\, Y{}_{--}\th\,Y_{ +-}\th\,X_{\p +-}\t)\\
        =&\Gamma(\Gamma(X_{\p ++}\o\,,Y_{+}\o)X_{\p ++}\t\,Y_{+}\t
\,Y{}_{-}\, X{}_{\m ++}, \Gamma(X{}_{\m -}, \Gamma(1, \,X_{\p -}) X_{\p +-}\o )X{}_{\m +-}\, X_{\p +-}\t)\\
=&\Gamma(\Gamma(X_{\p ++}\o,\,Y)X_{\p ++}\t\, X{}_{\m ++}, \Gamma(X{}_{\m -},  X_{\p -} )X{}_{\m +-}\, X_{\p +-})\\
=&\Gamma(\Gamma(X_{ ++}\o,\,Y)X_{++}\t\, X{}_{-+}, X{}_{--}\, X_{+-})\\
=&\Gamma(\Gamma(X_{ +}\o,\,Y)X_{+}\t\, X{}_{-}, 1)\\
=&\Gamma(\Gamma(X,Y), 1)=\Gamma(X,Y),
\end{align*}
    where the 2nd step uses Corollary \ref{cor. comodule twist back}. The stated formula is a direct result of the second equality in Corollary \ref{cor. comodule twist back} by exchanging $\Sigma$ and $\Gamma$.
\end{proof}

In summary, we see although there is in general no convolution inverse for an invertible left 2-cocycle $\Gamma$ in our sense, there is an invertible left 2-cocycle $\Gamma^{-1}$, such that
\begin{align*}
    \Gamma(\Gamma^{-1}(X\ro, Y\ro), \Gamma(X\rt\o\, ,Y\rt\o)X\rt\t\, Y\rt\t)=&\varepsilon(XY),\\
    \Gamma^{-1}(\Gamma(X\o, Y\o), \Gamma^{-1}(X\t\ro\,, Y\t\ro)\CG\,X\t\rt\,\CG\, Y\t\rt)=&\varepsilon^{\Gamma}(X\CG\,Y),
\end{align*}
    for all $X, Y\in \cL$. This motivates us to define the following composition of two 2-cocycles.

\begin{lemma}\label{compos}
    Let $\cL$ be a left $B$-Hopf algebroid, $\Gamma$ an invertible left 2-cocycle on $\cL$ and $\Sigma$ an invertible left 2-cocycle on $\cL^{\Gamma}$. Then $\Sigma\circ \Gamma:\cL\ot_{\BB}\cL\to B$ given by
    \begin{align*}
        \Sigma\circ \Gamma(X, Y)=&\Gamma(\Sigma(X\ro, Y\ro)\CG X\rt, Y\rt)
    \end{align*}
    is an invertible left 2-cocycle on $\CL$. Moreover, $(\Sigma\circ \Gamma)^{\#}=\GH\circ \Sigma^{\#}$.
\end{lemma}
\begin{proof}
    It is not hard to see $\Sigma\circ \Gamma$ is well defined and  factors through $\ot_{\overline{B}}$, and one can check directly it is left $\overline{B}$-linear. We show that is is a left 2-cocycle. First, we observe that
    \begin{align*}
        \Gamma(\Sigma(X\o\ro,& Y\o\ro)\CG X\o\rt, Y\o\rt)X\t Y\t\\
        =& \Gamma(\Sigma(X\ro, Y\ro)\CG X\rt\o, Y\rt\o)X\rt\t Y\rt\t\\
        =& \Gamma(\Sigma(X\ro, Y\ro), \Gamma( X\rt\o, Y\rt\o)X\rt\t Y\rt\t)X\rt\th Y\rt\th\\
        =&\Sigma(X\ro, Y\ro)\CG (\Gamma( X\rt\o, Y\rt\o)X\rt\t Y\rt\t)\\
        =&\Sigma(X\ro, Y\ro)\CG (X\rt{}_{\p}\CG\,Y\rt{}_{\p}\,\CG\overline{\Gamma^{-1}(Y\rt{}_{\m},X\rt{}_{\m})} ),
    \end{align*}
    where the last step uses Proposition \ref{prop. twist forward}.   Then on the one hand,
    \begin{align*}
        \Sigma&\circ \Gamma(\Sigma\circ \Gamma(X\o, Y\o)X\t Y\t, Z)\\
        =&\Sigma\circ \Gamma(\Gamma(\Sigma(X\o\ro, Y\o\ro)\CG X\o\rt, Y\o\rt)X\t Y\t, Z)\\
         =&\Sigma\circ \Gamma(\Sigma(X\ro, Y\ro)\CG (X\rt{}_{\p}\CG\,Y\rt{}_{\p}\,\CG\overline{\Gamma^{-1}(Y\rt{}_{\m},X\rt{}_{\m})} ), Z)\\
         =&\Gamma(\Sigma(\big(\Sigma(X\ro, Y\ro)\CG (X\rt{}_{\p}\CG\,Y\rt{}_{\p}\,\CG\overline{\Gamma^{-1}(Y\rt{}_{\m},X\rt{}_{\m})} )\big)\ro, \\
         &\quad Z\ro)\CG \big(\Sigma(X\ro, Y\ro)\CG (X\rt{}_{\p}\CG\,Y\rt{}_{\p}\,\CG\overline{\Gamma^{-1}(Y\rt{}_{\m},X\rt{}_{\m})} )\big)\rt, Z\rt)\\
         =&\Gamma(\Sigma\big(\Sigma(X\ro, Y\ro)\CG X\rt{}_{\p}\ro\CG\,Y\rt{}_{\p}\ro\, ,
         Z\ro\big)\CG \big(X\rt{}_{\p}\rt\CG\,Y\rt{}_{\p}\rt\,\CG\overline{\Gamma^{-1}(Y\rt{}_{\m},X\rt{}_{\m})} \big), Z\rt)\\
         =&\Gamma(\Sigma\big(\Sigma(X\ro, Y\ro)\CG X\rt\CG\,Y\rt\, ,
         Z\ro\big)\CG \big(X\rth{}_{\p}\CG\,Y\rth{}_{\p}\,\CG\overline{\Gamma^{-1}(Y\rth{}_{\m},X\rth{}_{\m})} \big), Z\rt)\\
         =&\Gamma(\Sigma\big(\Sigma(X\ro, Y\ro)\CG X\rt\CG\,Y\rt\, ,
         Z\ro\big)\CG \big(\Gamma(X\rth\o, Y\rth\o)X\rth\t\, Y\rth\t\big), Z\rt)\\
         =&\Gamma(\Sigma(\Sigma(X\ro, Y\ro)\CG X\rt\CG\,Y\rt\, ,
         Z\ro), \Gamma( \Gamma(X\rth\o, Y\rth\o)X\rth\t\, Y\rth\t, Z\rt\o)X\rth\th\, Y\rth\th\,Z\rt\t).
    \end{align*}
   On the other hand, by a similar method, one has 
    \begin{align*}
         \Sigma&\circ \Gamma(X, \Sigma\circ \Gamma(Y\o, Z\o)Y\t Z\t)\\
         =&\Gamma(\Sigma(         X\ro,\Sigma(Y\ro, Z\ro)\CG Y\rt\CG\,Z\rt\,), \Gamma(X\rt\o, \Gamma(Y\rth\o,Z\rth\o) Y\rth\t\,Z\rth\t)X\rt\t\, Y\rth\th\,Z\rth\th).
    \end{align*}
    The two are equal since $\Gamma$ and $\Sigma$ are left 2-cocycles. Finally,
   have for all $m\in M\in {}^{\cL}\CM$ and $n\in N\in {}^{\cL}\CM$ that 
    \begin{align*}
         (\Sigma\circ \Gamma)^{\#}(m\ot n)=&\Sigma\circ \Gamma(m\mo, n\mo)m\z\ot n\z\\
         =&\Gamma(\Sigma(m\mo\ro, n\mo\ro)\CG m\mo\rt, n\mo\rt)m\z\ot n\z\\
         =&\Gamma(\Sigma(m\rmo, n\rmo)\CG m\rz\mo, n\rz\mo)m\rz\z\ot n\rz\z\\
         =&\GH(\Sigma(m\rmo, n\rmo)\CG m\rz\ot n\rz) =\GH\circ \Sigma^{\#}(m\ot n),
    \end{align*}
    where the 3rd step uses Corollary \ref{cor. cocommute coaction}. Hence, $\Sigma\circ \Gamma$ is an invertible left 2-cocycle since $\Sigma$ and $\Gamma$ are invertible.
    \end{proof}
    \begin{proposition}
        Let $\cL$ be a left $B$-Hopf algebroid, $\Gamma$ an invertible 2-cocycle on $\cL$ and $\Sigma$ an invertible 2-cocycle on $\cL^{\Gamma}$. Then $\cL^{\Sigma\circ \Gamma}=(\cL^{\Gamma})^{\Sigma}$.
    \end{proposition}
    \begin{proof}
        We have $B^{\Sigma\circ \Gamma}=B^{\Sigma}$ since $b\cdot_{\Sigma\circ \Gamma}b'=\Gamma(\Sigma(b,b'),1)=\Sigma(b,b')$. Now let $B^{i}=B, B^{\Gamma}$ or $B^{\Sigma}$ and define 
        \[ \GH_{L}:\cL\ot_{\BG\ot \overline{B^{i}}}\cL\to \cL\ot_{B\ot \overline{B^{i}}}\cL,\quad \GH_{L}(X\ot Y)=\Gamma(X\o, Y\o)X\t\ot Y\t,\]
        \[  \GH_{R}:\cL\ot_{B^{i}\ot \overline{\BG}}\cL\to \cL\ot_{B^i\ot \overline{B}}\cL,\quad \GH_{R}(X\ot Y)=X_{+}\ot Y_{+}\overline{\Gamma(Y_{-},X_{-})}.\]
     One can check that $m_{\cL^{\Gamma}}=m_{\cL}\circ \GH_{L}\circ \GH_{R}=m_{\cL}\circ \GH_{R}\circ \GH_{L}$. Moreover,  $\GH{}_{R}\circ \Sigma^{\#}{}_{L}=\Sigma^{\#}{}_{L}\circ \GH{}_{R}:\cL_{B^{\Sigma}\ot\overline{B^{\Gamma}}}\to \cL_{B^{\Gamma}\ot\overline{B}}$ since
       \begin{align*}
           \Sigma^{\#}{}_{L}\circ \GH{}_{R}(X\ot Y)=& \Sigma^{\#}{}_{L}(X_{+}\ot Y_{+}\overline{\Gamma(Y_{-}, X_{-})})\\
           =&\Sigma(X_{+}\ro, Y_{+}\ro)\CG X_{+}\rt\ot Y_{+}\rt\overline{\Gamma(Y_{-}, X_{-})}\\
           =&\Sigma(X\ro, Y\ro)\CG X\rt{}_{+}\ot Y\rt{}_{+}\overline{\Gamma(Y\rt{}_{-}, X\rt{}_{-})}\\
           =&\Gamma^{\#}{}_{R}(\Sigma(X\ro, Y\ro)\CG X\rt\ot Y\rt)=\GH{}_{R}\circ \Sigma^{\#}{}_{L}(X\ot Y),
       \end{align*}
       where the 3rd step uses Proposition \ref{prop. cocomute coaction}. As a result,
\begin{align*}
    m_{(\cL^{\Gamma})^{\Sigma}}=&m_{\cL^{\Gamma}}\circ \Sigma^{\#}{}_{L}\circ \Sigma^{\#}{}_{R}=m_{\cL}\circ \Gamma^{\#}{}_{L}\circ \Gamma^{\#}{}_{R}\circ \Sigma^{\#}{}_{L}\circ \Sigma^{\#}{}_{R}\\
    =&m_{\cL}\circ \Gamma^{\#}{}_{L}\circ \Sigma^{\#}{}_{L}\circ \Gamma^{\#}{}_{R}\circ \Sigma^{\#}{}_{R}= m_{\cL}\circ (\Sigma\circ\Gamma)^{\#}{}_{L}\circ (\Sigma\circ\Gamma)^{\#}{}_{R}=m_{\cL^{(\Sigma\circ\Gamma)}},
\end{align*}
       where the last step uses $(\Sigma\circ\Gamma)^{\#}{}_{L/R}=\Gamma^{\#}{}_{L/R}\circ \Sigma^{\#}{}_{L/R}$ (the proof of which is similar to  that of $(\Sigma\circ\Gamma)^{\#}=\Gamma^{\#}\circ \Sigma^{\#}$). Clearly, $\Delta^{\Sigma\circ \Gamma}=(\Sigma\circ \Gamma)^{\# -1}\circ \Delta=\Sigma^{\# -1}\circ \GH^{-1}\circ \Delta=\Sigma^{\# -1}\circ \Delta^{\Gamma}=(\Delta^{\Gamma})^{\Sigma}$. Finally, we have
       \begin{align*}
           \varepsilon^{\Sigma\circ \Gamma}(X)=&\Sigma\circ \Gamma(X_{+}, X_{-})
           =\Gamma(\Sigma(X_{+}\ro, X_{-}\ro)\CG X_{+}\rt, X_{-}\rt)\\
           =&\Gamma(\Sigma(X_{\p +}\ro, X_{\m})\CG X_{\p +}\rt, X_{\p -})
           =\Gamma(\Sigma(X_{\p}\ro, X_{\m})\CG X_{\p}\rt{}_{+}, X_{\p}\rt{}_{-})\\
           =&\varepsilon^{\Gamma}(\Sigma(X_{\p}\ro, X_{\m})\CG X_{\p}\rt)
           =\Sigma(X_{\p}\ro, X_{\m})\CG\varepsilon^{\Gamma}( X_{\p}\rt)=\Sigma(X_{\p}, X_{\m})=(\varepsilon^{\Gamma})^{\Sigma}(X),
       \end{align*}
        where the 3rd and 4th steps use Proposition \ref{prop. cocomute coaction}.
    \end{proof}

Putting all this together, we arrive at our main theorem of the section:

\begin{theorem}
    The collection of left Hopf algebroids with invertible 2-cocycles with composition as in Lemma~\ref{compos} form a groupoid.
\end{theorem}

\begin{proof}
     We first observe that $\hat{\varepsilon^{\Gamma}}\circ \Gamma=\Gamma$ and $\Gamma\circ \hat{\varepsilon}=\Gamma$. It is not hard to see $\Gamma\circ \Gamma^{-1}$ is the trivial 2-cocycle on $\cL^{\Gamma}$ and $\Gamma^{-1}\circ \Gamma$ is the trivial 2-cocycle on $\cL$ by Lemma~\ref{prop. left inverse} and Lemma~\ref{prop. right inverse}. Next, we show that the composition is associative. Let $\Pi$ be an invertible left 2-cocycle on $\cL^{(\Sigma\circ \Gamma)}$ (with the coproduct denoted by $\Delta^{(\Sigma\circ \Gamma)}(X)=X_{\<1\>}\ot X_{\<2\>}$).  Then
    \begin{align*}
        \Pi\circ(\Sigma\circ \Gamma)(X,Y)=&\Sigma\circ \Gamma(\Pi(X_{\<1\>}, Y_{\<1\>})\cdot_{\Pi}X_{\<2\>}, Y_{\<2\>})\\
        =&\Gamma(\Sigma(\Pi(X_{\<1\>}, Y_{\<1\>})\cdot_{\Pi}X_{\<2\>}\ro, Y_{\<2\>}\ro)\CG X_{\<2\>}\rt, Y_{\<2\>}\rt)\\
        =&\Gamma(\Sigma(\Pi(X\ro{}_{\<1\>}, Y\ro{}_{\<1\>})\cdot_{\Pi}X\ro{}_{\<2\>}, Y\ro{}_{\<2\>})\CG X\rt, Y\rt)\\
        =&\Gamma((\Pi\circ\Sigma)(X\ro, Y\ro)\CG X\rt, Y\rt) =(\Pi\circ\Sigma)\circ \Gamma(X,Y),
    \end{align*}
    where the 3rd step uses Proposition \ref{prop. cocomute coaction}.
\end{proof}

\subsection{Dualisation of 2-cocycles}\label{secdualcocy}

In this section, we will see that a 2-cocycle in a left bialgebroid in the sense of \cite{Xu} induces a 2-cocycle on its dual bialgebroid in our sense. 

\begin{definition}\cite{schau1}\label{def. dual pairing for left bialgeboids}
    Let $\cL$ and $\CH$ be two left bialgebroids over $B$. A dual pairing between $\cL$ and $\CH$ is a linear map $\<\  |\  \> :\cL\ot\CH\to B$ such that:
    \begin{itemize}
        \item [(1)]$\<a\Bar{b}X c\Bar{d}| \alpha\>f=a\<X|c\Bar{f}\alpha d\Bar{b}\>$;
        \item [(2)] $\<X|\alpha\beta\>=\<X\o|\alpha \overline{\<X\t|\beta\>}\>=\<\overline{\<X\t|\beta\>}X\o|\alpha \>$;
        \item [(3)] $\<XY|\alpha\>=\<X \<Y|\alpha\o\>|\alpha\t\>=\<X |\<Y|\alpha\o\>\alpha\t\>$;
        \item [(4)] $\<X|1\>=\varepsilon(X)$;
        \item [(5)]$\<1|\alpha\>=\varepsilon(\alpha)$,
    \end{itemize}
    for all $a, b, c, d, f\in B$, $\alpha, \beta\in \CH$ and $X, Y\in \cL$.
\end{definition}
Let $\cL$ be a left $B$-bialgebroid that is finitely generated as a left $\BB$-module. It is known from \cite{schau1} that its left dual $\cL^{\vee}:=\Hom_{\overline{B}-}(\cL, B)$ is a left bialgebroid. There is a canonical dual pairing between $\cL^{\vee}$ and $\cL$ given by
\[\<\alpha|X\>:=\alpha(X),\]
for all $\alpha\in \cL^{\vee}$ and $X\in \cL$. The left $B$-bialgebroid structure on $\cL^{\vee}$ is determined by (1)-(5) in Definition \ref{def. dual pairing for left bialgeboids}.

\begin{definition}\cite{Xu}
    Let $\Lambda$ be a left $B$-bialgebroid. Then $F\in \Lambda\di_{B}\Lambda$ is called a left 2-cocycle in $\Lambda$ if
    \begin{itemize}
        \item [(1)] $(\varepsilon\di_{B}\id)F=1_{\Lambda}=(\id\di_{B}\varepsilon)F$;
        \item[(2)] $(\Delta\di_{B}\id)FF^{12}=(\id\di_{B}\Delta)FF^{23}$,
    \end{itemize}
    where $F^{12}=F\ot 1\in \Lambda\di_{B}\Lambda\ot \Lambda$ and $F^{23}=1\ot F\in \Lambda\ot\Lambda\di_{B} \Lambda$. In the following, we will use the notation $F=F^\alpha \di_{B}F_{\alpha}\in \Lambda\di_{B} \Lambda$. We call $F$ an invertible left 2-cocycle in $\Lambda$, if for all left $\Lambda$-module $M, N$, the map
\[F^{\#}:M\di_{B^F}N\to M\di_{B}N,\qquad m\ot n\mapsto F^{\alpha}\la m\di_{B}F_{\alpha}\la n,\forall m\in M, n\in N\]
is invertible.
\end{definition}
By \cite{Xu}, given a 2-cocycle in a left $B$-bialgebroid $\Lambda$, we can construct a new left $B^F$-bialgebroid, with a twisted base algebra $B^F$ defined on the underlying vector space $B$ with a twisted product
\begin{equation}\label{starF}{a\cdot_{F}b=\varepsilon(F^\alpha a)\varepsilon(F_{\alpha}b)},\end{equation}
source and target maps
\[s^F(b)=\varepsilon(F^\alpha b)F_{\alpha}, \qquad t^{F}(b)=\overline{\varepsilon(F_{\alpha}b)}F^\alpha,\]
and coproduct
\[\Delta^{F}(\alpha)=F^{\#}{}^{-1}(\alpha\o F^{\alpha}\di_{B}\alpha\t F_{\alpha}),\]
for all $\alpha\in \Lambda$. The counit, unit and product of $\Lambda^F$ are the same as $\Lambda$.

\begin{lemma}\label{lem. 2-cocycle dual}
 Suppose we are given a dual pairing  $\<\ |\ \>$ between two left $B$-bialgebroids $\Lambda$, $\cL$. If $F$ is an  invertible left 2-cocycle in $\Lambda$ then there is an invertible left 2-cocycle on $\cL$ given by
$\Gamma_{F}(X\ot_{\overline{B}} Y)=\<F^{\alpha}|X\overline{\<F_{\alpha}|Y\>}\>$ for all $X, Y\in \cL$.
\end{lemma}
\begin{proof}

It is not hard to see that $\Gamma_{F}$ is well defined and left $\BB$-linear. That it is a  left 2-cocycle is, on the one hand,
\begin{align*}
    \Gamma_{F}(\Gamma_{F}(X\o,& Y\o)X\t Y\t, Z)\\
    =&\Gamma_{F}(\<F^{\alpha}|X\o\overline{\<F_{\alpha}|Y\o\>}\>X\t Y\t, Z)\\
    =&\<F^{\beta}|\<F^{\alpha}|X\o\overline{\<F_{\alpha}|Y\o\>}\>X\t Y\t\overline{\<F_{\beta}|Z\>}\>\\
    =&\<F^{\beta}|\<F^{\alpha}|X\o\>X\t \<F_{\alpha}|Y\o\> Y\t\overline{\<F_{\beta}|Z\>}\>\\
    =&\<F^{\beta}\o|\<F^{\alpha}|X\o\>X\t \overline{\<F^{\beta}\t| \<F_{\alpha}|Y\o\> Y\t\overline{\<F_{\beta}|Z\>}\>}\>\\
    =&\<F^{\beta}\o F^{\alpha}|X \overline{\<F^{\beta}\t F_{\alpha}|Y\overline{\<F_{\beta}|Z\>}\>}\>.
\end{align*}
On the other hand, we can similarly get
\begin{align*}
     \Gamma_{F}(X, \Gamma_{F}(Y\o, Z\o)Y\t Z\t)=\<F^{\alpha}|X \overline{\<F_{\alpha}\o F^{\beta}|Y\overline{\<F_{\alpha}\t F_{\beta}|Z\>}\>}\>.
\end{align*}
The two are equal by the 2-cocycle condition of $F$. To see that $\Gamma_{F}^{\#}$ is invertible, assume $M, N$ are left $\cL$-comodules. It is known by \cite{schau1} that $M, N$ are also left $\Lambda$-modules with left action
\[\alpha\la m=\<\alpha|m\mo\>m\z,\]
for all $m\in M$. We have
\begin{align*}
    \Gamma_{F}^{\#}(m\ot_{B}n)=&\Gamma_{F}(m\mo\ot n\mo)m\z\ot n\z=\<F^{\alpha}|m\mo\overline{\<F_{\alpha}|n\mo\>}\>m\z\ot n\z\\
    =&\<F^{\alpha}|m\mo\>m\z\ot \<F_{\alpha}|n\mo\>n\z=F^{\alpha}\la m\ot F_{\alpha}\la n,
\end{align*}
for all $m\in M$ and $n\in N$. Therefore, $\Gamma_{F}^{\#}$ is invertible since $F^{\#}$ is.
\end{proof}

\begin{theorem}\label{thm. twisted dual pairing}
    Suppose we are given a dual pairing $\<\ |\ \>$ between two left $B$-bialgebroids $\Lambda$, $\cL$ and that $\cL$ is a left Hopf algebroid. If $F$ is an invertible left 2-cocycle in $\Lambda$ then there is a  dual pairing between $\Lambda^{F}$ and $\cL^{\Gamma_{F}}$  given by
    \[[\alpha|X]=\<F^{\alpha}\alpha|X_{+}\overline{\<F_{\alpha}|X_{-}\>}\>,\qquad \forall X\in \cL, \alpha\in \Lambda.\]
\end{theorem}
\begin{proof}
  First, we observe that
 \[[\alpha|X]=\Gamma_{F}(\<\alpha|X\o\> X\t{}_{+}, X\t{}_{-}).\]
 We denote $\Gamma_{F}$ for brevity by $\Gamma$ in the following. We have
 \begin{align*}
     [s^{F}(b)\alpha|X]=&\Gamma(\<\varepsilon(F^{\alpha}b)F_{\alpha}\alpha|X\o\> X\t{}_{+}, X\t{}_{-})\\
     =&\Gamma(\<\<F^{\alpha}|b\>F_{\alpha}|\<\alpha|X\o\>X\t\> X\th{}_{+}, X\th{}_{-})\\
     =&\Gamma(\<F^{\alpha}|b\>\<F_{\alpha}|\<\alpha|X\o\>X\t\> X\th{}_{+}, X\th{}_{-})\\
     =&\Gamma(\<F^{\alpha}|b\overline{\<F_{\alpha}|\<\alpha|X\o\>X\t\>}\> X\th{}_{+}, X\th{}_{-})\\
     =&\Gamma(\Gamma(b, \<\alpha|X\o\>X\t) X\th{}_{+}, X\th{}_{-})\\
     =&\Gamma(b, \Gamma(\<\alpha|X\o\> X\t{}_{+}\o, X\t{}_{-}\o)X\t{}_{+}\t\, X\t{}_{-}\t)\\
     =&\Gamma(b, \Gamma(\<\alpha|X\o\> X\t{}_{+}, X\t{}_{-})) =b\CG[\alpha|X].
 \end{align*}
 Also,
\begin{align*}
     [\alpha|\overline{b}\CG X]=&\Gamma(\<\alpha|X_{+}\o\> X_{+}\t{}_{+}, \Gamma(X_{-}, b)X_{+}\t{}_{-})\\
     =&\Gamma(\<\alpha|X\o\> X\t{}_{++}, \Gamma(X\t{}_{-}, b)X\t{}_{+-})\\
     =&\Gamma(\<\alpha|X\o\> X\t{}_{+}, \Gamma(X\t{}_{-}\o, b)X\t{}_{-}\t)\\
     =&\Gamma(\Gamma(\<\alpha|X\o\> X\t{}_{+}\o, X\t{}_{-}\o)X\t{}_{+}\t X\t{}_{-}\t, b)\\
     =&\Gamma(\Gamma(\<\alpha|X\o\> X\t{}_{+}, X\t{}_{-}), b)= [\alpha|X]\CG b.
\end{align*}
Moreover, 
\begin{align*}
    [\alpha s^{F}(b)|X]=&\Gamma(\<\alpha\<F^{\alpha}|b\>F_{\alpha}|X\o\> X\t{}_{+}, X\t{}_{-})\\
    =&\Gamma(\<\alpha|\<\<F^{\alpha}|b\>F_{\alpha}|X\o\>X\t \> X\th{}_{+}, X\th{}_{-})\\
     =&\Gamma(\<\alpha|\<F^{\alpha}|b\>\,\<F_{\alpha}|X\o\>X\t \> X\th{}_{+}, X\th{}_{-})\\
      =&\Gamma(\<\alpha|\<F^{\alpha}|b\overline{\<F_{\alpha}|X\o\>}\>\,X\t \> X\th{}_{+}, X\th{}_{-})\\
    =&\Gamma(\<\alpha|\Gamma(b, X\o)X\t \> X\th{}_{+}, X\th{}_{-})=[\alpha |b\CG X].
\end{align*}
One can similarly show that $[t^{F}(a)\alpha t^{F}(b)|X]=[\alpha|X\CG b \CG\overline{a}]$. Next, we show that
 \begin{align*}
     [\alpha&|[\beta|X\ro]\CG X\rt]\\
     =&[\alpha|\Gamma(\<\beta|X\ro\o\> X\ro\t{}_{+}, X\ro\t{}_{-})\CG X\rt]\\
     =&[\alpha|\Gamma(\Gamma(\<\beta|X\ro\o\> X\ro\t{}_{+}, X\ro\t{}_{-}), X\rt\o)X\rt\t]\\
     =&[\alpha|\Gamma(\Gamma(\<\beta|X\ro\o\> X\ro\t{}_{+}\o, X\ro\t{}_{-}\o)X\ro\t{}_{+}\t X\ro\t{}_{-}\t, X\rt\o)X\rt\t]\\
     =&[\alpha|\Gamma(\<\beta|X\ro\o\> X\ro\t{}_{+}, \Gamma( X\ro\t{}_{-}\o, X\rt\o) X\ro\t{}_{-}\t X\rt\t)X\rt\th]\\
     =&[\alpha|\Gamma(\<\beta|X\ro\o\> X\ro\t{}_{++}, \Gamma( X\ro\t{}_{-}, X\rt\o) X\ro\t{}_{+-} X\rt\t)X\rt\th]\\
     =&[\alpha|\Gamma(\<\beta|X\ro{}_{+}\o\> X\ro{}_{+}\t{}_{+}, \Gamma( X\ro{}_{-}, X\rt\o) X\ro{}_{+}\t{}_{-} X\rt\t)X\rt\th]\\
     =&[\alpha|\Gamma(\<\beta|X\o\> X\t{}_{+},  X\t{}_{-} X\th)X\fo]\\
     =&[\alpha|\<\beta|X\o\>  X\t]\\
     =&\<F^{\alpha}\alpha|\<\beta|X\o\>  X\t{}_{+}\overline{\<F_{\alpha}|X\t{}_{-}\>}\>\\
     =&\<F^{\alpha}\alpha|\<\beta|X{}_{+}\o\>  X{}_{+}\t\overline{\<F_{\alpha}|X{}_{-}\>}\>\\
      =&\<F^{\alpha}\alpha\beta|X{}_{+}\overline{\<F_{\alpha}|X{}_{-}\>}\>=[\alpha\beta|X].
 \end{align*}
  where the 7th step uses (\ref{equ. relation between old and new coproduct}). Moreover, we see on the one hand
  \begin{align*}
      [\alpha&|X\CG Y]\\
      =&[\alpha|\Gamma(X\o, Y\o)X\t{}_{+} Y\t{}_{+}\overline{\Gamma(Y\t{}_{-}, X\t{}_{-})}]\\
        =&\Gamma(\<\alpha|\Gamma(X\o, Y\o)X\t{}_{+}\o Y\t{}_{+}\o\>X\t{}_{+}\t{}_{+} Y\t{}_{+}\t{}_{+}, \Gamma(Y\t{}_{-}, X\t{}_{-})Y\t{}_{+}\t{}_{-} X\t{}_{+}\t{}_{-})\\
      =&\Gamma(\<\alpha|\Gamma(X\o, Y\o)X\t Y\t\>X\th{}_{++} Y\th{}_{++}, \Gamma(Y\th{}_{-},X\th{}_{-})Y\th{}_{+-} X\th{}_{+-})\\
      =&\Gamma(\<\alpha|\Gamma(X\o, Y\o)X\t Y\t\>X\th{}_{+} Y\th{}_{+}, \Gamma(Y\th{}_{-}\o,X\th{}_{-}\o)Y\th{}_{-}\t X\th{}_{-}\t)\\
      =&\Gamma(\Gamma(\<\alpha|\Gamma(X\o, Y\o)X\t Y\t\>X\th{}_{+}\o Y\th{}_{+}\o, Y\th{}_{-}\o)X\th{}_{+}\t Y\th{}_{+}\t Y\th{}_{-}\t,  X\th{}_{-})\\
       =&\Gamma(\Gamma(\<\alpha|\Gamma(X\o, Y\o)X\t Y\t\>X\th{}_{+}\o Y\th{}_{+}, Y\th{}_{-})X\th{}_{+}\t,  X\th{}_{-})\\
       =&\Gamma(\Gamma(\<\alpha|\Gamma(X\o, Y\o)X\t Y\t\>X\th Y\th{}_{+}, Y\th{}_{-})X\fo{}_{+},  X\fo{}_{-}).
  \end{align*}
  On the other hand,
  \begin{align*}
      [\alpha\ro&|X\CG\overline{[\alpha\rt|Y]}]\\
      =&[\alpha\ro|X\CG\overline{\Gamma(\<\alpha\rt|Y\o\> Y\t{}_{+}, Y\t{}_{-})}]\\
      =&[\alpha\ro|X_{+}\overline{\Gamma(\Gamma(\<\alpha\rt|Y\o\> Y\t{}_{+}, Y\t{}_{-}), X_{-})}]\\
      =&\Gamma(\<\alpha\ro|X{}_{+}\o\>X{}_{+}\t{}_{+},\Gamma(\Gamma(\<\alpha\rt|Y\o\>Y\t{}_{+},Y\t{}_{-}),X{}_{-})X{}_{+}\t{}_{-})\\
       =&\Gamma(\<\alpha\ro|X\o\>X\t{}_{++},\Gamma(\Gamma(\<\alpha\rt|Y\o\>Y\t{}_{+},Y\t{}_{-}),X\t{}_{-})X\t{}_{+-})\\
      =&\Gamma(\<\alpha\ro|X\o\>X\t{}_{+},\Gamma(\Gamma(\<\alpha\rt|Y\o\>Y\t{}_{+},Y\t{}_{-}),X\t{}_{-}\o)X\t{}_{-}\t)\\
      =& \Gamma(\Gamma(\<\alpha\ro|X\o\>X\t,\Gamma(\<\alpha\rt|Y\o\>Y\t{}_{+},Y\t{}_{-}))X\th{}_{+},X\th{}_{-})\\
      =& \Gamma(\Gamma(\<\alpha\ro|X\o\>X\t,\Gamma(\<\alpha\rt|Y\o\>Y\t{}_{+}\o,Y\t{}_{-}\o)Y\t{}_{+}\t\,Y\t{}_{-}\t)X\th{}_{+},X\th{}_{-})\\
      =& \Gamma(\Gamma(\Gamma(\<\alpha\ro|X\o\>X\t, \<\alpha\rt|Y\o\>Y\t)X\th Y\th{}_{+},Y\th{}_{-})X\fo{}_{+},X\fo{}_{-})\\
      =& \Gamma(\Gamma(\<F^{\alpha}|\<\alpha\ro|X\o\>X\t\overline{\<F_{\alpha}|\<\alpha\rt|Y\o\>Y\t\>}\>X\th Y\th{}_{+},Y\th{}_{-})X\fo{}_{+},X\fo{}_{-})\\
      =& \Gamma(\Gamma(\<F^{\alpha}\alpha\ro|X\o\overline{\<F_{\alpha}\alpha\rt|Y\o\>}\>X\t Y\t{}_{+},Y\t{}_{-})X\th{}_{+},X\th{}_{-})\\
      =& \Gamma(\Gamma(\<\alpha\o F^{\alpha}|X\o\overline{\<\alpha\t F_{\alpha}|Y\o\>}\>X\t Y\t{}_{+},Y\t{}_{-})X\th{}_{+},X\th{}_{-})\\
      =& \Gamma(\Gamma(\<\alpha\o|\<F^{\alpha}|X\o\>X\t\overline{\<\alpha\t|\<F_{\alpha}|Y\o\>Y\t\>}\>X\th Y\th{}_{+},Y\th{}_{-})X\fo{}_{+},X\fo{}_{-})\\
       =& \Gamma(\Gamma(\<\alpha\o|\<F^{\alpha}|X\o\>X\t\overline{\<\alpha\t\,\<F_{\alpha}|Y\o\>|Y\t\>}\>X\th Y\th{}_{+},Y\th{}_{-})X\fo{}_{+},X\fo{}_{-})\\
        =& \Gamma(\Gamma(\<\alpha\o\,\overline{\<F_{\alpha}|Y\o\>}|\<F^{\alpha}|X\o\>X\t\overline{\<\alpha\t|Y\t\>}\>X\th Y\th{}_{+},Y\th{}_{-})X\fo{}_{+},X\fo{}_{-})\\
         =& \Gamma(\Gamma(\<\alpha\o|\<F^{\alpha}|X\o\>X\t\,\<F_{\alpha}|Y\o\>\,\overline{\<\alpha\t|Y\t\>}\>X\th Y\th{}_{+},Y\th{}_{-})X\fo{}_{+},X\fo{}_{-})\\
      =& \Gamma(\Gamma(\<\alpha\o|\<F^{\alpha}|X\o\overline{\<F_{\alpha}|Y\o\>}\>X\t\overline{\<\alpha\t|Y\t\>}\>X\th Y\th{}_{+},Y\th{}_{-})X\fo{}_{+},X\fo{}_{-})\\
      =& \Gamma(\Gamma(\<\alpha|
      \Gamma(X\o, Y\o)X\t Y\t\>X\th Y\th{}_{+},Y\th{}_{-})X\fo{}_{+},X\fo{}_{-}).
  \end{align*}
 Finally, we have $[\alpha|1]=\varepsilon(\alpha)$ and $ [1|X]=\<F^{\alpha}|X_{+}\overline{\<F_{\alpha}|X_{-}\>}\>=\Gamma_{F}(X_{+},X_{-})=\varepsilon^{\Gamma_{F}}(X).$
\end{proof}

\section{Quantum jet Hopf algebroids} \label{secqjet}

In this section, we construct the jet Hopf algebroid  $\CJ(B)$ (denoted by $\CJ^\infty(B)$ in the introduction) of a commutative algebra $B$ by working in the pro-object category. This includes the case of $B$ finite-dimensional and at least in this case the cotwisting theory applies to give new jet Hopf algebroid $\CJ(B)^\Gamma$ over a new base algebra $B_\Gamma$ for certain cocycle data.

\subsection{Pro-objects in  a category\cite{AGV}}
In this section, we recall some basic facts needed about the pro-object category. 
\begin{definition}\label{def. cofiltered category}
 A cofiltered category $I$ is a category such that
\begin{itemize}
    \item [(1)] $I$ is non empty;
    \item [(2)] For every objects $i,i'$ in $I$, there is an object $k$ and two morphism $f:k\to i$, $f':k\to i'$ in $I$;
    \item[(3)] For every two morphisms $g,g':j\to i$, there is an object $k$ and a morphism $f:k\to j$, such that $g\circ f=g'\circ f$.
\end{itemize}  
\end{definition}

For example, natural numbers $\mathbb{N}$ and the even natural numbers $2\mathbb{N}$ (or $n\N$ for all $n\in \N$) are cofiltered categories.

\begin{definition}
 Let $\CC$ be a category.  A pro-object in $\CC$ is a cofiltered category $I$ together with a functor $F:I\to \CC$. Let $F:I\to \mathcal{C}$, $G:K\to \mathcal{C}$ be two such pro-objects. A morphism between them  is an element of 
\[\Hom_{\textup{Pro}(\mathcal{C})}(F, G):=\varprojlim_{k}\, \varinjlim_{i}\Hom_{\mathcal{C}}(F(i),G(k)).\]
More precisely, a morphism $\phi\in \varprojlim_{k}\, \varinjlim_{i}\Hom_{\mathcal{C}}(F(i),G(k))$ is determined by the following data:  For every $k\in K$, there is an object $\varrho(k)\in I$ and a morphism $\phi_{k}:F(\varrho(k))\to G(k)$ in $\mathcal{C}$, such that for every morphism $\gamma:k'\to k$ in $K$, we have the commuting diagram:
\[\xymatrix{F(i)\ar[r]^{F(\nu')}\ar[d]_{F(\nu)}&F(\varrho(k'))\ar[r]^{\phi_{k'}} &G(k')\ar[d]_{G(\gamma)}\\
F(\varrho(k))\ar[rr]^{\phi_{k}} &&G(k),}\]
where $\nu:i\to \varrho(k)$ and $\nu':i\to \varrho(k')$ are the canonical maps in the cofiltered category $I$ as in (2) of Definition \ref{def. cofiltered category}. We will use $\{(\varrho(k), \phi_{k})\}_{k\in K}$ to denote a morphism between two pro-objects as above. We denote the category of pro-objects in $\CC$ by $\textup{Pro}(\CC)$.
\end{definition}

\begin{example}\label{exp. object as pro-object}
  If $C$ is an object in $\CC$ then it is also a pro-object in $\CC$. Indeed, for any cofiltered category $I$, we have $C:I\to\CC$  the constant funtor $C(i)=C$ for every $i\in I$.  
\end{example}

Two morphisms $\{(\varrho(k), \phi_{k})\}_{k\in K}$, $\{(\varsigma(k), \psi_{k})\}_{k\in K}\in \varprojlim_{k}\, \varinjlim_{i}\Hom_{\mathcal{C}}(F(i),G(k))$ are \textit{equivalent} if  for every $k\in K$, there exists $\gamma(k)\in I$ and a pair of morphisms $\eta:\gamma(k)\to \varrho(k)$  and $\eta':\gamma(k)\to \varsigma(k)$ such that the diagram
\[\xymatrix{&F(\gamma(k))\ar[r]^{F(\eta)}\ar[d]_{F(\eta')}&F(\varrho(k))\ar[d]_{\phi_{k}}\\
&F(\varsigma(k))\ar[r]^{\psi_{k}} &G(k)}\]
commutes. We denote the equivalence relation by $\thicksim$. Let $\Phi:=\{(\varrho(k), \phi_{k})\}_{k\in K}\in \Hom_{\textup{Pro}(\mathcal{C})}(F, G)$ as above and  $\Psi:=\{(\sigma(j), \psi_{j})\}_{j\in J}\in \Hom_{\textup{Pro}(\mathcal{C})}(G, E)$ for another pro-object $E:J\to \CC$. The composition
\[\Psi\circ \Phi:=\{(\varrho(\sigma(j)),\psi_{j}\circ \phi_{\sigma(j)})\}_{j\in J}\]
 is an element in  $\Hom_{\textup{Pro}(\mathcal{C})}(F, E)$. It is not hard to see if $\Phi\thicksim\Phi'\in \Hom_{\textup{Pro}(\mathcal{C})}(F, G)$ and $\Psi\thicksim \Psi'\in \Hom_{\textup{Pro}(\mathcal{C})}(G, E)$, then $\Psi\circ \Phi\thicksim\Psi\circ \Phi'$ and $\Psi'\circ \Phi\thicksim\Psi\circ \Phi$ in $\Hom_{\textup{Pro}(\mathcal{C})}(F, E)$. If $F:I\to \CC$ is a pro-object in $\CC$ then $\id_{F}\in \Hom_{\textup{Pro}(\mathcal{C})}(F, F)$ is defined to be $\id_{F}=\{(\lambda(i)=i, \id_{F(i)}:F(i)\to F(i))\}_{i\in I}$. Moreover, two pro-objects $F$ and $G$ are called \textit{isomorphic} in $\textup{Pro}(\CC)$ if there is $\Phi\in \Hom_{\textup{Pro}(\mathcal{C})}(F, G)$ and $\Psi\in \Hom_{\textup{Pro}(\mathcal{C})}(G, F)$, such that $\Phi\circ \Psi\thicksim \id_{G}$ and $\Psi\circ\Phi\thicksim\id_{F}$. In this case, $\Psi$ is called the \textit{inverse} of $\Phi$, and denoted $\Phi^{-1}$.

If $F:I\to \CC$ is a pro-object as above and $H:K\to I$ is a functor between two cofiltered category $I$ and $K$, then we have a pro-object $F\bullet H:K\to \CC$, where we use  $\bullet$ for the composition of functors to distinguish if from the composition between pro-objects morphisms as above. We see that there is a natural element $H^*\in \Hom_{\textup{Pro}(\mathcal{C})}(F, F\bullet H)$, given by $\{(H(k),\id_{F(H(k))})\}_{k\in K}$. Recall that $H:K\to I$ is called \textit{cofinal} if, for every object $i$ in $I$, there is an object $\lambda(i)$ in $K$ and a morphism $\gamma(i):H(\lambda(i))\to i$ in  $I$. 
\begin{lemma}\label{lem. equivalence of pro-objects}
If $F:I\to \CC$ is a pro-object and $H:K\to I$ is a cofinal functor between the cofiltered categories $I$ and $K$ as above  then the pro-objects $F$ and $F\bullet H$ are isomorphic. More precisely, the inverse of $H^*$ is given by
\[(H^*)^{-1}:=\{(\lambda(i),\phi_{i}=F(\gamma(i)):F(H(\lambda(i)))\to F(i))\}_{i\in I}.\]
\end{lemma}
\begin{proof}
  On the one hand,
    \begin{align*}
       (H^*)^{-1}\circ H^{*} =\{(H(\lambda(i)),F(\gamma(i))\circ \id_{F(H(\lambda(i)))}=F(\gamma(i)))\}_{i\in I},      
    \end{align*}
    which is  equivalent to $\id_{F}$. On the other hand,
    \begin{align*}
     H^{*}\circ(H^*)^{-1}=\{(\lambda(H(k)),\id_{F(H(k))}\circ F(\gamma(H(k)))=F(\gamma(H(k))))\}_{k\in K},      
    \end{align*}
    which is equivalent to $\id_{F\bullet H}$.
\end{proof}

\begin{example}\label{exp. equivalence of pro-objects}
 For every pro-object $F:\N\to\CC$,  the inclusion $2\N\to \N$ is a cofinal functor. We denote the composition of $F$ with the inclusion by $F^2:2\N\to\CC$. As a result, $F^2\cong F$ as pro-objects. Similar for all $n\in \N$ and $F^n$. \end{example}

\subsection{Pro-Hopf algebroids}
We now focus on pro-objects in the category of $B$-bimodules ${}_{B}\CM_{B}$ for any algebra $B$ in order to formulate a definition of pro-Hopf algebroids. Let $F,G:I\to {}_{B}\CM_{B}$ be two pro-objects and define the tensor product pro-object $F\ot_{B}G:I\to {}_{B}\CM_{B}$ by $(F\ot_{B}G)(i)=F(i)\ot_{B}G(i)$. Let $E:I\to {}_{B}\CM_{B}$ be another pro-object. We have $(E\ot_{B}F)\ot_{B}G\cong E\ot_{B}(F\ot_{B}G)$. For convience, we simply write $E\ot_{B}F\ot_{B}G$. Also, let $\Phi=\{(\lambda(i),\phi_{i})\}_{i\in I}\in \Hom_{\textup{Pro}({}_{B}\CM_{B})}(E,F)$. We can define $\Phi\ot_{B}\id_{G}\in \Hom_{\textup{Pro}({}_{B}\CM_{B})}(E\ot_{B}G,F\ot_{B}G)$ by $\{(\lambda(i),\phi_{i}\ot_{B}\id_{G(i)})\}_{i\in I}$. Similarly for $\id_{G}\ot_{B}\Phi$. We also have that $B\ot_{B}F\cong F\cong F\ot_{B} B$ as pro-objects in an obvious  way.
\begin{remark}
Since the diagonal map $I\to I\times I,i\mapsto(i,i)$ is a cofinal funtor,  rather than to define $F\ot_{B}G:I\times I\to {}_{B}\CM_{B}, (i,i')\mapsto F(i)\ot_{B}G(i')$, we can define its composition with the diagonal map as above.  
\end{remark}

\begin{definition}
    A pro-$B$-ring is a pro-object $\A:I\to {}_{B}\CM_{B}$ equipped with a morphism (the product) $m\in \Hom_{\textup{Pro}({}_{B}\CM_{B})}(\A\ot_{B}\A, \A)$ and a morphism (the unit) $\eta\in \Hom_{\textup{Pro}({}_{B}\CM_{B})}(B, \A)$ (where $B$ is a constant pro-object as in Example \ref{exp. object as pro-object} with the same $I$), such that
    \[m\circ (\id\ot_{B}m)\thicksim m\circ (m\ot_{B}\id)\in  \Hom_{\textup{Pro}({}_{B}\CM_{B})}(\A\ot_{B}\A\ot_{B}\A, \A),\]
    and
    \[ m\circ(\id\ot_{B}\eta)\thicksim\id\thicksim m\circ (\eta\ot_{B}\id)\in \Hom_{\textup{Pro}({}_{B}\CM_{B})}(\A, \A).\]
\end{definition}
In the last axiom of the definition, we have identified the pro-objects $\A$ with $\A\ot_{B}B$ and $ B\ot_{B}\A$. 
 Similarly,  one can define a pro-$B^e$-ring by replacing $B$ by $B^e$. It is obvious that a pro-$B^e$-ring $\A$ (with unit $\eta$) is a pro-$B$-ring and a pro-$\BB$-ring such that the $B$-bimodule structure and $\BB$-bimodule structure commute, so there are corresponding source and target maps given by
 \[s=\eta|_{B\ot 1}\in\Hom_{\textup{Pro}({}_{B}\CM_{B})}(B, \A),\qquad t=\eta|_{1\ot \BB}\in\Hom_{\textup{Pro}({}_{\BB}\CM_{\BB})}(\BB, \A).\]
 We also see that every pro-object in the category of left $B^e$-modules is also a pro-object in the category of $B$-bimodules, for which we still denote the balanced tensor product $\diamond_{B}$ between two pro-objects in $B$-bimodule given in this way (as for standard bialgebroids). Let $\A:I\to {}_{B}\CM_{B}$ be a pro-$B^e$-ring, we can define its Takeuchi product as a pro-object
$\A\times_{B}\A:I\to {}_{B}\CM_{B}$ by $(\A\times_{B}\A)(i)=A(i)\times_{B} A(i)$, where $A(i)\times_{B} A(i)$ is the Takeuchi product of $A(i)$. We have that $\A\times_{B}\A$ is a pro-$B^e$-ring with multiplication $m_{\A\times_{B} \A}\in \Hom_{\textup{Pro}({}_{B^{e}}\CM_{B^{e}})}((\A\times_{B} \A)\ot_{B^e}(\A\times_{B} \A), \A\times_{B} \A)$ given by
\[\{(\lambda(i),m^{i}_{1,3}\times_{B}m^{i}_{2,4}:(A(\lambda(i))\times_{B} A(\lambda(i)))\ot_{B^e}(A(\lambda(i))\times_{B} A(\lambda(i)))\to A(\lambda(i))\times_{B} A(\lambda(i)))\}_{i\in I},\]
where $\{(\lambda(i),m^{i}:A(\lambda(i))\ot_{B^e}A(\lambda(i))\to A(i))\}_{i\in I}$ is the multipication of the pro-$B^e$-ring $\A$, and $m_{1,3}^{i}$ (similarly for $m_{2,4}^{i}$) means the multiplication on the 1st and 3rd factors.

We now proceed to the coproduct side.

\begin{definition}
    A pro-B-coring is a pro-object $\mathbb{D}:I\to {}_{B}\CM_{B}$ equipped with a morphism $\Delta\in\Hom_{\textup{Pro}({}_{B}\CM_{B})}(\mathbb{D}, \mathbb{D}\ot_{B}\mathbb{D})$, and $\varepsilon\in \Hom_{\textup{Pro}({}_{B}\CM_{B})}(\mathbb{D}, B)$, such that
    \[(\id\ot_{B}\Delta)\circ \Delta\thicksim (\Delta\ot_{B}\id)\circ\Delta\in \Hom_{\textup{Pro}({}_{B}\CM_{B})}(\mathbb{D}, \mathbb{D}\ot_{B}\mathbb{D}\ot_{B}\mathbb{D}),\]
    and
    \[(\varepsilon\ot_{B}\id)\circ \Delta\thicksim \id\thicksim(\id\ot_{B}\varepsilon)\circ \Delta\in \Hom_{\textup{Pro}({}_{B}\CM_{B})}(\mathbb{D}, \mathbb{D}).\]
\end{definition}

Notice that a pro-$B$-(co)ring is {\em not} a pro-object in the category of $B$-(co)rings. However, every pro-object in the category of $B$-(co)rings is a pro-$B$-(co)ring.

\begin{definition}
    Let $B$ be an algebra, a pro-left bialgebroid over $B$ is a pro-$B^e$-ring $\BL:I\to {}_{B^e}\CM_{B^e}$ with source and target maps $s$ and $t$. Moreover, the induced pro-object in ${}_{B}\CM_{B}$ derived from the left $B^e$-module structure is a pro-coring with coproduct and counit 
    \[\Delta\in \Hom_{\textup{Pro}({}_{B}\CM_{B})}(\BL, \BL\diamond_{B}\BL),\qquad\varepsilon\in \Hom_{\textup{Pro}({}_{B}\CM_{B})}(\BL, B),\]
    such that 
    \begin{itemize}
        \item [(1)]  $\Delta\in\Hom_{\textup{Pro}({}_{B}\CM_{B})}(\BL, \BL\times_{B}\BL)$;
        \item[(2)] The coproduct satisfies
        \[\Delta\circ m_{\BL}\thicksim m_{\BL\times_{B}\BL}\circ(\Delta\ot_{B^e}\Delta)\in \Hom_{\textup{Pro}({}_{B}\CM_{B})}(\BL\ot_{B^e}\BL, \BL\times_{B}\BL),\]
        where $m_{\BL}$ is the product of the pro- $B^e$-ring $\BL$;
        \item[(3)] The counit satisfies
        \[\varepsilon\circ m_{\BL}\thicksim \varepsilon\circ (\id_{\BL}\ot_{B^e}s\circ \varepsilon)\thicksim\varepsilon\circ (\id_{\BL}\ot_{B^e}t\circ \varepsilon)\in \Hom_{\textup{Pro}({}_{B}\CM_{B})}(\BL\ot_{B^e}\BL, B).\]
    \end{itemize}
Moreover, we call $\BL$ a pro-left Hopf algebroid if 
\[(\id\diamond_{B}m_{\BL})\circ(\Delta\ot_{\BB}\id)\in \Hom_{\textup{Pro}({}_{B}\CM_{B})}(\BL\ot_{\BB}\BL, \BL\diamond_{B}\BL)\]
is invertible. We call $\BL$ a pro-anti-left Hopf algebroid if 
\[(m_{\BL_{1,3}}\diamond_{B}\id)\circ(\Delta\ot_{B}\id)\in \Hom_{\textup{Pro}({}_{B}\CM_{B})}(\BL\ot_{B}\BL, \BL\diamond_{B}\BL)\]
is invertible. We call $\BL$ a pro-Hopf algebroid if it is both a pro-left Hopf algebroid and pro-anti-left Hopf algebroid.

\end{definition}
As for pro-$B^e$-rings, a pro-Hopf algebroid is in general not a pro-object in the category of Hopf algebroids, but a pro-object in the category of Hopf algebroids is a pro-Hopf algebroid.



\subsection{Pair Hopf algebroid and jet pro-Hopf algebroid $\CJ(B)$}\label{secpair}

We are now ready to construct a certain pro-Hopf algebroid $\CJ(B)$. The starting point, which applies to any algebra $B$, not necessarily commutative, is the well-known {\em pair Hopf algebroid} $B\ot \BB$, with  the $B^e$-ring structure
\[s(a)=a\ot 1,\quad t(a)=1\ot a,\quad (a\ot a')(b\ot b')=aa'\ot b'b,\]
for all $a,a',b,b'\in B$, and the $B$-coring structure
\[\Delta(a\ot a')=a\ot 1\diamond_{B}1\ot a',\quad \varepsilon(a\ot a')=aa'.\]
It is not hard to see that $B\ot \BB$ is in fact a Hopf algebroid with
\[(a\ot a')_{+}\ot_{\overline{B}}(a\ot a')_{-}=a\ot 1\ot_{\overline{B}} a'\ot 1,\quad (a\ot a')_{[+]}\ot_{B}(a\ot a')_{[-]}=1\ot a'\ot_{B} 1\ot a.\]
 There is a left ideal of $B\ot \BB$, defined by
\[\mu_k:=\{(a\ot b)(\extd_{\textup{uni}}a_{0})\,(\extd_{\textup{uni}}a_{1})\cdots \,(\extd_{\textup{uni}}a_{k})|\forall a,b,a_0, a_1,\cdots a_k\in B\}=(\Omega^1_{\rm uni})^{k+1},\]
where $\extd_{\textup{uni}}a=1\ot a-a\ot 1\in B\ot \BB$ and $\mu=\mu_0=\Omega^1_{\rm uni}=\ker(m_{B}:B
\tens B\to B)$ is the `universal calculus' on $B$. We see that $\mu_n\subseteq\cdots \subseteq\mu_1\subseteq \mu_0$. Following \cite{Ver}, we let
\[ \CJ^{k}(B):=B^{e}/\mu_k\]
 be the $k$-th jet bundle over $B$. We  see that $\CJ^{k}(B)$ has a canonical left $B^e$-module structure.  If $B$ is commutative, $\mu_{k}$ is a 2-sided ideal. In this case, we have a sequence of algebras
\[ B\twoheadleftarrow \CJ^1(B) \twoheadleftarrow \CJ^2(B) \twoheadleftarrow \CJ^3(B) \twoheadleftarrow\cdots.\]
Moreover, we have $\mu_{k}=\mu^{k+1}$. We define a pro-object $\BJ:\N\to {}_{B}\CM_{B}$ by 
\[ \BJ(i):=\hJ^{i}(B):=\CJ^{i-1}(B)=B^e/\mu^{i},\]
 or equivalently $\BJ=\varprojlim_{i}\CJ^{i}(B)$.

\begin{theorem}\label{thm. classical jet Hopf algebroid}
   If $B$ is a commutative algebra then $\BJ$ is a pro-Hopf algebroid. We call this pro-Hopf algebroid the jet Hopf algebroid of $B$ and denote it by  $\CJ(B)$.
\end{theorem}
\begin{proof}
 We first observe that
\begin{align*}
    \Delta(\extd_{\textup{uni}}a)=&1\ot1\diamond_{B}1\ot a-a\ot1\diamond_{B}1\ot 1\\
    =&1\ot1\diamond_{B}1\ot a-1\ot1\diamond_{B}a\ot 1+1\ot a\diamond_{B}1\ot 1-a\ot1\diamond_{B}1\ot 1\\
    =&1\ot1\diamond_{B}\extd_{\textup{uni}}a+\extd_{\textup{uni}}a\diamond_{B}1\ot 1.
\end{align*}
Since  $\Delta$ is an algebra map, we have $\Delta(\mu^{2k})\subseteq \mu^{k}\diamond_{B}B^{e}+B^{e}\diamond_{B}\mu^{k}$ for every $k$. Also, $\varepsilon(\mu^{k})=0$. Therefore, we have a $B$-bimodule map induced by the coproduct of $B^e$ on the quotient spaces
\[\Delta^{k}:\hJ^{2k}(B)\to \hJ^{k}(B)\diamond_{B}\hJ^{k}(B).\]
This defines a coproduct on $\HJ$ as a morphism between pro-objects,
\[\Delta_{\HJ}=\{(\lambda(k)=2k, \Delta^{k}:\hJ^{2k}(B)\to \hJ^{k}(B)\diamond_{B}\hJ^{k}(B))\}_{k\in \N}\in \Hom_{\textup{Pro}({}_{B}\mathcal{M}_{B})}(\HJ,\HJ\diamond_{B}\HJ).\]
We first observe that $(\Delta_{\HJ}\diamond_{B}\id)\circ \Delta_{\HJ}\in \Hom_{\textup{Pro}({}_{B}\mathcal{M}_{B})}(\HJ,\HJ\diamond_{B}\HJ\diamond_{B}\HJ)$ can be given by
\[(\Delta_{\HJ}\diamond_{B}\id)\circ \Delta_{\HJ}=\{(\lambda(\lambda(k))=4k,(\Delta^{k}\diamond_{B}\id)\circ \Delta^{2k}:\hJ^{4k}\to \hJ^{k}\diamond_{B}\hJ^{k}\diamond_{B}\hJ^{2k})\}_{k\in\N},\]
where we identify $\HJ\diamond_{B}\HJ\diamond_{B}\HJ$ and $\HJ\diamond_{B}\HJ\diamond_{B}\HJ^{2}$ by Lemma~ \ref{lem. equivalence of pro-objects} (or Example \ref{exp. equivalence of pro-objects}). We also see that $(\Delta\diamond_{B}\id)\circ\Delta(\mu^{3k})\subseteq B^{e}\diamond_{B}B^{e}\diamond_{B}\mu^{k}+B^{e}\diamond_{B}\mu^{k}\diamond_{B}B^{e}+\mu^{k}\diamond_{B}B^{e}\diamond_{B}B^{e}$, so we can define a map on the corresponding quotient space
\[((\Delta\diamond_{B}\id)\circ\Delta)^{k}:\hJ^{3k}(B)\to \hJ^{k}(B)\diamond_{B}\hJ^{k}(B)\diamond_{B}\hJ^{k}(B).\]
Hence, there is a morphism $((\Delta\diamond_{B}\id)\circ \Delta))_{\HJ}\in  \Hom_{\textup{Pro}({}_{B}\mathcal{M}_{B})}(\HJ,\HJ\diamond_{B}\HJ\diamond_{B}\HJ)$ given by
\[\{(\mu(k)=3k, ((\Delta\diamond_{B}\id)\circ\Delta)^{k}:\hJ^{3k}(B)\to \hJ^{k}(B)\diamond_{B}\hJ^{k}(B)\diamond_{B}\hJ^{k}(B))\}_{k\in \N}.\]
We prove that $((\Delta\diamond_{B}\id)\circ \Delta)_{\HJ}\thicksim (\Delta_{\HJ}\diamond_{B}\id)\circ \Delta_{\HJ}$. Indeed, for all $k\in \N$, we  define $\gamma(k)=5k$ such that the diagram 
\[\xymatrix{\hJ^{\gamma(k)}(B)\ar[r]\ar[d]&\hJ^{\lambda(\lambda(k))}(B)\ar[r]^{\kern-20pt {}^{{}^{\scriptstyle (\Delta^{k}\diamond_{B}\id)\circ\Delta^{2k}}}} &\hJ^{k}(B)\diamond_{B}\hJ^{k}(B)\diamond_{B}\hJ^{2k}(B)\ar[d]_{\cong}\\
\hJ^{\mu(k)}(B)\ar[rr]^{((\Delta\diamond_{B}\id)\circ\Delta)^{k}} &&\hJ^{k}(B)\diamond_{B}\hJ^{k}(B)\diamond_{B}\hJ^{k}(B)}\]
commutes. Similarly, one can define $((\id\diamond_{B}\Delta)\circ \Delta)_{\HJ}$ and $(\id\diamond_{B}\Delta_{\HJ})\circ \Delta_{\HJ}$, and show that $((\id\diamond_{B}\Delta)\circ \Delta)_{\HJ}\thicksim (\id\diamond_{B}\Delta_{\HJ})\circ \Delta_{\HJ}$. We also have that $((\Delta\diamond_{B}\id)\circ \Delta)_{\HJ}\thicksim((\id\diamond_{B}\Delta)\circ \Delta)_{\HJ}$ since   $((\Delta\diamond_{B}\id)\circ \Delta)^{k}=((\id\diamond_{B}\Delta)\circ \Delta)^{k}$. Hence, we have that $(\Delta_{\HJ}\diamond_{B}\id)\circ \Delta_{\HJ}\thicksim (\id\diamond_{B}\Delta_{\HJ})\circ \Delta_{\HJ}$. As $\mu^{k}$ belongs to the kernel of $\varepsilon$,  we can define $\varepsilon^{k}:\hJ^{k}\to B$ on the quotient spaces. Hence, 
\[\varepsilon_{\HJ}=\{(\rho(k)=k, \varepsilon^{k}:\hJ^{k}\to B(k)=B)\}_{k\in \N}\in \Hom_{\textup{Pro}({}_{B}\CM_{B})}(\HJ,B),\]
where $B$ is a constant pro-object as in Example \ref{exp. object as pro-object}.
We have
\begin{align*}
   (\varepsilon_{\HJ}\diamond_{B}\id)\circ \Delta_{\HJ}=&\{(\lambda(k)=2k,(\varepsilon^{k}\diamond_{B}\id)\circ\Delta^{k}:\hJ^{2k}\to \hJ^{k})\}_{k\in\N} \\
   =&\{(\id_{k},\pi:\hJ^{2k}\to\hJ^{k} )\}_{k\in\N}\thicksim\id_{\HJ}.
\end{align*}
As a result, we have that $(\varepsilon_{\BJ}\diamond_{B}\id)\circ \Delta_{\HJ}\thicksim\id_{\HJ}\thicksim(\id\diamond_{B}\varepsilon_{\HJ})\circ \Delta_{\HJ}$. Since $B$ is commutative, the image of $\Delta_{\HJ}$ belongs to the Takeuchi product. $\HJ$ is clearly a pro-$B^e$-ring, as $\hJ^{k}$ are $B^e$-ring for each $k\in\N$. Writing $m^k:\hJ^{k}\ot_{B^e}\hJ^{k}$ for the product, we see that
\[\Delta^{k}\circ m^{2k}=m_{\hJ^{k}\times_{B}\hJ^{k}}\circ (\Delta^{k}\ot_{B^e}\Delta^{k}):\hJ^{2k}\ot_{B^e}\hJ^{2k}\to \hJ^{k}\diamond_{B}\hJ^{k}.\]
Hence, $\Delta_{\BJ}\circ m_{\BJ}\thicksim m_{\BJ\times_{B}\BJ}\circ(\Delta_{\BJ}\ot_{B^e}\Delta_{\BJ})$. Similarly, it is not hard to show that $\varepsilon_{\BJ}\circ m_{\BJ}\thicksim \varepsilon_{\BJ}\circ (\id_{\BJ}\ot_{B^e}s\circ \varepsilon_{\BJ})\thicksim\varepsilon_{\BJ}\circ (\id_{\BJ}\ot_{B^e}t\circ \varepsilon_{\BJ})\in \Hom_{\textup{Pro}({}_{B}\CM_{B})}(\BJ\ot_{B^e}\BJ, B)$. Thus,  we have shown that $\BJ$ is a pro-left bialgebroid.

 We now show that $\BJ$ is a pro-Hopf algebroid, by similar methods to those above. We observe that 
    \begin{align*}
        (\extd_{\textup{uni}}a)_{+}&\ot_{\overline{B}}(\extd_{\textup{uni}}a)_{-}\\
        =&1\ot1\ot_{\overline{B}}a\ot 1-a\ot1\ot_{\overline{B}}1\ot 1\\
        =&1\ot1\ot_{\overline{B}}a\ot 1-1\ot1\ot_{\overline{B}}1\ot a+1\ot a\ot_{\overline{B}}1\ot 1-a\ot1\ot_{\overline{B}}1\ot 1\\
        =&-1\ot1\ot_{\overline{B}}\extd_{\textup{uni}}a+\extd_{\textup{uni}}a\ot_{\overline{B}}1\ot 1.
    \end{align*}
By using $(\extd_{\textup{uni}}a\, \extd_{\textup{uni}}b)_{+}\ot (\extd_{\textup{uni}}a\, \extd_{\textup{uni}}b)_{-}=(\extd_{\textup{uni}}a)_{+}\,(\extd_{\textup{uni}}b)_{+}\ot (\extd_{\textup{uni}}b)_{-}\,(\extd_{\textup{uni}}a)_{-}$, we have $(\mu^{2k})_{+}\ot_{\BB}(\mu^{2k})_{-}\subseteq \mu^{k}\ot_{\overline{B}}B^{e}+B^{e}\ot_{\overline{B}}\mu^{k}$. Hence, the translation map $\alpha$ (\ref{X+-}) on $B^e$ induces maps $\alpha^{k}:\hJ^{2k}(B)\to \hJ^{k}(B)\ot_{\overline{B}}\hJ^{k}(B)$ on the quotient spaces.
Therefore, we can define the translation map $\alpha_{\BJ}$ by
\[\alpha_{\BJ}=\{(\lambda(k)=2k,\alpha^{k}:\hJ^{2k}(B)\to \hJ^{k}(B)\ot_{\overline{B}}\hJ^{k}(B))\}_{k\in\N}\in \Hom_{\textup{Pro}({}_{B}\CM_{B})}(\BJ,\BJ\ot_{\BB}\BJ).\]
Here,  $(\id\ot_{\BB}m_{\BJ})\circ(\alpha_{\BJ}\diamond_{B}\id)\in \Hom_{\textup{Pro}({}_{B}\CM_{B})}(\BJ\diamond_{B}\BJ,\BJ\ot_{\BB}\BJ)$ can be given by 
\[\{(\lambda(k)=2k, (\id\ot_{\BB}m^{k})\circ(\alpha^{k}\diamond_{B}\id):\hJ^{2k}(B)\diamond_{B}\hJ^{k}(B)\to \hJ^{k}(B)\ot_{\BB}\hJ^{k}(B))\}_{k\in\N}.\]
Similarly, $(\id\diamond_{B}m_{\BJ})\circ(\Delta_{\BJ}\ot_{\BB}\id)$ can be given by
\[\{(\lambda(k)=2k, (\id\diamond_{B}m^{k})\circ(\Delta^{k}\ot_{\BB}\id):\hJ^{2k}(B)\ot_{\BB}\hJ^{k}(B)\to \hJ^{k}(B)\diamond_{B}\hJ^{k}(B))\}_{k\in\N}.\]
The composition of these two maps is
\begin{align*}
    (\id&\diamond_{B}m_{\BJ})\circ(\Delta_{\BJ}\ot_{\BB}\id)\circ (\id\ot_{\BB}m_{\BJ})\circ(\alpha_{\BJ}\diamond_{B}\id)\\
    =&\{(\lambda(\lambda(k))=2k, (\id\diamond_{B}m^{k})\circ(\Delta^{k}\ot_{\BB}\id))\circ(\id\ot_{\BB}m^{2k})\circ(\alpha^{2k}\diamond_{B}\id):\\
    &\qquad\hJ^{4k}(B)\diamond_{B}\hJ^{k}(B)\to \hJ^{k}(B)\diamond_{B}\hJ^{k}(B)\}_{k\in\N}\\
    =&\{(\varrho(k)=4k,\pi\diamond_{B}\id:\hJ^{4k}(B)\diamond_{B}\hJ^{k}(B)\to \hJ^{k}(B)\diamond_{B}\hJ^{k}(B))\}_{k\in\N}
    \thicksim\id_{\BJ\diamond_{B}\BJ}.
\end{align*}
Similarly, $(\id\ot_{\BB}m_{\BJ})\circ(\alpha_{\BJ}\diamond_{B}\id)\circ(\id\diamond_{B}m_{\BJ})\circ(\Delta_{\BJ}\ot_{\BB}\id)\thicksim\id_{\BJ\ot_{\BB}\BJ}$. Hence $\BJ$ is a pro-left Hopf algebroid. The proof that $\BJ$ is a pro-anti-left Hopf algebroid is similar.
\end{proof}

If $B$ is finite-dimensional then $\mu_k$ and hence $\CJ^k(B)$ eventually stabilise and there is no need for pro-limits, we just take $k$ large enough. More generally, we require for this only that the descending chain $\cdots\subseteq \mu_{k+1}\subseteq \mu_k\subseteq \cdots$ stabilises, which could still of interest for our specific $\mu_k$ even if $B$ is not finite-dimensional. In this case, the pro-Hopf algebroid structures all stabilise and we can identify the pro-Hopf algebroid $\CJ(B)$ with $\CJ^k(B)$ for large enough $k$, which then becomes a usual $B$-Hopf algebroid.

\begin{remark} \label{remJ1}
We also note that for $B$ commutative and each $k$, we have
\[ 0\leftarrow \CJ^{k-1}(B) \leftarrow \CJ^k(B) \leftarrow \Omega_S^k(B)\leftarrow 0\]
as a short exact sequence, where $\Omega_S^k(B)$ is defined \cite{MaSim} as the joint kernel of all adjacent wedge products on the tensor algebra $T_B(\Omega^1(B))$. Here  $\Omega^1(B)= \mu/\mu^2$ is the space of 1-forms in classical (algebraic) geometry, so this is clear for $k=1$.  However, any sub-bimodule of $\mu=\Omega^1_{\rm uni}$ defines a space of 1-forms (or first order differential calculus)   so we are at liberty to introduce a larger differential calculus
\[ \Omega_k^1(B):={\mu/ \mu^{k+1}}\]
so that
\begin{equation}\label{JkOmegak} 0\leftarrow B \leftarrow \CJ^k(B) \leftarrow \Omega^1_k(B)\leftarrow 0\end{equation}
 is a short exact sequence and the jet prolongation map $j_k:B\to\CJ^k(B)$ can be formulated in terms of this. Here $\pi: \CJ^k(B)\to \Omega^1_k(B)$, $\pi(a\tens b)=(\extd a)b$ splits the inclusion of $\Omega_k^1(B)$ giving a projection so that $\CJ^k(B)=B\oplus \Omega^1_k(B)$ where $B=[1\tens B]$ viewed in $\CJ^k(B)$. We then define
 \[ j_k(b)=b\oplus \extd b= [1\tens b+ b\tens 1-1\tens b]=[b\tens 1]\]
  as the jet prolongation. This is a left module map and hence defined by $j_k(1)=[1\tens 1]$ mod $\CN$.  In other words, we can use a non-standard (and noncommutative) differential structure on $B$ to encode the higher order jet bundles and prolongation maps as if 1-jets with respect to this different calculus.

In the case of an algebraic group, we can translate everything to the identity and $\mu\cong B\tens B^+$ where $B^+$ is the kernel of the counit. Then $\mu^{k+1}\cong B\tens (B^+)^{k+1}$ and $\Omega^1_k(B)\cong B \tens B^+/(B^+)^{k+1}$. For example, for $B=\C[x]$, $B\tens B=\C[x,y]$, $\mu\cong \C[x]\tens \<y\>$ and $\Omega^1_k(B)=\C[x]\tens \<y\>/\<y^{k+1}\>$ is a $k$-dimensional calculus where the exterior derivative contains the $k$-fold usual derivatives.
\end{remark}

If $B$ is not commutative, Theorem~\ref{thm. classical jet Hopf algebroid} no longer applies since in general $(\extd_{\textup{uni}}a)(b\ot c)$ need not belong to the kernel of $\varepsilon$, so in general we cannot define a Hopf ideal. Even so, any first order differential calculus $\Omega^1(B)$ is given by $\mu/\CN$ for some sub-bimodule $\CN\subseteq\mu$ and we can {\em define} the associated jet bundle and jet prolongation map as
\[  \CJ^1(B) :=(B\tens B)/\CN,\quad j_1(b)=[b\tens 1].\]
As a left $B$-module, we can think of $\CJ= B^e/I$ for $I$ a left ideal of $B^e$. This is equivalent to $\CJ^1(B)=B\oplus\Omega^1(B)$ in \cite{MaSim,Flo}. Higher $\CJ^k(B)$ at this level as less clear but in view of Remark~\ref{remJ1} we may not actually need higher jet bundles since, since at this level of generality, we could obtain analogues of them by sticking with $\CJ^1(B)$ and taking a different $\CN$.

\subsection{Dual pairing with Hopf algebroid of differential operators of algebras}\label{secpairdif}

As explained in the introduction, results in \cite{Ver} imply that given a smooth manifold $M$, the $B$-module of $k$-th order differential operators is isomorphic to the dual of the sections $\CJ^{k}(B)$ of the $k$-th jet bundle, where $B=C^{\infty}(M)$. We can generalise this idea to any, possibly noncommutative, algebra $B$ as follows. Given any $b\in B$, we define $\delta_{b}:\Hom(B, B)\to \Hom(B, B)$ by
\[\delta_{b}(D)(a)=D(a)b-D(ab),\]
for all $a\in B$ and $D\in \Hom(B, B)$. We define the $k$-th order differential operators of $B$ by
\[\textup{Diff}^{\textup{k}}(B)=\{D\in \Hom(B,B)\,|\,\delta_{b_{0}}\circ\delta_{b_{1}}\cdots \circ\delta_{b_{k}}(D)=0, \forall b_{0},\cdots b_{k}\in B\}.\]
Similarly to the classical case in \cite{Ver}, we have
\begin{lemma}\label{lem. vbski}
    Let $B$ be an algebra (not necessary commutative). Then $\textup{Diff}^{\textup{k}}(B)\cong \Hom_{\BB-}(\CJ^{k}(B),B)$.
\end{lemma}
\begin{proof}
If $D\in \textup{Diff}^{\textup{k}}(B)$, we can define $\phi_D\in\Hom_{\BB-}(\CJ^{k}(B),B)$ by
\[\phi_{D}([a\ot b])=D(a)b.\]
Clearly, $\phi_{D}$ is left $\BB$-linear. We show that it factors through $\mu_k$ by the definition of a $k$-th order differential operator. Indeed, we  show the following inductively
\begin{align*}
    \phi_{D}([(b\ot b')(\extd_{\textup{uni}}b_{k})\,(\extd_{\textup{uni}}b_{k-1})\cdots \,(\extd_{\textup{uni}}b_{0})])=\delta_{b_{k}}\circ\delta_{b_{k-1}}\cdots \circ\delta_{b_{0}}(D)(b)\,b'=0.
\end{align*}
For $k=0$,
\begin{align*}
   \phi_{D}([(b\ot b')(\extd_{\textup{uni}}b_{0})])=\phi_{D}([(b\ot b')(1\ot b_{0}-b_{0}\ot 1)])\\
   =D(b)b_{0}\,b'-D(b\,b_{0})b'=\delta_{b_{0}}(D)(b)b'.
\end{align*}
Assuming the claim is true for $k=n$, we have 
\begin{align*}
     \phi_{D}([(b\ot b')(\extd_{\textup{uni}}&b_{n+1})\,(\extd_{\textup{uni}}b_{n})\cdots \,(\extd_{\textup{uni}}b_{0})])\\
     =&\phi_{D}([(b\ot b')(1\ot b_{n+1})(\extd_{\textup{uni}}b_{n})\,(\extd_{\textup{uni}}b_{n-1})\cdots \,(\extd_{\textup{uni}}b_{0})])\\
     &-\phi_{D}([(b\ot b')(b_{n+1}\ot 1)(\extd_{\textup{uni}}b_{n})\,(\extd_{\textup{uni}}b_{n-1})\cdots \,(\extd_{\textup{uni}}b_{0})])\\
     =&\delta_{b_{n}}\circ\delta_{b_{n-1}}\cdots \circ\delta_{b_{0}}(D)(b)b_{n+1}\,b'-\delta_{b_{n}}\circ\delta_{b_{n-1}}\cdots \circ\delta_{b_{0}}(D)(b\,b_{n+1})\,b'\\
     =&\delta_{b_{n+1}}\circ\delta_{b_{n}}\cdots \circ\delta_{b_{0}}(D)(b)\,b'.
\end{align*}

Conversely, let $\phi\in  \Hom_{\BB-}(\CJ^{k}(B),B)$, we  define a  $k$-th order differential operator $D_{\phi}$ by
\[D_{\phi}(a)=\phi([a\ot 1]).\]
By a similar inductive method, we see that
\begin{align*}
    \delta_{b_{0}}\circ\delta_{b_{1}}\cdots \circ\delta_{b_{k}}(D_{\phi})(b)=\phi([(b\ot 1)(\extd_{\textup{uni}}b_{0})\,(\extd_{\textup{uni}}b_{1})\cdots \,(\extd_{\textup{uni}}b_{k})])=0.
\end{align*}
For $k=0$, we have
\begin{align*}
    \delta_{b_{0}}(D_{\phi})(b)=D_{\phi}(b)b_{0}-D_{\phi}(b\,b_{0})=\phi([(b\ot 1)(\extd_{\textup{uni}}b_{0})]).
\end{align*}
Assuming this is true for $k=n$ and $\phi\in \Hom_{\BB-}(\CJ^{n}(B),B)$, we have
\begin{align*}
    \delta_{b_{n+1}}&\circ\delta_{b_{n}}\cdots \circ\delta_{b_{0}}(D_{\phi})(b)\\
    =&\delta_{b_{n}}\circ\delta_{b_{n-1}}\cdots \circ\delta_{b_{0}}(D_{\phi})(b)b_{n+1}-\delta_{b_{n}}\circ\delta_{b_{n-1}}\cdots \circ\delta_{b_{0}}(D_{\phi})(b\,b_{n+1})\\
    =&\phi([(b\ot 1)(\extd_{\textup{uni}}b_{n})\,(\extd_{\textup{uni}}b_{n-1})\cdots \,(\extd_{\textup{uni}}b_{0})])b_{n+1}\\
    &-\phi([(b\,b_{n+1}\ot 1)(\extd_{\textup{uni}}b_{n})\,(\extd_{\textup{uni}}b_{n-1})\cdots \,(\extd_{\textup{uni}}b_{0})])\\
    =&\phi([(b\ot b_{n+1})(\extd_{\textup{uni}}b_{n})\,(\extd_{\textup{uni}}b_{n-1})\cdots \,(\extd_{\textup{uni}}b_{0})])\\
    &-\phi([(b\,b_{n+1}\ot 1)(\extd_{\textup{uni}}b_{n})\,(\extd_{\textup{uni}}b_{n-1})\cdots \,(\extd_{\textup{uni}}b_{0})])\\
    =&\phi([(b\ot 1)(\extd_{\textup{uni}}b_{n+1})\,(\extd_{\textup{uni}}b_{n})\cdots \,(\extd_{\textup{uni}}b_{0})]),
\end{align*}
where the 3rd step uses the fact that $\phi$ is left $\BB$-linear.
Moreover,
\begin{align*}
    D_{\phi_{D}}(a)=\phi_{D}([a\ot 1])=D(a),
\end{align*}
and
\begin{align*}
    \phi_{D_{\phi}}([a\ot b])=D_{\phi}(a)b=\phi([a\ot 1])b=\phi([a\ot b]).
\end{align*}
\end{proof}

Next, similarly to \cite{KK, Xu}, let $B$ be a commutative algebra (which we think it as the smooth functions $C^\infty(M)$ of a smooth manifold $M$ as in \cite{Xu}). We suppose that the algebra of all differential operators $\cD(B)$ is a finitely generated projective $B$-module. In this case, it forms a  Hopf algebroid. The source and target maps are
\[s(a)(b)=ab,\qquad t(a)(b)=ba, \qquad\forall a,b\in B.\]
The product is operator composition. In addition, the coproduct and counit are given by
\[\Delta(D)(a\ot b)=D(ab),\qquad \varepsilon(D)=D(1),\]
where $\CD(B)\times_B \CD(B)\subseteq \CD(B)\diamond_B \CD(B)\subseteq \Hom(B,B)\diamond_B\Hom(B,B)\subseteq \Hom(B\tens B,B)$. The last inclusion is given by $(D\diamond_{B}D')(b\ot b')=D(b)D'(b')$ for any $D\diamond_{B}D'\in \Hom(B,B)\diamond_B\Hom(B,B)$. Moreover, by \cite{schau3}, if a Hopf algebroid $\CL$ is  left and right finitely generated projective $B$-module then its left and right duals $\Hom_{\BB-}(\CL,B)$ and $\Hom_{B-}(\CL,B)$ are also finitely generated projective Hopf algebroids.

\begin{proposition}  
Let $B$ be a commutative algebra and suppose that $\{\mu_k\}$ stabilize. If $\CJ(B)$ is a finitely generated projective $B$-module then $\CD(B)= \Hom_{\BB-}(\CJ(B),B)$ is the finitely generated projective Hopf algebroid with dual pairing given by
 \begin{align*}
        \<D|[a\ot b]\>=D(a)b,
    \end{align*}
    for all $[a\ot b]\in \CJ(B)$ and $D\in \cD(B)$.
 \end{proposition}
\begin{proof} The dual pairing between $\CD(B)$ and $\CJ(B)$ is from Lemma~\ref{lem. vbski}. First, we observe that
    \begin{align*}
         \<c\overline{d}\,D\,e\overline{f}|[a\ot b]\>g=(c\circ\overline{d}\circ D\circ e \circ\overline{f})(a)\,b\,g=c\,D(eaf)\,d\,b\,g=c\,\<D|e\,\overline{g}[a\ot b]\,f\,\overline{d}\>.
    \end{align*}
    Second, for all $[a\ot b],[c\ot d]\in \CJ(B)$ and $D\in \cD(B)$, we have on the one hand,
    \begin{align*}
        \<D|[a\ot b][c\ot d]\>=\<D|[ac\ot db]\>=D(ac)\,db.
    \end{align*}
    On the other hand
    \begin{align*}
        \<D\o|[a\ot b]\,\overline{\<D\t|[c\ot d]\>}\>=\<D\o|[a\ot D\t(c)db]\>=D\o(a)\, D\t(c)\,db=D(ac)\,db.
    \end{align*}
    Also, $\<D|[1\ot 1]\>=\varepsilon(D)$. Third, for all $[a\ot b]\in \CJ(B)$ and $D,D'\in \cD(B)$, we have 
    \begin{align*}
    \<D'\circ D|[a\ot b]\>= D'(D(a))\,b = \<D'|[D(a)\ot b]\>= \<D'|\<D|[a\ot 1]\>[1\ot b]\>.
    \end{align*}
    Also, $\<1|[a\ot b]\>=a\, b=\varepsilon([a\ot b])$.
\end{proof}

\subsection{Cotwist quantization of jet Hopf algebroids}\label{secquant}

We have constructed the sections of the $k$-th jet bundle $\CJ^{k}(B)$ and the $k$-th differential operators for any algebra $B$. However, for the jet Hopf algebroid $\CJ(B)$, we needed $B$ to be commutative (and then so is $\CJ(B)$). Hence, it remains to outline a route to possibly noncommutative examples. We do this by dualizing the construction in \cite{Xu}, where it was shown how $\cD(B)$ could be twisted by an invertible left 2-cocycle $F$, but working there in a deformation setting over power series in a deformation parameter. The data for this amounted to a certain $*$-product \[a\ast\,b=a\cdot_{F}b=\varepsilon(F^{\alpha}\,a)\varepsilon(F_{\alpha}\,b)=(F^{\alpha}\circ\,a)(1)(F_{\alpha}\circ\,b)(1)=F^{\alpha}(a)\,F_{\alpha}(b),\]
for all $a,b\in B$ and with a formal sum over $\alpha$ understood. In our case, we are working over a field, but following the same ideas, we can then dualise to obtain a deformed and possibly noncommutative jet Hopf algebroid as follows.

\begin{proposition}\label{thm:quant} Let $B$ be commutative, $\{\mu_k\}$ stabilise and $\CJ(B)$ be a finitely generated projective $B$-module.  If $F\in \CD(B)\diamond_{B}\CD(B)$ is an invertible left 2-cocycle in $\cD(B)$  with product $*$ denoting the twisted product (\ref{starF}) then there is an invertible 2-cocycle $\Gamma$ on $\CJ(B)$ given by

\[\Gamma([a\ot b], [c\ot d])=(a\ast c)\,db,\]
with inverse
\[\Gamma^{-1}([a\ot b], [c\ot d])=(ac)\ast(d\ast b).\]
Denoting $B$ with product $*$ by $B^\Gamma$,  the left $\BG$-Hopf algebroid structure on $\CJ(B)^{\Gamma}$ is 
\[s(b)=[b\ot 1],\quad t(b)=[1\ot b],\quad[a\ot b]\CG [c\ot d]=[a\ast c\ot d\ast b],\]
and
\[\Delta^{\Gamma}([a\ot b])=[a\ot 1]\diamond_{\BG}[1\ot b],\qquad \varepsilon^{\Gamma}([a\ot b])=a\ast b,\]
and
\[[a\ot b]_{\p}\ot [a\ot b]_{\m}=[a\ot 1]\ot [b\ot 1].\]
Its dual pairing with the twisted differential operators $\CD^F(B)$ is
\[\<D|[a\ot b]\>^{\Gamma}=D(a)\ast b.\]
\end{proposition}
\begin{proof}
    As $F$ is an invertible left 2-cocycle in $\cD(B)$, by Lemma \ref{lem. 2-cocycle dual}, we can construct an invertible left 2-cocycle $\Gamma$ by $\Gamma(X\ot Y)=\<F^{\alpha}|X\overline{\<F_{\alpha}|Y\>}\>$ for all $X, Y\in \CJ(B)$. More precisely,
    \begin{align*}
        \Gamma([a\ot b], [c\ot d])=&\<F^{\alpha}|[a\ot b]\overline{\<F_{\alpha}|[c\ot d]\>}\>=\<F^{\alpha}|[a\ot b]\overline{F_{\alpha}(c)d}\>\\
        =&\<F^{\alpha}|[a\ot F_{\alpha}(c)d\,b]\>=F^{\alpha}(a)F_{\alpha}(c)db=(a\ast c)\,db.
    \end{align*}
By Theorem \ref{thm. invertible 2 cocycle}, the inverse of $\Gamma$ is given by
\begin{align*}
    \Gamma^{-1}&([a\ot b], [c\ot d])\\
    =&\Gamma([a\ot b]_{+}\,[c\ot d]_{+},\Gamma([c\ot d]_{-}\o,\,[a\ot b]_{-}\o)[c\ot d]_{-}\t\,[a\ot b]_{-}\t)\\
    =&\Gamma([ac\ot 1],\Gamma([d\ot 1]\,,[b\ot 1])[1\ot 1])\\
    =&\Gamma([ac\ot 1],[d\ast\,b\ot 1])=(ac)\ast (d\ast b).
\end{align*}
For the left Hopf algebroid structure, we see firstly that
\[a\CG b=\Gamma([a\ot 1], [b\ot 1])=a\ast b.\]
We also see that
\begin{align*}
    [a\ot& b]\CG[c\ot d]\\
    =&\Gamma([a\ot b]\o,\,[c\ot d]\o)[a\ot b]\t{}_{+}\,[c\ot d]\t{}_{+}\,\overline{\Gamma([c\ot d]\t{}_{-},\,[a\ot b]\t{}_{-})}\\
     =&\Gamma([a\ot 1],\,[c\ot 1])[1\ot 1]\,[1\ot 1]\,\overline{\Gamma([d\ot 1],\,[b\ot 1])}=[a\ast c\ot d\ast b].
\end{align*}
To see the twisted coproduct is the one given above, it is sufficient to check that
\begin{align*}
    \GH([a\ot 1]\diamond_{\BG}[1\ot b])& =[a\ot 1]_{+}\overline{\Gamma([a\ot 1]_{-}, [1\ot b]\o)}\diamond_{B}[1\ot b]\t\\
    =&[a\ot 1]\overline{\Gamma([1\ot 1], [1\ot 1])}\diamond_{B}[1\ot b]=\Delta([a\ot b]).
\end{align*}
For the counit, we have
\begin{align*}
    \varepsilon^{\Gamma}([a\ot b])=\Gamma([a\ot b]_{+}, [a\ot b]_{-})=\Gamma([a\ot 1], [b\ot 1])=a\ast\,b.
\end{align*}
To see that the cotwisted left Hopf structure is the one given above, it is sufficient to check that
\begin{align*}
    \GH([a\ot 1]\ot [b\ot 1])=&[a\ot 1]_{+}\ot [b\ot 1]_{+}\overline{\Gamma([b\ot 1]_{-}, [a\ot 1]_{-})}\\
    =&[a\ot 1]\ot [b\ot 1]
    =[a\ot b]_{+}\ot [a\ot b]_{-}.
\end{align*}
Finally, using Theorem \ref{thm. twisted dual pairing}, we have
\begin{align*}
    \<D|[a\ot b]\>^{\Gamma}=&\Gamma(\<D|[a\ot b]\o\> [a\ot b]\t{}_{+},[a\ot b]\t{}_{-})\\
    =&\Gamma(\<D|[a\ot 1]\> [1\ot 1],[b\ot 1])\\
    =&\Gamma([D(a)\ot 1], [b\ot 1])=D(a)\ast\,b
\end{align*}
for the pairing. \end{proof}
\begin{remark} Note that from $B^\Gamma$ we still have a pair Hopf algebroid $\BG\ot \overline{\BG}$ but when $B^\Gamma$ is noncommutative, we will not have a 2-sided ideal from powers $\mu^k$ and hence cannot directly construct a noncommutative jet Hopf algebroid as a quotient of it. 
\end{remark}

Of course, there is no guarantee that we can find such an $F$ when $B$ is not finite-dimensional, however, the formulae in Proposition~\ref{thm:quant} involving $\Gamma$ and the construction of $\CJ(B)^\Gamma$ make sense and give a Hopf algebroid for any new product $*$ such that $\Gamma$ is well-defined. We end with a concrete example where $B$ is finite-dimensional as an easy case of the theory. More examples will be developed elsewhere.

\begin{example} Consider the cubic variety $B=k[x]/ (x^3-(a x^2+ bx +c))$ for parameters $a,b,c\in k$. 

(i) for generic $a,b,c$ and $k=\Bbb C$,  $\mu^k=\mu$ for all $k\ge 1$ and have dimension 6, hence $\CJ(B)=\CJ^0(B)=B$ as a trivial Hopf algebroid over $B$. Here $\mu$ is given by 3 relations on a 9-dimensional space and it is easy enough to write down six linearly independent elements depending on generic $a,b,c$. One then computes their products and finds that $\mu^2$ also reduces to a 6-dimensional space provided we can solve a cubic equation with parameters given in terms of $a,b,c$. In this case $\mu^2=\mu$ by dimensions, hence the result. An easy example of this generic case is $a=b=0$ and $c=1$. Then $\C[x]/(x^3-1)$ is the group algebra of $\Z/3\Z$ and by Fourier transform this is isomorphic to the algebra of functions on 3 points $i=0,1,2$. In the latter form, $\mu$ is spanned by $\delta_i\tens\delta_j$ for $i\ne j$, where $\delta_i$ are Kronecker delta-functions, from which it is immediate that $\mu^2=\mu$.

(ii) If $a=b=0$ then $\mu^\infty=\{0\}$ and $\CJ(B)=B^e$ the pair Hopf algebroid. Here $\mu$ dimension 6, $\mu^2$ has dimension 4, $\mu^3$ dimension 2 and $\mu^4$ has dimension 1 (it is spanned by $x^2 \tens x^2$). As a result, $\mu^k=\{0\}$ for $k\ge 5$.

(iii) There are special values of $a,b,c$ where $\CJ(B)$ falls between these two extremes. For example, if $a=1$ and $b=c=0$, the relation is  $x^3=x^2$ and then $\mu$ is again 6-dimensional but $\mu^2$ is 5-dimensional and $\mu^3$ is 4-dimensional, namely spanned by
\[ y^2-x^2,\quad x^2 y - x^2,\quad x y^2 - x^2,\quad  x^2 y^2 - x^2,\]
where $x:=x\tens 1$ and $y:=1\tens x$ in $B\tens B$. One then takes products by elements of $\mu$ to find that these again span a 4-dimensional space, hence $\mu^k=\mu^3$ for $k\ge 3$. Then $\CJ(B)=\CJ^2(B)= B^e/\mu^3$ is a Hopf algebroid by our results. Its structure is that of the pair Hopf algebroid $B^e$ in Section~\ref{secpair} but modulo $\mu^3$. Thus, the product and coproduct are
\[ [a\tens b][c\tens d]= [ac\tens db   ],\quad \Delta [a\tens b]= [a\tens 1]\tens_B [1\tens b]\]
etc., where $[a\tens b]$ denotes $a\tens b$ modulo $\mu^3$. One can take as representatives $1,x,x^2, xy, xy-x^2$ and determine all the structure maps explicitly. 

For cotwisting, following Theorem~\ref{thm:quant}, we need 
\[\Gamma([a\ot b],[c\ot d])=\gamma(a,c)db\]
for some linear map $\gamma:B\ot B\to B$, such that $\Gamma$ is well defined and such that $a*b:=\gamma(a,b)$ defines a unital product on the vector space of $B$ with the same $1$. Then $\Gamma$ will be a cocycle on $\CJ(B)$ in our sense. The resulting Hopf algebroid from Theorem~\ref{thm:quant} has base $B^\Gamma$ given by the new product $*$ and the product of $\CJ(B)^\Gamma$ given by 
\[[a\ot b]\cdot_{\Gamma}[c\ot d]=[a*c,d*b], \]
which again can be computed explicitly once we know $\gamma$. The coproduct is unchanged (but now over $B^\Gamma$) and the counit is $\varepsilon^{\Gamma}([a\ot b])=a*b$.  This is the general construction. In our case, we obtain a 2-parameter family of $\gamma$ or $*$-products. Namely, compatibility with $\mu^3$ so that $\Gamma$ descends leads to
\[ x*x=\alpha x^2+\beta x+\gamma,\quad x*(x^2)=(x^2)*x=(x^2)*(x^2)=x^2\]
and the trivial products with 1. For example, $0=\Gamma([y^2-x^2]\tens[b\tens d])=\gamma(1,b)x^2 d-\gamma(x^2,b)d$ for all $b,d$ tells us $(x^2)*b=x^2$ for all $b$, etc., while $x*x$ is not constrained. Associativity $(x^2)*(x*x)= (x^2)*(\alpha x^2+\beta x+\gamma)=x^2(\alpha+\beta+\gamma)=((x^2)*x)*x=x^2$ then imposes the condition 
\[ \alpha+\beta+\gamma=1.\]
In this simple example, the product $*$ remained commutative, but this will not generally be the case.
\end{example}

\subsection*{Acknowledgements}  We thank Peter Schauenburg and Ivan Tomasic for helpful discussions and, particularly,  Behrang Noohi for advice on the pro-object category. XH and SM were supported by a Leverhulme Trust project grant RPG-2024-177.

\end{document}